\documentclass[11pt,oneside,english]{amsart}
\usepackage[T1]{fontenc}
\usepackage[latin9]{inputenc}
\usepackage{geometry}
\usepackage{float}
\usepackage{mathrsfs}
\usepackage{amsthm}
\usepackage{amsbsy}
\usepackage{amstext}
\usepackage{amssymb}
\usepackage{fixltx2e}
\usepackage{graphicx}
\usepackage[all]{xy}

\makeatletter

\providecommand{\tabularnewline}{\\}

\numberwithin{equation}{section}
\numberwithin{figure}{section}
\numberwithin{table}{section}
\theoremstyle{plain}
\newtheorem{thm}{\protect\theoremname}[section]
  \theoremstyle{plain}
  \newtheorem{cor}[thm]{\protect\corollaryname}
  \theoremstyle{definition}
  \newtheorem{defn}[thm]{\protect\definitionname}
  \theoremstyle{remark}
  \newtheorem{rem}[thm]{\protect\remarkname}
  \theoremstyle{plain}
  \newtheorem{prop}[thm]{\protect\propositionname}
  \theoremstyle{definition}
  \newtheorem{example}[thm]{\protect\examplename}
  \theoremstyle{plain}
  \newtheorem{fact}[thm]{\protect\factname}
  \theoremstyle{plain}
  \newtheorem{lem}[thm]{\protect\lemmaname}
  \newtheorem{question}[thm]{Question}
  \newtheorem{notation}[thm]{Notations}

  \theoremstyle{plain}
  \newtheorem*{cor*}{\protect\corollaryname}

\@ifundefined{date}{}{\date{}}
\usepackage[all]{xy}
\newcommand{\xyR}[1]{
  \xydef@\xymatrixrowsep@{#1}}
\newcommand{\xyC}[1]{
  \xydef@\xymatrixcolsep@{#1}}

\newcommand{\bb}[1]{\mathbb{#1}}

\usepackage{ifpdf}
\ifpdf

\IfFileExists{lmodern.sty}
 {\usepackage{lmodern}}{}

\fi 

\usepackage{tikz}
\usetikzlibrary{arrows}

\makeatother

\usepackage{babel}
  \providecommand{\corollaryname}{Corollary}
  \providecommand{\definitionname}{Definition}
  \providecommand{\examplename}{Example}
  \providecommand{\factname}{Fact}
  \providecommand{\lemmaname}{Lemma}
  \providecommand{\propositionname}{Proposition}
  \providecommand{\remarkname}{Remark}
\providecommand{\theoremname}{Theorem}

\begin{document}

\title{HEEGAARD FLOER HOMOLOGY OF SURGERIES ON TWO-BRIDGE LINKS}

\author{YAJING LIU}
\begin{abstract}
We give an $O(p^{2})$ time algorithm to compute the generalized Heegaard
Floer complexes $A_{s_{1},s_{2}}^{-}(\overrightarrow{L})$'s for a two-bridge
link $\overrightarrow{L}=b(p,q)$ by using nice diagrams. Using the
link surgery formula of Manolescu-Ozsv\'{a}th, we also show that
${\bf HF}^{-}$ and their $d$-invariants of all integer surgeries
on two-bridge links are determined by $A_{s_{1},s_{2}}^{-}(\overrightarrow{L})$'s.
We obtain a polynomial time algorithm to compute ${\bf HF}^{-}$ of
all the surgeries on two-bridge links, with $\mathbb{Z}/2\mathbb{Z}$
coefficients. In addition, we calculate some examples explicitly:
 ${\bf HF}^{-}$ and the $d$-invariants of all integer surgeries
on a family of hyperbolic two-bridge links including the Whitehead link.
\end{abstract}
\maketitle

\section{Introduction}

\subsection{Background and motivation.}

Heegaard Floer homology is a package of invariants of 3-manifolds
invented by Ozsv\'{a}th and Szab\'{o}, using holomorphic disks and
Heegaard splittings of the 3-manifold \cite{OS_HF1,OS_HF2}. It detects
the Thurston norm and fiberedness of a 3-manifold \cite{[Ghiggini]fibred,[OS]genus,[YiNi]Fiber}.
Furthermore, it fits into a kind of 3+1 dimensional topological quantum
field theory, which is important in the study of smooth structures
on 4-manifolds. Unlike other Floer homological invariants, Heegaard
Floer homology is combinatorially computable, and there are several
algorithms for computing various versions of it. Manolescu, Ozsv\'{a}th
and Sarkar described knot Floer homology combinatorially using grid
diagrams in \cite{Combinatorial_Knot_Floer}. Sarkar and Wang in \cite{Sakar-Wang}
found an algorithm for computing $\widehat{HF}(M^{3})$ over $\mathbb{Z}/2\mathbb{Z}$
by using nice Heegaard diagrams. Lipshitz, Ozsv\'{a}th and Thurston
used bordered Floer homology to give another algorithm for computing
$\widehat{HF}(M^{3})$ in \cite{Bordered_Floer}. In \cite{MOT}, Manolescu, Ozsv\'{a}th and
Thurston showed that the plus and minus versions of Heegaard Floer
homology (over $\mathbb{Z}/2\mathbb{Z}[[U]]$) can also be described
combinatorially, by using link surgery and grid diagrams. (Admittedly,
the MOT algorithm has a high time complexity.) Improving these algorithms
and developing new methods for computations are still important and
interesting questions.

\subsection{The basic idea. }

This paper is aimed at studying the Heegaard Floer homology of
 surgeries on two-component links by using the link surgery formula due
to Manolescu-Ozsv\'{a}th \cite{link_surgery}. When $\overrightarrow{L}$
is a two-bridge link $b(p,q)$, we find a fast algorithm for computing
the Floer homology of surgeries on $L$, $\mathbf{HF}^{-}(S_{\Lambda}^{3}(\overrightarrow{L}))$
over $\mathbb{F}=\mathbb{Z}/2\mathbb{Z}$, where $\Lambda$ is the
framing matrix of a surgery. (Here, $\mathbf{HF}^{-}(S_{\Lambda}^{3}(\overrightarrow{L}))$ is the $U$-completion of $HF^{-}$. See \cite{link_surgery} Section 2.) This algorithm uses genus-0 nice diagrams
and algebraic arguments to simplify the Manolescu-Ozsv\'{a}th surgery
formula. Its worst-case time complexity is a polynomial of $p$ and $\det(\Lambda)$.

Let us mention some related work. In \cite{Rasmussen_2_bri}, Rasmussen
studied Heegaard Floer homology of surgeries on two-bridge knots.
In \cite{knot_surgery}, Ozsv\'{a}th and Szab\'{o} developed a formula
for the Heegaard Floer homology of surgeries on knots. The paper
\cite{link_surgery} presents a generalization of this formula to
the case of links. Two sets of data are needed in the surgery formula
in \cite{link_surgery}: the generalized Floer complexes $A_{s}^{-}(\overrightarrow{L})$'s
and the maps in the surgery formula, namely the maps $I_{s}^{\overrightarrow{L}'},D_{s}^{\overrightarrow{L}'}$
connecting the complexes associated to oriented sublinks. In general,
the Heegaard Floer homology of link surgeries is more difficult to
compute, due to more involved algebraic structures. However, in some
cases, computations using this surgery formula can be simplified.

The main complexity in the link surgery formula is the counting of
the holomorphic domains on the Heegaard surface, which corresponds
to holomorphic bigons and polygons in the symmetric product. For the
special case of two-bridge links, we directly find a formula for the
counts of holomorphic bigons. Furthermore, the general link surgery
formula involves counting holomorphic polygons in the symmetric product
for computing some cobordism maps, and this is of considerably high
time complexity. Here we notice that, for two-bridge links, all these
maps can be determined algebraically.

Note that this paper provides new examples of hyperbolic 3-manifolds for which
we can compute their Heegaard Floer homology.

\subsection{Main results and organization.}

In Section 2, we review some preliminaries for the link surgery formula,
including the generalized Floer complexes, polygon maps and nice diagrams.

In Section 3, using the Schubert normal form of two-bridge links we
get a class of nice Heegaard diagrams called \emph{Schubert Heegaard diagrams},
in which every region is either a bigon or a square. We can explicitly
describe all the composite bigons on a Schubert Heegaard diagram,
and hence the Floer differentials. Further, we get a formula for the
Alexander gradings of all intersection points, thus giving a formula
for the multi-variable Alexander polynomial of a two-bridge link $b(p,q)$
in terms of $p,q$. See Theorem \ref{thm:bigon} and Proposition \ref{prop:Alexander grading}
below for the precise statements. This implies that $A_{\mathrm{s}}^{-}(\overrightarrow{L})$
can be directly computed from this diagram. For a two-bridge link
$\overrightarrow{L}=b(p,q)$, we get an $O(p^{2})$ time algorithm
for computing $A_{\mathrm{s}}^{-}(\overrightarrow{L}).$ We also found
different two-bridge links (modulo mirror and reorientation) sharing
the same multi-variable Alexander polynomial, signature, and linking
number.

In Section 4, we review the link surgery formula from \cite{link_surgery}
for two-component links $\overrightarrow{L}=\overrightarrow{L_{1}}\cup\overrightarrow{L_{2}}$
with basic diagrams. First, we review some algebraic tools, \emph{hyperboxes
of chain complexes}. In order to see the algebraic structure of the
link surgery formula, we define a \emph{twisted gluing} of squares
of chain complexes. Then the link surgery formula is a twisted gluing
of certain squares of chain complexes derived from $L$. These squares
are constructed in \cite{link_surgery} by means of \emph{complete
system of hyperboxes}, which is a set of compatible Heegaard diagrams
for the sublinks. For any two-component link, we define a type of
complete system of hyperboxes which generalizes the basic systems
used in \cite{link_surgery}, called a \emph{primitive system of hyperboxes}.
We also show that any basic diagram of $\overrightarrow{L}=\overrightarrow{L_{1}}\cup\overrightarrow{L_{2}}$
produces a primitive system.

In Section 5, we use algebraic arguments to show some rigidity results
of the destabilization maps $D_{\text{s}}^{\overrightarrow{M}}\text{'s},M\subset L$
up to chain homotopy, for two-bridge links. Further, if we perturb
the destabilization maps $D_{\text{s}}^{\pm L_{i}}$'s by chain homotopy,
i.e. replace $D_{\text{s}}^{\pm L_{i}}$ by $\tilde{D}_{\text{s}}^{\pm L_{i}}\simeq D_{\text{s}}^{\pm L_{i}}$,
we can construct a new square of chain complexes called the \emph{perturbed
surgery complex.} Using the rigidity results, we show that the perturbed
surgery complex is isomorphic to the original complex in the link
surgery formula. Based on the perturbed surgery complex, we give the
algorithm for computing $\mathbf{HF}^{-}(S_{\Lambda}^{3}(\overrightarrow{L}))$
mentioned before.

Throughout this paper, we use $\mathbb{F}=\mathbb{Z}/2\mathbb{Z}$
coefficients. The main result we obtain is the following:
\begin{thm}
\label{thm:perturbed_surgery_formula}
Suppose $\overrightarrow{L}$ is an oriented two-bridge link with
framing $\Lambda$. Let $\mathcal{H}^{L}$ be a basic Heegaard diagram
of $\overrightarrow{L}$ and let $\mathcal{H}$ be a primitive system
induced by $\mathcal{H}^{L}$. After we determine the $\mathbb{F}[[U_1,U_2]]$-modules
 $A_{\mathrm{s}}^{-}(\overrightarrow{L})$'s sitting at the
vertices of the square in the link surgery formula, any choices of
\begin{itemize}
\item $\mathbb{F}[[U_1]]$-linear chain homotopy equivalences $\tilde{D}_{s_{1},s_{2}}^{-L_{i}}$ for the edge maps,
\item $\mathbb{F}[[U_1]]$-linear chain homotopies for the diagonal maps
\end{itemize} yield a perturbed surgery complex $(\mathcal{\tilde{C}}^{-}(\mathcal{H}^{L},\Lambda),\tilde{\mathcal{D}}^{-})$
which is isomorphic to the original surgery complex in \cite{link_surgery}
as an $\mathbb{F}[[U_{1}]]$-module. By imposing the $U_{2}$-action
to be the same as the $U_{1}$-action, the $\mathbb{F}[[U_{1},U_{2}]]$-module
$H_{*}(\mathcal{\tilde{C}}^{-}(\mathcal{H},\Lambda),\tilde{\mathcal{D}}^{-})$
becomes isomorphic to the homology $\mathbf{HF}^{-}(S_{\Lambda}^{3}(\overrightarrow{L}))$.
This isomorphism preserves the absolute grading.
\end{thm}

See Theorem \ref{thm:invariance of perturbed surg. formula} below for a more
precise statement and the proof.

\begin{cor}
\label{cor:A_s->HF}
For a two-bridge link $\overrightarrow{L}$, knowledge
of the $A_{\mathrm{s}}^{-}(\overrightarrow{L})$ determines ${\bf HF}^{-}$
of all the surgeries on $L$.
\end{cor}

In Section 6, we compute some examples explicitly: the surgeries on
$b(4n,2n+1)$, $n\in\mathbb{N}$, which are two sequences of hyperbolic two-bridge links generalizing
the Whitehead link and the torus link $T(2,4)$. See Proposition \ref{prop:surgery on Wh}
and Theorem \ref{thm:surgery_on_b(8k,4k+1)}. Actually in the course of the computation, we also show that
 the Whitehead link is an L-space link, which means all of its large surgeries are L-spaces, i.e.
 $A^-_{s_1,s_2}(\mathit{Wh})$'s all have homology $\mathbb{F}[[U]]$. This provides examples
of hyperbolic L-spaces.

To compute these examples, we study the filtered homotopy type of $CFL^-(L)$.
In Section 6.2, we prove that when $L=b(4n,2n+1)$,
the filtered chain homotopy type of $CFL^{-}(L)$ is determined by
the filtered chain homotopy type of $\widehat{CFL}(L)$. See Proposition
\ref{prop:CFL^-(b(8k,4k+1))} and Proposition \ref{prop:CFL^-b(8k,4k+3)}
for the precise statements. Basically, this is based
on an observation that the Alexander polytope is simple and there
are several symmetries on $CFL^{-}(L)$, which give constraints for
the differentials in $CFL^{-}(L)$. From $CFL^{-}(L)$ we derive all
the $A_{\mathrm{s}}^{-}(L)$'s and the inclusion maps. Finally, using
the perturbed surgery complex,  we compute the Floer homology of their
surgeries and the associated $d$-invariants in Section 6.3.

Since $CFL^{-}(L)$ is the same as $A_{+\infty,+\infty}^{-}(L)$ viewed
as a $\mathbb{Z}\oplus \mathbb{Z}$-filtered chain complex (with the Alexander
filtration), Corollary \ref{cor:A_s->HF} means the filtered homotopy
type of $CFL^{-}(L)$ contains all the information about the Floer homology
of the surgeries on $L$, when $L$ is a two-bridge link. In \cite{OS_link_FLoer},
it is shown that, for an alternating two-component link $L$, the
filtered chain homotopy type of $\widehat{CFL}(L)$ is determined
by the set of data:
\begin{itemize}
\item the multi-variable Alexander polynomial $\Delta_{L}(x,y)$,
\item the signature $\sigma(L)$,
\item the linking number $\mathrm{lk}(L)$,
\item  the filtered homotopy type of $\widehat{CFK}(L_{i})$ of each component.
\end{itemize}
However, it is hard to determine the filtered homotopy type of $CFL^{-}(L)$
in general. For two-bridge links, the Schubert Heegaard diagrams show that the
$U_{1},U_{2}$-differentials in $CFL^{-}(L)$ are quite simple, since the bigons always
contain exactly one basepoint. In addition, every component of a two-bridge link
 is the unknot. Thus, for two-bridge links, we raise the following question:

\begin{question}
Given an oriented two-bridge link $L$, is the filtered homotopy type
of $CFL^{-}(L)$ determined by the set of data $\{\Delta_{L}(x,y),\sigma(L),\mathrm{lk}(L)\}$?
\end{question}

We note that ${\bf HF}^{-}$ of surgeries on two-bridge links may
also be computed by other methods. For example, as long as one of
the framing coefficients is not $0$, one can view one component as
a knot in a lens space and compute using the grid diagram methods
in \cite{[Baker-Hedden-Grigsby]}. Another method is to consider these surgeries
as surgeries on (1,1)-knots in lens spaces and use the method of \cite{(1-1)_knot}.
Nevertheless, the method in this paper is more conceptual.
Some of the arguments here could be potentially used for other classes
of links. In fact, Theorem \ref{thm:perturbed_surgery_formula} and Corollary \ref{cor:A_s->HF}
can be directly generalized to the two-component links with every component being an L-space knot.

\subsection{Acknowledgments.}

This project was written under the supervision of my adviser Ciprian
Manolescu. I am greatly indebted to him for his continuous guidance
and support, and for his numerous valuable suggestions and encouragements.

\section{Heegaard diagrams and generalized Floer complexes}

In this section, we give the precise definitions of what we need in
the link surgery formula, including Heegaard diagrams, generalized
Floer complexes, polygon maps, and nice diagrams. Here, the Heegaard
diagram is adapted for a link inside a 3-manifold with multiple basepoints;
the generalized Floer complexes of a link $L$ are derived from the
filtered complex $CFL^{-}(L)$, and they govern the large surgeries;
the polygon maps are used in constructing cobordism maps and certain
maps in the link surgery formula; knowledge of nice diagrams are also
introduced to deal with two-bridge links.

\subsection{Heegaard diagrams of links.}

We give the most general definition of Heegaard diagrams for
an oriented link $\overrightarrow{L}$ inside a 3-manifold $M^{3}$.
When the link $\overrightarrow{L}=\emptyset$, the Heegaard diagram
is simply for $M^{3}$.
\begin{defn}[Heegaard diagram of links]
A\emph{ multi-pointed} \emph{Heegaard
diagram} for the oriented link $\overrightarrow{L}$ in $M^{3}$ is
the data of the form $\mathcal{H}=(\Sigma,\boldsymbol{\alpha},\boldsymbol{\beta},\text{{\bf{w}}},\text{{\bf{z}}})$,
where:
\begin{itemize}
\item $\Sigma$ is a closed, oriented surface of genus $g$;
\item $\boldsymbol{\alpha}=\{\alpha_1,...,\alpha_{g+k-1}\}$ is a collection of disjoint, simple closed curves on $\Sigma$ which span a $g$-dimensional lattice of $H_1(\Sigma;\mathbb{Z})$, hence specify a handlebody $U_\alpha$; the same goes for $\boldsymbol{\beta}=\{\beta_1,...,\beta_{g+k-1}\}$ specify a handlebody $U_\beta$.
\item ${\bf{w}}=\{w_1,...,w_k\}$ and ${\bf{z}}=\{z_1,...,z_m\}$ (with $k\geq m$) are collections of points on $\Sigma$ with the following property. Let $\{A_i\}_{i=1}^k$ be the connected components of $\Sigma-\alpha_{1}-\dots-\alpha_{g+k-1}$ and $\{B_{i}\}_{i=1}^{k}$ be the connected components of $\Sigma-\beta_{1}-\dots-\beta_{g+k-1}.$ Then there is a permutation $\sigma$ of $\{1,...,m\}$ such that $w_{i}\in A_{i}\cap B_{i}$ for $i=1,...,k$, and $z_{i}\in A_{i}\cap B_{\sigma(i)}$ for $i=1,...,m$, such that connecting $w_i$ to $z_i$ inside $A_i$ and connecting $z_i$ to $w_{\sigma(i)}$ inside $B_i$ will give rise to the link $\overrightarrow{L}$.
\end{itemize}
\end{defn}

\begin{defn}[Admissible diagrams]
A \emph{periodic domain} is a two-chain $\phi$ on $\Sigma$ which is
a linear combination of components of $\Sigma-\boldsymbol{\alpha}\cup\boldsymbol{\beta}$ with integral coefficients,
such that the local multiplicity of $\phi$ at every $w_{i}\in{\bf w}$
is $0$ and the boundary of $\phi$ is a integral combination of $\alpha$- and $\beta$-curves.
 A multi-pointed Heegaard diagram $\mathcal{H}=(\Sigma,\boldsymbol{\alpha},\boldsymbol{\beta},{\bf w},{\bf z})$
is called \emph{admissible }if every non-trivial periodic domain has
some positive local multiplicities and some negative local multiplicities.
\end{defn}

\begin{defn}[Basic diagrams of links]
An admissible Heegaard diagram of $\overrightarrow{L}$
is called \emph{basic}, if $l=k=m$, meaning there are exactly two
basepoints $w_{i},z_{i}$ for every component $\overrightarrow{L_{i}}$
and no free basepoints.
\end{defn}

\begin{rem}
(1) The definitions of pointed Heegaard moves are systematically formulated
in \cite{link_surgery} section 4.

(2) In order to avoid the issue of naturality, we fix the Heegaard
surface $\Sigma$ as an embedded surface in the underlying 3-manifold
$M^{3}$. Thus, a Heegaard diagram is equivalent to a self-indexed
Morse function.

(3) In this paper we will only consider maximally colored diagrams
in the sense of \cite{link_surgery}.
\end{rem}

\subsection{Generalized Floer complexes.}
Here we define some chain complexes of a Heegaard diagram for an oriented
link in $S^{3}$, which govern the large surgeries on this link. Suppose
$\overrightarrow{L}=\overrightarrow{L_{1}}\cup\overrightarrow{L_{2}}\cup\cdots\cup\overrightarrow{L_{l}}$,
and $\overrightarrow{M}$ is an oriented sublink of $\overrightarrow{L}$,
where $\overrightarrow{M}$ may not have the induced orientation from
$\overrightarrow{L}$ on each component. By $\overrightarrow{L}-M$,
we denote the oriented link obtained by deleting all the components
of $\overrightarrow{M}$ from $\overrightarrow{L}$.

The identity $H_{1}(S^{3}-\overrightarrow{L})\cong\mathbb{Z}^{l}$
provides a way to record the $\text{Spin}^{c}$ structures over $S^{3}$
relative to $L$ as an affine lattice over $\mathbb{Z}^{l}$.

\begin{defn}[$\mathbb{H}(L)$ and reduction maps]
Define the affine
lattice $\mathbb{H}(\overrightarrow{L})$ over $H_{1}(S^{3}-\overrightarrow{L})$
as follows:
\[
\mathbb{H}(\overrightarrow{L})_{i}=\frac{\mathrm{lk}(\overrightarrow{L_{i}},\overrightarrow{L}-\overrightarrow{L_{i}})}{2}+\mathbb{Z}\subset\mathbb{Q},\mathbb{H}(\overrightarrow{L})=\underset{i}{\overset{l}{\bigoplus}}\mathbb{H}(\overrightarrow{L})_{i},
\]
together with its completion
\[
\overline{\mathbb{H}}(\overrightarrow{L})_{i}=\mathbb{H}(\overrightarrow{L})_{i}\cup\{-\infty,+\infty\},\overline{\mathbb{H}}(\overrightarrow{L})=\underset{i}{\overset{l}{\bigoplus}}\overline{\mathbb{H}}(\overrightarrow{L})_{i}.
\]

The map $\psi^{\overrightarrow{M}}:\mathbb{H}(\overrightarrow{L})\rightarrow\mathbb{H}(\overrightarrow{L}-M)$
is defined by $\psi^{\overrightarrow{M}}(\text{s})=\text{s}-[\overrightarrow{M}]/2$.
More precisely, let $M=L_{j_{1}}\cup...\cup L_{j_{m}}$. Then for
all $i$ not in $\{j_{1},...,j_{m}\}$, let $L_{i}=(L-M)_{k_{i}}$,
set
\begin{equation}
\psi_{i}^{\overrightarrow{M}}:\overline{\mathbb{H}}(\overrightarrow{L})_{i}\rightarrow\overline{\mathbb{H}}(\overrightarrow{L}-M)_{k_{i}},s_{i}\rightarrow s_{i}-\frac{\mathrm{lk}(\overrightarrow{L_{i}},\overrightarrow{M})}{2}.
\label{eq:psi}
\end{equation}
The map $\psi^{\overrightarrow{M}}$ is defined to be the direct sum
of the maps $\psi_{i}^{\overrightarrow{M}}$, for those $i$'s with
$L_{i}$ not in $M$.
\end{defn}

 The reduction maps $\psi^{\overrightarrow{M}}$ are used
 in the definition of the destabilization maps in Section 4.3.2 and
 also in the statement of link surgery formula in Section 4.2.

For convenience to define the generalized Floer complexes, here we focus
on Heegaard diagrams with only one pair of
basepoints $w_{i},z_{i}$ on each component and allow free basepoints.
Given an admissible multi-pointed Heegaard diagram
 $\mathcal{H}=(\Sigma,\boldsymbol{\alpha},\boldsymbol{\beta},\text{{\bf{w}}},\text{{\bf{z}}})$
for $\overrightarrow{L}$ with exactly two basepoints ${z_{i}}$ and $w_{i}$ for each link
component $L_i$,  we consider the Lagrangian pair $\mathbb{T}_{\alpha},\mathbb{T}_{\beta}$
in   $\mathrm{Sym}^{g+k-1}(\Sigma)$
and the Floer complex $CF(\mathbb{T}_{\alpha},\mathbb{T}_{\beta})$.
There is an Alexander multi-grading  $A:\mathbb{T}_{\alpha}\cap\mathbb{T}_{\beta}\rightarrow\mathbb{H}(L)$
characterized by the property
\[
A_{i}({\bf{x}})-A_{i}({\bf{y}})=n_{z_{i}}(\phi)-n_{w_{i}}(\phi),\forall\phi\in\pi_{2}(\bf{x},\bf{y})
\]
and a normalization condition on the Alexander polynomial. The Alexander grading induces
a filtration on $CF^-(\mathbb{T}_{\alpha},\mathbb{T}_{\beta})$.
Given a $\text{Spin}^{c}$ structure on $S^{3}-L$, i.e. an element
$\mathrm{s}\in\mathbb{H}(L)$, we associate a chain complex $\mathfrak{A}^{-}(\mathcal{H},\text{s})$
called the  \emph{generalized Heegaard Floer complex} using the Alexander
filtration. We introduce variables $U_{i}$ with $1\leq i\leq l$ for each link component
$L_{i}$, and $U_{i}$ with $l+1\leq i\leq k$ for each free basepoint $w_{i}$.

\begin{defn}[Generalized Floer complex]
For $\text{s}\in\mathbb{H}(L)$, the \emph{generalized Floer complex}
$\mathfrak{A}^{-}(\mathcal{H},\text{s})$ is the free module over
$\mathcal{R}=\mathbb{F}[[U_{1},...,U_{l}]]$ generated by $\mathbb{T}_{\alpha}\cap\mathbb{T}_{\beta}\in \mathrm{Sym}^{g+k-1}(\Sigma)$,
and equipped with the differential:
\begin{equation}
\label{eq:A_s}
\partial_{\mathrm{s}}^{-}{\bf x}=\underset{{\bf y}\in\mathbb{T}(\alpha)\cap\mathbb{T}(\beta)}{\displaystyle \sum}{\displaystyle \underset{\begin{array}{c}
\phi\in\pi_{2}({\bf x},{\bf y})\\
\mu(\phi)=1
\end{array}}{\displaystyle \sum}\#(\mathcal{M}(\phi)/\mathbb{R})}\cdot U_{1}^{E_{s_{1}}^{1}(\phi)}\cdots U_{l}^{E_{s_{l}}^{l}(\phi)}\cdot U_{l+1}^{n_{w_{l+1}}(\phi)}\cdots U_{k}^{n_{w_{k}}(\phi)}\cdot {\bf y},
\end{equation}
where $E_{s}^{i}(\phi)$ is defined by
\[
E_{s}^{i}(\phi)=\max\{s-A_{i}({\bf x}),0\}-\max\{s-A_{i}({\bf y}),0\}+n_{z_{i}}(\phi)=\max\{A_{i}({\bf x})-s,0\}-\max\{A_{i}({\bf y})-s,0\}+n_{w_{i}}(\phi).
\]
For simplicity, we also write
\[
{\bf U}^{E_{\mathrm{s}}(\phi)}=\prod_{i=1}^{l}U_{i}^{E_{s_{i}^i}(\phi)}\prod_{i=l+1}^{k}U_{i}^{n_{w_{i}}(\phi)}.
\]
\end{defn}

When the Heegaard diagram in the context is unique, we simply denote
$\mathfrak{A}^{-}(\mathcal{H},\text{s})$ by $A_{s_{1},s_{2}}^{-}(L)$
or $A_{s_{1},s_{2}}^{-}$, where $\text{s}=(s_{1},s_{2})$. The direct
product of all the generalized Floer complexes forms the first input
of the surgery formula.

\begin{rem}
Let us explain the relation between $A_{\mathrm s}^{-}(L)$ and $CFL^{-}(L).$
First, the filtered chain complex $CFL^{-}(L)$ defined in
\cite{OS_link_FLoer} is the chain complex $CF^-(S^3)$ with
the Alexander filtration. Second, the subcomplexes forming the Alexander
filtration are isomorphic to the $A_{\mathrm s}^-(L)$'s.
The Equation \eqref{eq:A_s} is an explicit formulation of those differentials
in $A_{\mathrm s}^-(L)$.
\end{rem}

\subsection{Polygon maps and homotopy equivalences between Floer complexes.}

In the Fukaya category of a symplectic manifold $(X,\omega)$ (when
it is well-defined), the product of morphisms
\[
\mu^{2}:\text{Hom}(L_{1},L_{2})\otimes\text{Hom}(L_{0},L_{1})\rightarrow\text{Hom}(L_{0},L_{1})
\]
 is defined by counting holomorphic triangles. In general, higher
products are defined by means of holomorphic polygons. In Heegaard
Floer theory, the polygon maps are defined similarly. However, the
technical issue is the compactness of moduli spaces of holomorphic
polygons in the symmetric product of the Heegaard surface, i.e. whether
the polygon counts are finite. This problem breaks down to periodic
domains on the Heegaard surface. Admissibility of Heegaard multi-diagrams
solves this problem. For more details, one can see Section 4.4 in \cite{link_surgery}.

\begin{defn}[Strongly equivalent Heegaard diagrams]
(1) If two Heegaard diagrams $\mathcal{H}$ and $\mathcal{H}'$ have
the same underlying Heegaard surface $\Sigma$, and their collections
of curves $\beta$ and $\beta'$ are related by isotopies and handleslides
only (supported away from the basepoints), we say that $\beta$ and
$\beta'$ are \emph{strongly equivalent}.

(2) Two multi-pointed Heegaard diagrams $\mathcal{H}=(\Sigma,\boldsymbol{\alpha},\boldsymbol{\beta},{\bf w},{\bf z},\tau),\mathcal{H}'=(\Sigma',\boldsymbol{\alpha}',\boldsymbol{\beta}',{\bf w}',{\bf z}',\tau')$
are called \emph{strongly equivalent}, if $\Sigma=\Sigma',{\bf w}={\bf w}',{\bf z}={\bf z}',\tau=\tau'$,
the curve collections $\boldsymbol{\alpha}$ and $\boldsymbol{\alpha}'$
are strongly equivalent, and $\boldsymbol{\beta}$ and $\boldsymbol{\beta}'$
are strongly equivalent as well.

(3) We say that two Heegaard diagrams $\mathcal{H}$ and $\mathcal{H}'$
differ by a \emph{surface isotopy} if there is a self-diffeomorphism
$\phi:\Sigma\rightarrow\Sigma$ isotopic to the identity and supported
away from the link $\overrightarrow{L}$, such that $\Sigma=\Sigma'$
and $\phi$ takes all the attaching curves and basepoints on $\Sigma$
to the corresponding one on $\Sigma'$. If $\mathcal{H}$ and $\mathcal{H}'$
are surface isotopic, we write $\mathcal{H}\cong\mathcal{H}'$.
\end{defn}

\begin{defn}[Triangle maps] Let $(\Sigma,\boldsymbol{\alpha},\boldsymbol{\beta},\boldsymbol{\gamma},{\bf w},{\bf z})$
be a generic, admissible Heegaard triple-diagram, where $\boldsymbol{\beta}$
and $\boldsymbol{\gamma}$ are strongly equivalent, such that $\mathcal{H}=(\Sigma,\boldsymbol{\alpha},\boldsymbol{\beta},{\bf w},{\bf z}),\mathcal{H}''=(\Sigma,\boldsymbol{\alpha},\boldsymbol{\gamma},{\bf w},{\bf z})$
are both Heegaard diagrams of the link $\overrightarrow{L}$ and $\mathcal{H}'=(\Sigma,\boldsymbol{\beta},\boldsymbol{\gamma},{\bf w},{\bf z})$
is the Heegaard diagram of the unlink in  $\#^{g+k-1}(S^{1}\times S^{2})$.
Then we can define the triangle map
\[
f_{\alpha\beta\gamma}:\mathfrak{A}^{-}(\mathbb{T}_{\alpha},\mathbb{T}_{\beta},\text{s})\otimes\mathfrak{A}^{-}(\mathbb{T}_{\beta},\mathbb{T}_{\gamma},\text{s}')\rightarrow\mathfrak{A}^{-}(\mathbb{T}_{\alpha},\mathbb{T}_{\gamma},\text{s+s}')
\]
\begin{align*}
f_{\alpha\beta\gamma}({\bf x}\otimes{\bf y})= & \sum_{{\bf z}\in\mathbb{T_{\alpha}\cap\mathbb{T}_{\beta}}}\sum_{\{\phi\in\pi_{2}({\bf x},{\bf y},{\bf z})|\mu(\phi)=0\}}\#(M(\phi))\cdot{\bf U}^{E_{\text{s},\text{s}'}(\phi)}{\bf z},
\end{align*}
where
\begin{align*}
{\bf U}^{E_{\text{s},\text{s}'}(\phi)}= & U_{1}^{E_{s_{1},s_{1}'}(\phi)}\cdots U_{l}^{E_{s_{l},s_{l}'}(\phi)}U_{l+1}^{n_{w_{l+1}}(\phi)}\cdots U_{k}^{n_{w_{k}}(\phi)},\text{s}=(s_{1},\dots,s_{l}),\text{s}'=(s_{1}',\dots,s_{l}'),\\
E_{s,s'}^{i}(\phi)= & \max\{A_{i}({\bf x})-s,0\}+\max\{A_{i}({\bf y})-s',0\}-\max\{A_{i}({\bf z})-s-s',0\}+n_{w_{i}}(\phi).
\end{align*}
\end{defn}

\begin{defn}[Quadrilateral maps] Let $(\Sigma,\boldsymbol{\eta}^{0},\boldsymbol{\eta}^{1},\boldsymbol{\eta}^{2},\boldsymbol{\eta}^{3},{\bf w},{\bf z})$
be a generic, admissible multi-diagram, such that there are two equivalence
classes of strongly equivalent attaching curves among $\{\boldsymbol{\eta}^{i}\}_{i}$,
and $\boldsymbol{\eta}^{0},\boldsymbol{\eta}^{3}$ are in different
equivalent classes so that $(\Sigma,\boldsymbol{\eta}^{0},\boldsymbol{\eta}^{3},{\bf w},{\bf z})$
is a Heegaard diagram for the link $\overrightarrow{L}$. Now we can
define the quadrilateral maps
\begin{align*}
f_{\eta^{0},...,\eta^{3}}: & \bigotimes_{i=1}^{3}\mathfrak{A}^{-}(\mathbb{T}_{\eta^{i-1}},\mathbb{T}_{\eta^{i}},\text{s}_{i})\rightarrow\mathfrak{A}^{-}(\mathbb{T}_{\eta^{0}},\mathbb{T}_{\eta^{3}},\text{s}_{1}+\text{s}_{2}+\text{s}_{3})\\
f_{\eta^{0},...,\eta^{3}}({\bf x}_{1}\otimes{\bf x}_{2}\otimes{\bf x}_{3})= & \sum_{{\bf y}\in\mathbb{T}_{\eta^{0}}\cap\mathbb{T}_{\eta^{3}}}\sum_{\{\phi\in\pi_{2}({\bf x}_{1},{\bf x}_{2},{\bf x}_{3},{\bf y})|\mu(\phi)=-1\}}\#(M(\phi))\cdot{\bf U}^{E_{\text{s}_{1},\text{s}_{2},\text{s}_{3}}(\phi)}{\bf y},
\end{align*}
 where
\begin{align*}
{\bf U}^{E_{\text{s}_{1},\text{s}_{2},\text{s}_{3}}(\phi)}= & U_{1}^{E_{s_{1}^{1},s_{2}^{1},s_{3}^{1}}^{1}(\phi)}\cdots U_{l}^{E_{s_{1}^{l},s_{2}^{l},s_{3}^{l}}^{l}(\phi)}U_{l+1}^{n_{w_{l+1}}(\phi)}\cdots U_{k}^{n_{w_{k}}(\phi)},\text{s}_{i}=(s_{i}^{1},...,s_{i}^{l}),i=1,2,3,\\
E_{s_{1},s_{2},s_{3}}^{i}(\phi) & =\max\{A_{i}({\bf x}_{1})-s_{1},0\}+\max\{A_{i}({\bf x}_{2})-s_{2},0\}\\
 & +\max\{A_{i}({\bf x}_{3})-s_{3},0\}-\max(A_{i}({\bf y})-s_{1}-s_{2}-s_{3},0)+n_{w_{i}}(\phi).
\end{align*}
\end{defn}

One can define higher polygon counts $f_{\eta^{0}\cdots\eta^{l}}$
similarly, although the case $l>3$ will not be needed in this paper.
For simplicity, we ignore the subscripts of $f_{\eta^{0}\cdots\eta^{l}}$.
An important property of polygon maps is the so-called \emph{quadratic
$A_{\infty}$-associativity equation}
\begin{equation}
\sum_{0\leq i<j\leq l}f(x_{1},...x_{i},f(x_{i+1},...,x_{j}),x_{j+1},...,x_{l})=0.\label{eq:A-infinity}
\end{equation}

\subsection{Nice diagrams.}

In \cite{Sakar-Wang}, Sarkar and Wang use nice Heegaard diagrams
to combinatorially compute the $\widehat{HF}(M)$. This algorithm
is based on a fact: in a nice diagram $\mathcal{H}=(\Sigma,\boldsymbol{\alpha},\boldsymbol{\beta},{\bf w})$,
the index-$1$ pseudo-holomorphic disks in  $\text{Sym}^{g+k-1}(\Sigma)$
with $n_{w}=0$ have simple domains on $\Sigma$ and can be combinatorially
counted.

\begin{defn}[Nice diagrams] A Heegaard diagram $\mathcal{H}=(\Sigma,\boldsymbol{\alpha},\boldsymbol{\beta},{\bf w})$
is called \emph{nice}, if any region (i.e. connect component of $\Sigma-\boldsymbol{\alpha}-\boldsymbol{\beta}$)
without any basepoint $w_{i}\in{\bf w}$ is either a bigon or a square.
For ${\bf x},{\bf y}\in\mathbb{T}_{\alpha}\cap\mathbb{T}_{\beta}$,
a domain $\phi\in\pi_{2}({\bf x},{\bf y})$ is called an \emph{empty
embedded $2n$-gon, }if it is an embedded disk with $2n$ vertices
on its boundary, such that for each vertex $v$, $\mu_{v}(\phi)=\frac{1}{4}$,
and it does not contain any $x_{i}$ or $y_{i}$ in its interior.
An empty embedded $4$-gon is also called empty embedded square.
\end{defn}

\begin{rem}
The notation $\pi_{2}({\bf x},{\bf y})$ in \cite{Sakar-Wang} denotes
the sets of domains, namely $2$-chains $\phi$ on $\Sigma$ such
that $\partial(\partial\phi|_{\alpha})=\mathbf{y}-\mathbf{x}$, whereas
in this paper $\pi_{2}(\mathbf{x},\mathbf{y})$ denotes the homology
classes of Whitney disks in $\text{Sym}^{g+k-1}(\Sigma)$ from
${\bf{x}}$ to ${\bf{y}}$.
\end{rem}

\begin{thm}[\cite{Sakar-Wang}]
Let $\phi\in\pi_{2}({\bf x},{\bf y})$
be a domain on a nice diagram such that $\mu(\phi)=1$ and $n_{w_{i}}(\phi)=0,\forall i$.
If $\phi$ has a holomorphic representative, then $\phi$ is either
an empty embedded bigon or an empty embedded square. Conversely, if
$\phi\in\pi_{2}(x,y)$ with $n_{w_{i}}(\phi)=0,\forall i$ is an empty
embedded bigon or an empty embedded square, then the product complex
structure on $\Sigma\times D^{2}$ achieves transversality for $\phi$
under a generic perturbation of the $\alpha$ and the $\beta$ curves,
and $\mu(\phi)=1$ as well as $\#M(\phi)/\mathbb{R}=1\pmod2$.
\end{thm}
The above theorem enables us to combinatorially count differentials
in $\widehat{CF}$ on a nice diagram, by counting empty embedded bigons
and squares. Following the same lines of the proof, we can obtain the
following adaption.
\begin{prop}
\label{prop:Super nice diagram}Suppose $\mathcal{H}=(\Sigma,\boldsymbol{\alpha},\boldsymbol{\beta},{\bf w},{\bf z})$
is a Heegaard diagram such that any region of $\Sigma$ is either
a bigon or a square. Then, there is a 1-1 correspondence between the
differentials in $\mathfrak{A}{}^{-}(\mathcal{H},{\bf \infty})$ and
the set of empty embedded bigons and empty embedded squares. Thus,
the complex $\mathfrak{A}^{-}(\mathcal{H},\mathrm{s})$ can be described combinatorially.
\end{prop}

\section{Generalized Floer complexes of two-bridge links}

In this section, we combinatorially compute the generalized Floer complexes $A_{s_{1},s_{2}}^{-}(\overrightarrow{L})$
for all two-bridge links $\overrightarrow{L}$  by
using nice diagrams.

\subsection{Schubert normal form.}

A two-bridge link/knot can be obtained by closing a rational tangle.
For the definition of rational tangles, one can see the reference
\cite{knot_theory[M]} chapter 9 and \cite{burde_Zieschang} chapter
7E, 12D. Let us adopt the notations in \cite{burde_Zieschang}. By
$b(p,q)$ where $\gcd(p,q)=1$, we denote the two-bridge link/knot
according to the rational tangle of slope $\frac{q}{p}$.

\begin{defn}[Schubert normal form] For a two-bridge link/knot $L=b(p,q)$, the
\emph{Schubert normal form} is a canonical projection of $L$ with two
over-bridges and two under-bridges, where we regard the projection
 plane as a sphere $S$ in $S^{3}$. The two
over-bridges $O_{1},O_{2}$ are straight segments on the projection
plane, and each component of the other two under-bridges $U_{1},U{}_{2}$
 crosses $O_{1},O_{2}$ alternatively. Together
with the lower half space (which is a ball in $S^{3}$), the under-bridges
$U_{1},U_{2}$ form a rational tangle of slope $\frac{q}{p}$. Moreover,
if $L$ has two components, we arrange the notation such that $L_{i}=O_{i}\cup U_{i}.$
We denote the Schubert form by $(S,O_{1},O_{2},U_{1},U_{2})$.

Concretely, the Schubert normal form can be obtained by gluing two
disks $D_{1}^{\alpha},D_{2}^{\alpha}$ shown in Figure \ref{schubert form}.
The endpoint $a_{i}$ is glued to $a'_{q-i}$, where
all the subscripts are modulo $2p$. When $L$ has two components,
$L$ can be endowed with a canonical orientation induced by the orientation
of $\overrightarrow{O_{1}}=\overrightarrow{a_{0}a_{p}},\overrightarrow{O_{2}}=\overrightarrow{a'_{0}a'_{p}}$,
which is also shown in Figure \ref{schubert form}.
\end{defn}

\begin{figure}
\centering
\includegraphics[scale=0.7]{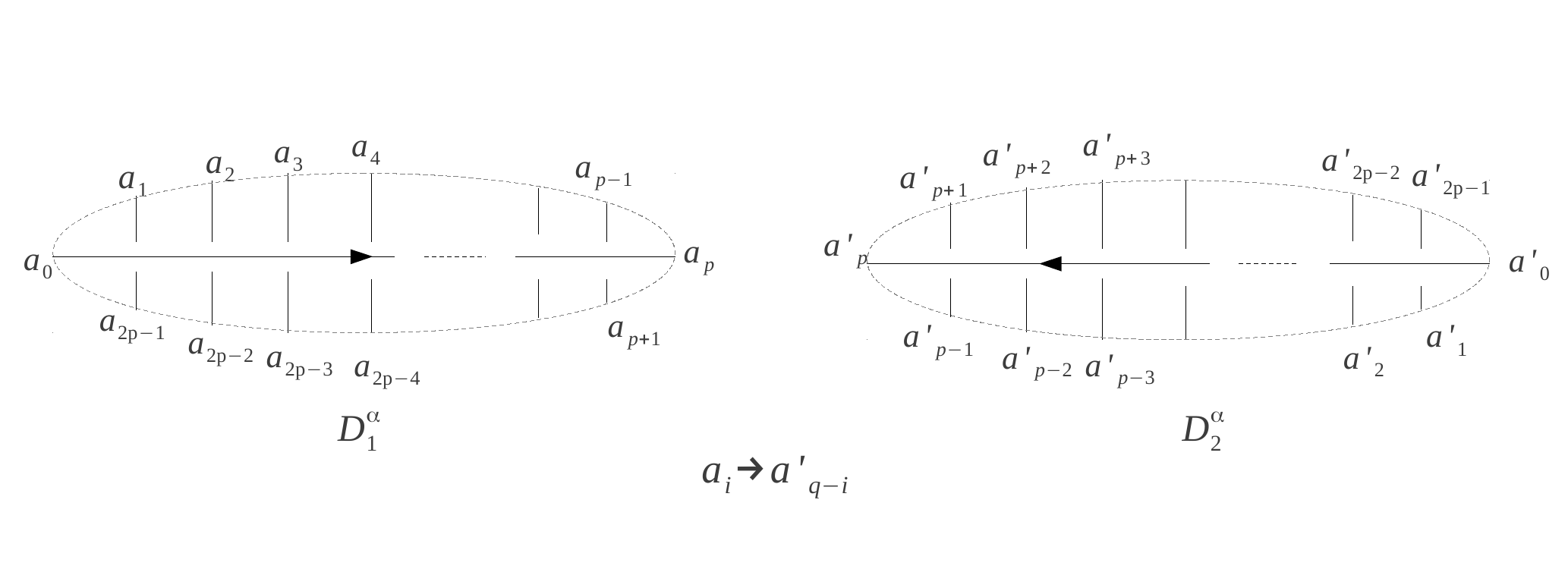}
\caption{\textbf{The Schubert form: the neighborhoods of the two over-bridges.}}
\label{schubert form}
\end{figure}

\begin{example}
In Figure \ref{whitehead}, we show the Schubert normal form of the
Whitehead link $b(8,3)$.
\end{example}

\begin{figure}
\centering
\includegraphics[scale=0.45]{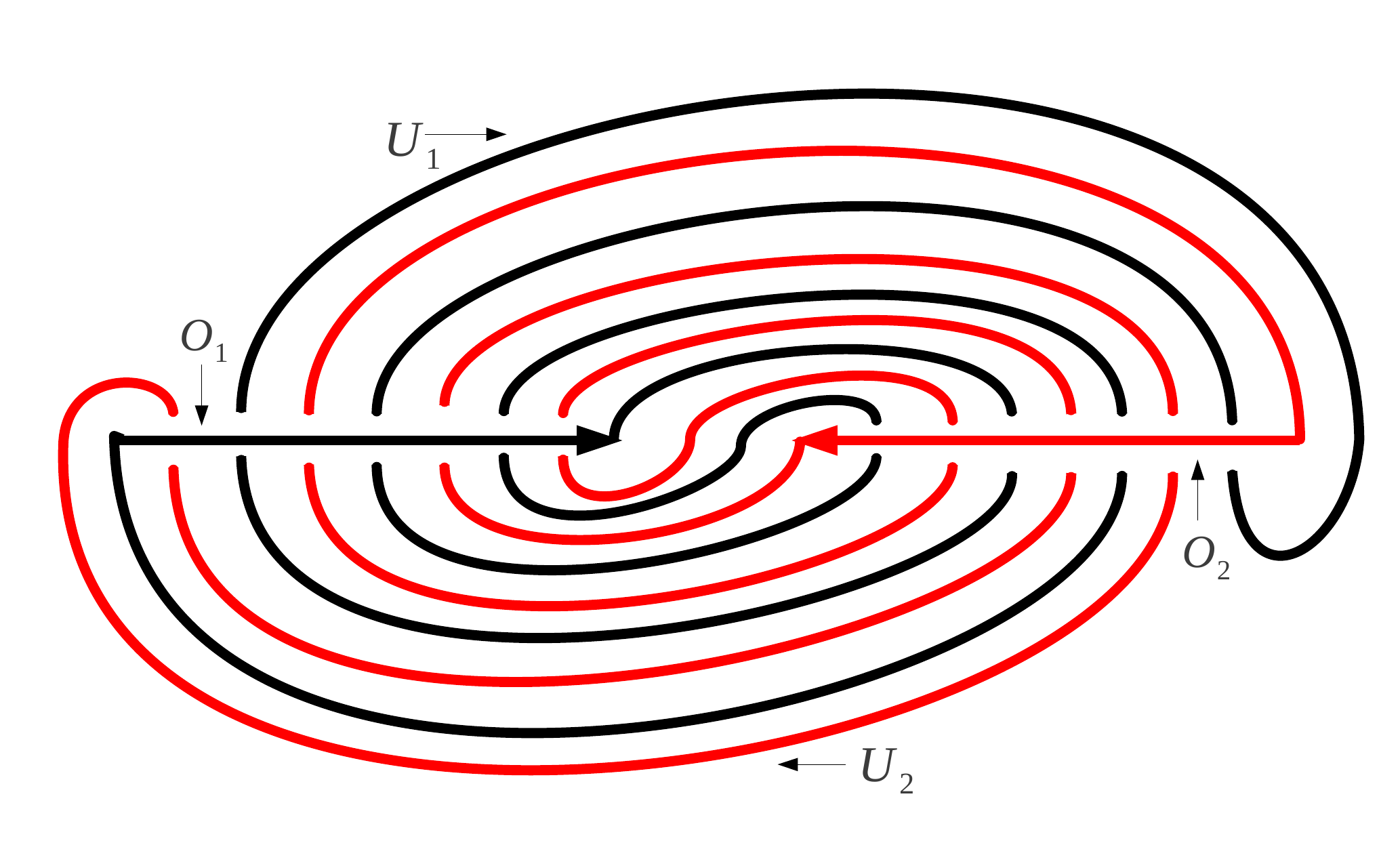}
\caption{\textbf{The Schubert normal form of the Whitehead link.}}
\label{whitehead}
\end{figure}

\begin{rem}
Here in the definition of Schubert normal form, we set the straight
arcs to be over-bridges, while in other places the straight
arcs are set to be under-bridges. However, given $p,q$, these two links are the same up
to taking the mirror of each other.
\end{rem}

\begin{fact}
Let $b(p,q)$ denote the two-bridge link defined as above, where $p,q\in\mathbb{Z},\,\gcd(p,q)=1$,
$p>0$. Then

(1)\emph{(12.1. in \cite{burde_Zieschang})} When $p$ is odd, the link $b(p,q)$ is a knot; and when $p$
is even, $b(p,q)$ has two components.

(2)\emph{(Schubert, \cite{burde_Zieschang}, Theorem 12.6.)}
As an oriented link/knot, $b(p,q)$ is equivalent to $b(p',q')$ if
and only if $p'=p,q'\equiv q^{\pm1}(mod\ 2p)$; as an unoriented link/knot,
$b(p,q)$ is equivalent to $b(p',q')$ if and only if $p'=p,q'\equiv q^{\pm1}(mod\ p)$.

(3)\emph{(12.8 in \cite{burde_Zieschang})} When $b(p,q)$ has two components, the link $b(p,-q)$ is the
mirror of $b(p,q)$, and the link $b(p,q+p)$ can be obtained by changing
the orientation on one component of $b(p,q)$.

(4)\emph{(Remark 12.7, \cite{burde_Zieschang})}
The linking number can be computed by the formula:
\[
\mathrm{lk}(b(p,q))=-\sum_{i=1}^{\frac{p}{2}}(-1)^{\left\lfloor \frac{(2i-1)q}{p}\right\rfloor }.
\]

(5)\emph{(Theorem 9.3.6, \cite{knot_theory[M]})}
The signature can be computed by the formula:
\[
\sigma(b(p,q))=\sum_{i=1}^{p-1}(-1)^{\left\lfloor \frac{iq}{p}\right\rfloor }.
\]
\end{fact}

\subsection{Heegaard diagrams of two-bridge links.}

In this section, we construct nice Heegaard diagrams of two-bridge
links by using their Schubert forms.
\begin{defn}
A \emph{bridge presentation} of a link $L$ is a topological pair $(L,S)$ inside $S^3$, such that
\begin{itemize}
\item $S$ is an embedded sphere transversely intersecting $L$,
\item $S^{3}-S\cong B_{1}\cup B_{2}$, where $B_{1},B_{2}$ are homeomorphic to the unit ball $B^3$,
\item each pair $(L\cap B_{i},B_{i})$ with $i=1,2$ is homeomorphic to the pair $(\{P_{j}\}_{j=1}^{k}\times I,D^{2}\times I)$, where $\{P_{j}\}_{j=1}^{k}$
is a set of points in the interior of the unit disk $D^{2}$.
\end{itemize} The minimum over
all possible $k$ is called the \emph{bridge number} of the link, denoted
by $\mathrm{br}(L)$.
\end{defn}
Every bridge presentation of $L$ gives rise to a genus-0 multi-pointed Heegaard
diagram for  $L$. Let the sphere $S$  be the Heegaard surface.
In each ball $B_{i}$, choose $k-1$ disjoint proper disks
to divide $B_{i}$ into $k$ chambers, such that every component of
the $k$ bridges is in a distinct chamber. The boundaries of these
disks are the alpha, beta curves. The basepoints $w_{i},z_{i}$ are
the intersection points of $L$ and $S$. In other words, by pushing all the bridges
onto the sphere $S$, we obtain a projection of $L$ consisting of $2k$ arcs
$a_1,\cdots,a_k,b_1,\cdots,b_k$, such that the arcs $\{a_i\}$ are disjoint, the arcs
$\{b_j\}$ are disjoint, and the arcs $\{a_i\}$ are always over the arcs $\{b_j\}$.
Then the boundaries of tubular neighborhoods of the arcs $\{a_{i}\}_{i=1}^{k-1},\{b_{j}\}_{j=1}^{k-1}$
are the alpha, beta curves.

\begin{defn}[Schubert Heegaard diagrams] Let $\overrightarrow{L}=\overrightarrow{L_{1}}\cup\overrightarrow{L_{2}}$
be a two-bridge link, and let $(S,O_{1},O_{2},U_{1},U_{2})$ be its Schubert
form. The \emph{Schubert Heegaard diagram} of $L$ is the Heegaard diagram
$\mathcal{H}=(S^{2},\{\alpha\},\{\beta\},\{z_{1},z_{2}\},\{w_{1},w_{2}\})$, where
\begin{itemize}
\item $\alpha=\partial N(O_{1}),\beta=\partial N(U_{1})$ with $N(O_{1}),N(U_{1})$ being disjoint tubular neighborhoods of
$O_{1},U_{1}$ on $S$,
\item $\{z_{1},w_{1}\}=\{L_{1}\cap S\}$ and $\{z_{2},w_{2}\}=\{L_{2}\cap S\}$.
\end{itemize}
\end{defn}

Concretely, regarding the Schubert form as the gluing of
two disks $D_{1}^{\alpha},D_{2}^{\alpha}$ in Figure \ref{schubert form},
we can take $\alpha=\partial D_{1}^{\alpha}$ and $\beta=\partial N(U_{1})$.
The basepoint $z_{1}$ can be any point in $D_{1}^{\alpha}$ near
 $a_{p}$, and the basepoint $w_{1}$ can be any point in
$D_{1}^{\alpha}$ near $a_{0}$; whereas the basepoint $z_{2}$
can be any point in $D_{2}^{\alpha}$ near $a_{p}'$,
and the basepoint $w_{2}$ can be any point in $D_{2}^{\alpha}$ near $a_{0}'.$

\begin{example}
The two-bridge link $b(8,3)$ is the Whitehead link $\mathit{Wh}$
(or its mirror due to the convention). The Schubert Heegaard diagram
of $\mathit{Wh}$ is in Figure \ref{whitehead Heegaard}.
\end{example}

\begin{notation}
\label{notations of diagram} Since we will
repeatedly discuss the Schubert Heegaard diagram, it is convenient
to make a notational convention for all the intersection points and regions as follows.
\begin{itemize}

\item
The components of $S-\beta$ are both disks, denoted by $D_{1}^{\beta},D_{2}^{\beta}$
such that the disk $D_{i}^{\beta}$ is a neighborhood of $U_{i}$.

\item
There is a total of four bigons among the components of $S-(\alpha\cup\beta)$,
and each of them contains a distinct basepoint in $\{z_{1},z_{2},w_{1},w_{2}\}$.
All the other components are squares.

\item
We label all the $p+1$ components of $D_{1}^{\alpha}-\beta$ by $X_{0},X_{1},...,X_{p}$,
and label all the $p+1$ components of $D_{2}^{\alpha}-\beta$ by
$Y_{0},Y_{1},...,Y_{p}$, such that $X_{0},X_{p},Y_{0}$, and $Y_{p}$
are bigons and $w_{1}\in X_{0},w_{2}\in Y_{0},z_{1}\in X_{p},z_{2}\in Y_{p}$.

\item
There is a total of $2p$ intersection points of $\alpha$ and $\beta$.
We label them by $b_{0},b_{1},...,b_{2p-1}$ clockwise, such that
$b_{0},b_{2p-1}$ are vertices of $X_{0}$,  and $b_{p-1},b_{p}$ are
vertices of $X_{p}$. All the subscripts are modulo $2p$.
\end{itemize}
\end{notation}

The above properties and conventions are illustrated in Figure \ref{heegaard diagram for b(p,q)}.

\begin{figure}
\centering

\includegraphics[scale=0.4]{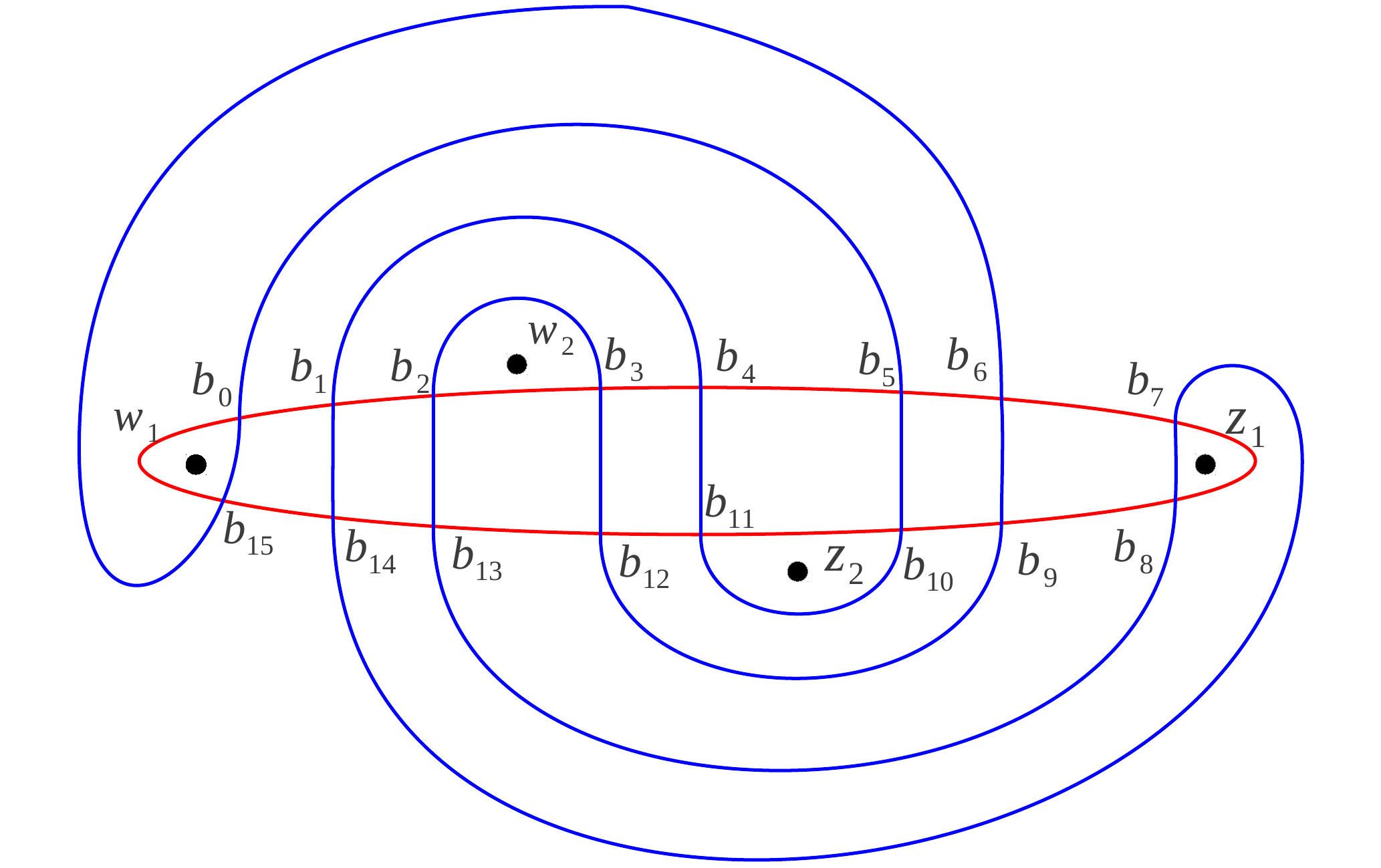}

\caption{\textbf{The Schubert Heegaard diagram of the Whitehead link}. The
red curve is $\alpha$, and the blue curve is $\beta$.}

\label{whitehead Heegaard}

\end{figure}

\begin{figure}
\centering

\includegraphics[width=12cm,height=5cm]{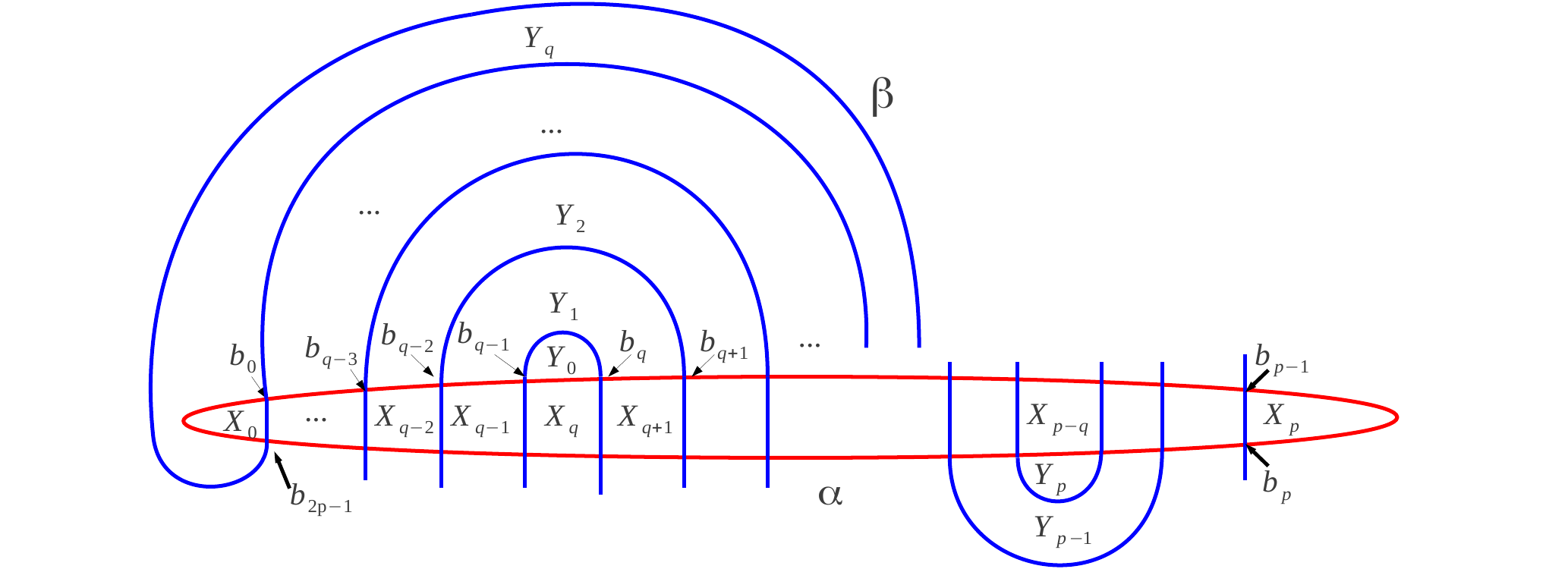}

\caption{\textbf{The Schubert Heegaard diagram for the two-bridge link $b(p,q)$.}}

\label{heegaard diagram for b(p,q)}

\end{figure}

\begin{lem}
The Schubert Heegaard diagram of the two-bridge link $b(p,q)$ is
a nice diagram. By Proposition \ref{prop:Super nice diagram}, the generalized
Floer complexes of the Schubert Heegaard diagram are combinatorial.
\end{lem}

From the property of Schubert normal form, it follows a direct description
of the Schubert Heegaard diagram.

\begin{lem}
\label{lem:Schubert Heegaard diagram}In the Schubert Heegaard diagram
of $b(p,q)$,

(1) in the disk $D_{1}^{\alpha}$, the points
$b_{i}$ and $b_{2p-1-i}$ are connected by a $\beta$-arc,

(2) in the disk $D_{2}^{\alpha}$, the points
$b_{i}$ and $b_{j}$ are connected by a $\beta$-arc if and only
if $i+j\equiv2q-1\ (\mathrm{mod}\ 2p)$.
\end{lem}

\subsection{The multi-variable Alexander polynomial of two-bridge links.}

With the help of link Floer homology, we can directly calculate the
multi-variable Alexander polynomial of knots and links. In \cite{OS_link_FLoer},
there is a formula of the Euler characteristic of $\widehat{HFL}(L)$:
\begin{equation}
\sum_{h\in\mathbb{H}(L)}\chi(\widehat{HFL}_{*}(L,h))\cdot e^{h}=\prod_{i=1}^{l}(T_{i}^{\frac{1}{2}}-T_{i}^{-\frac{1}{2}})\Delta_{L}.\label{eq:AlexPoly}
\end{equation}

\begin{defn} [Thin complex and $E_{2}$-collapsed complex]
Suppose $(C,\partial)$ is a $\mathbb{Z}^{2}$-filtered chain complex of
$\mathbb{F}$-vector spaces. Let $(i,j)$ denote the filtration, and let $g$ denote the
internal grading. The complex  $(C,\partial)$ is called \emph{thin}, if $i+j-g$
is a constant for all elements in $C$. The chain complex $C$ is
called \emph{$E_{2}$-collapsed}, if the differential can be decomposed as
$\partial=\partial_{1}+\partial_{2}$, such that $F(\partial_{1}(x))=F(x)-(1,0)$
and $F(\partial_{2}(x))=F(x)-(0,1)$, where $F(x)$ is the $\mathbb{Z}^{2}$-filtration
of $x$.\end{defn}

\begin{rem}
A thin complex is $E_{2}$-collapsed. The classification
of $E_{2}$-collapsed complexes of $\mathbb{F}$-vector spaces is
shown in \cite{OS_link_FLoer} Section 12.1.
\end{rem}

\begin{prop}
\label{prop:Alexander grading}Let $L=b(p,q)$ be a two-bridge link,
 where $p$ is even and $-p<q<p$.  Let  $A(b_{i})$ be the Alexander
grading, and let $q^{-1}$ be the
number theoretical reciprocal of $q$ modulo $2p$. Then
\[
A(b_{i})-A(b_{i-1})=\begin{cases}
((-1)^{\left\lfloor q^{-1}\cdot i/p\right\rfloor },0), & i\text{ is even,}\\
(0,(-1)^{\left\lfloor (q^{-1}\cdot i+1)/p\right\rfloor }), & i\text{ is odd}.
\end{cases}
\]
Furthermore, we have $A_{1}(b_{i})+A_{2}(b_{i})-M(b_{i})$ is a constant,
i.e. not dependent on $i$, where $M(b_{i})$ is the Maslov grading.
In other words, the chain complex $A_{+\infty,+\infty}^{-}(L)$ is
thin.
\end{prop}

\begin{proof}
We use Notations \ref{notations of diagram}.
There is  a set of bigons of Maslov index 1 connecting $b_{i}$
and $b_{i+1}$, for $i=0,1,...,2p-2.$ Each of these
bigons is a part of one of the disks $D_{1}^{\beta}$ and $D_{2}^{\beta}$.
In fact, these bigons can be obtained by chasing the under-bridges
$U_1$ and $U_2$.

The under-bridge $U_{1}$ starts from $z_{1}=a_{p}={a'}_{q-p}$ and passes the
disks $D_{2}^{\alpha}$ and $D_{1}^{\alpha}$ alternately.  For $i$ even, at the point $a_{i}$, if the under-bridge $U_{1}$ is pointing out of
 $D_{1}^{\alpha}$, then $i\equiv p-2kq\ (\mathrm{mod}\ 2p)$ for some $k$ with $0<k<\frac{p}{2}$,
which is equivalent to $\left\lfloor \frac{i\cdot q^{-1}}{p}\right\rfloor $
is even. In this case, there is a bigon $\phi$ from
$b_{i}$ to $b_{i-1}$ with a single basepoint $z_{1}$ on it, and thereby
 $$A(b_{i})-A(b_{i-1}) = ((-1)^{\lfloor\frac{q^{-1}\cdot i}{p}\rfloor},0).$$
If the under-bridge $U_{1}$ is pointing into
$D_{1}^{\alpha}$, then $i\equiv 2kq-p\ (\mathrm{mod}\ 2p)$ for some $k$
with $0<k<\frac{p}{2}$, which is equivalent to $\left\lfloor \frac{i\cdot q^{-1}}{p}\right\rfloor $
is odd. In this case, there is a bigon $\phi$ from
$b_{i-1}$ to $b_{i}$ with a single basepoint $z_{1}$ on it, and still
 $$A(b_{i})-A(b_{i-1}) = ((-1)^{\lfloor\frac{q^{-1}\cdot i}{p}\rfloor},0).$$

Similarly, by keeping track of $U_{2}$, we can prove the other cases. For $i$ odd,
at the point $a_{i}$, if the under-bridge $U_{2}$ is pointing off $D_{1}^{\alpha}$,
then $i\equiv p-q-2kq\ (\mathrm{mod}\ 2p)$ for some $k$ with $0<k<\frac{p}{2}$,
which is equivalent to $\left\lfloor \frac{1+i\cdot q^{-1}}{p}\right\rfloor $
is even. In this case, there is a bigon $\phi$ of index $1$ from
$b_{i}$ to $b_{i-1}$ with a single basepoint $z_{2}$ on it. Thus,
we have $A(b_{i})-A(b_{i-1})=(0,(-1)^{\lfloor\frac{q^{-1}\cdot i+1}{p}\rfloor})$
for $i$ odd.

From Lipshitz's formula $\mu(\phi)=e(\phi)+\mu_{b_{i}}(\phi)+\mu_{b_{i-1}}(\phi)$,
it follows $\mu(\phi)=1$, and thereby for all $i$,
$$A_{1}(b_{i})+A_{2}(b_{i})-M(b_{i})=A_{1}(b_{i-1})+A_{2}(b_{i-1})-M(b_{i-1}).$$
\end{proof}

Now we are able to compute the multi-variable Alexander polynomial
of two-bridge links by Equation \eqref{eq:AlexPoly}. Since
the Floer chain complex for the Schubert Heegaard diagram is thin,
there are no differentials in the associated graded complex of $\widehat{CFL}(L,h)$
for the Alexander filtration. That is, for this thin complex,
 $\widehat{HFL}(L,h)=\widehat{CFL}(L,h)$. Thus,

\begin{equation*}
{\prod_{i=1}^{l}(T_{i}^{\frac{1}{2}}-T_{i}^{-\frac{1}{2}})}\cdot\Delta_{L}(x,y)=\sum_{i=0}^{2p-1}(-1)^{A_{1}(b_{i})+A_{2}(b_{i})}\cdot x^{A_{1}(b_{i})}\cdot y^{A_{2}(b_{i})}.
\end{equation*}

By computer experiments, we have found two distinct two-bridge links $b(126,47)$ and $b(126,55)$ that
share the same Alexander polynomial, signature, and linking number, but are not the same or mirror
to each other.
\begin{align*}
& \Delta_{b(126,47)}(x,y)=\Delta_{b(126,55)}(x,y)=-15+\frac{8}{x}+8x+\frac{y}{8}+8y-\frac{4}{xy}-4xy-\frac{4x}{y}-\frac{4y}{x},\\
& \sigma(b(126,47))=\sigma(b(126,55))=3,\\
& \mathrm{lk}(126,47)=\mathrm{lk}(126,55)=1.
\end{align*}
\subsection{The Floer complexes for two-bridge links.}

Let $\mathcal{H}=(S,\boldsymbol{\alpha},\boldsymbol{\beta},\text{{\bf w},{\bf z}})$
be the Schubert Heegaard diagram of $b(p,q)$. By Lemma 3.9, the generalized Floer
complex $\mathfrak{A}^{-}(\mathcal{H},\text{s})$ is combinatorial.
It consists of counting the empty embedded bigons of Maslov index
$1$ on $S$, since here $g+k-1=1$.

The pattern of empty embedded bigons of Maslov index $1$
is illustrated in Figure \ref{bigons} as in \cite{Sakar-Wang}.
In Schubert Heegaard diagram, these bigons are always
in a similar form of Figure \ref{bigons}, where the
function $f_{p}$ is defined as follows.

\begin{defn}
For all $n,m,k\in\mathbb{Z}$, let $\mathrm{Mod}(n,m,k)$ be the residue
of $n$ modulo $m$ starting from $k$, that is,
\[
\mathrm{Mod}(n,m,k)\equiv n\ (\mathrm{mod}\ m)\quad \text{and}\quad k\leq \mathrm{Mod}(n,m,k)\leq k+m-1.
\]
Then $f_{p}$ is defined by
\[
f_{p}(n)=\left|\mathrm{Mod}(n,2p,-p+1)\right|.
\]
\end{defn}

\begin{figure}
\centering

\includegraphics[scale=0.3]{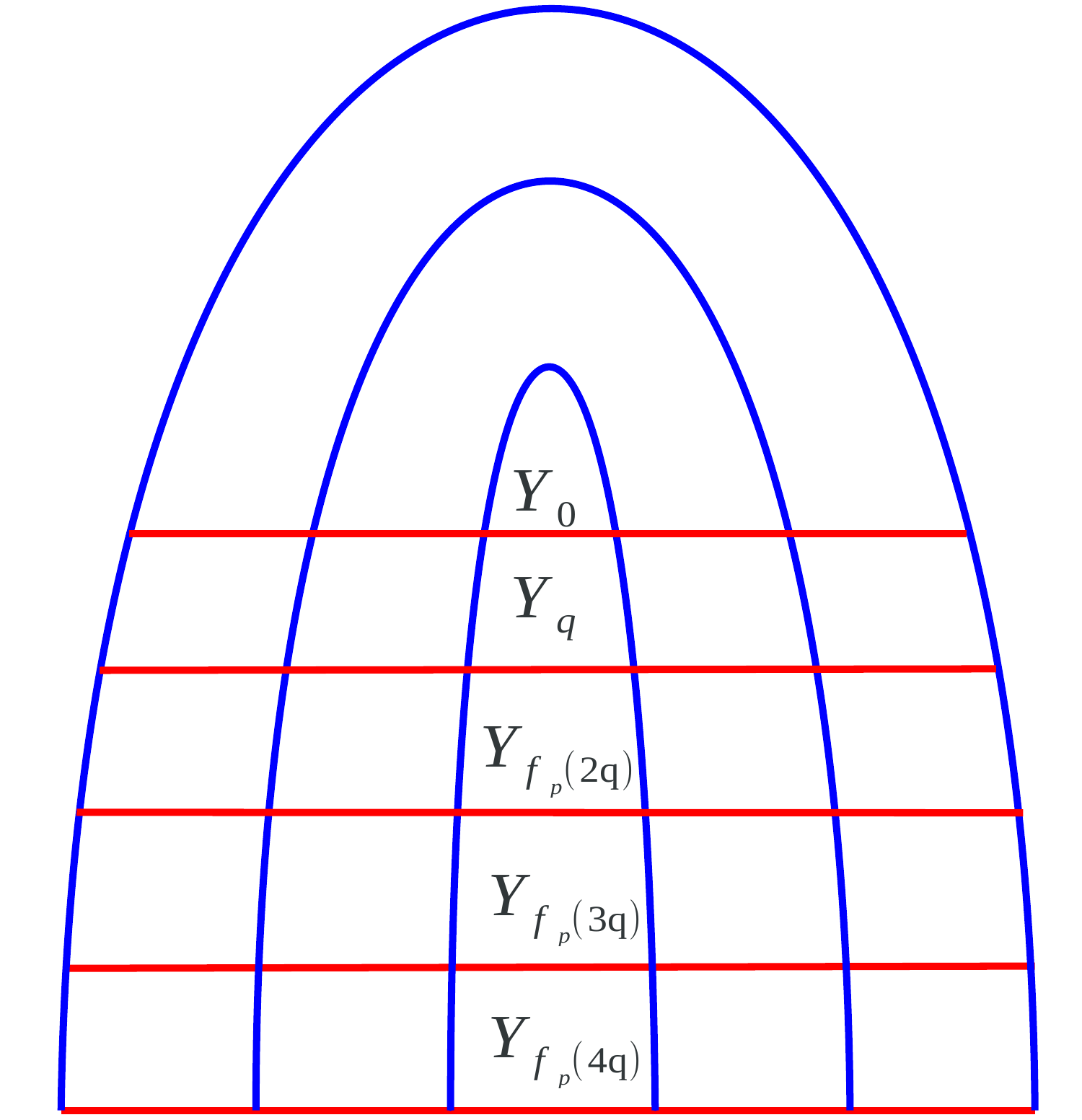}

\caption{\textbf{A bigon in the Schubert Heegaard diagram of a two-bridge link.} The
red lines are parts of $\alpha$, and the blue curves are parts of $\beta$.}

\label{bigons}

\end{figure}
\begin{lem}
\label{lem:X,Y-spine-of-bigon}
In the Schubert Heegaard diagram $(S,\{\alpha\},\{\beta\},\{w_{1},w_{2}\},\{z_{1},z_{2}\})$
of the two-bridge link $\overrightarrow{L}=b(p,q)$, the regions in $D_{1}^{\beta}$ are $X_{0},Y_{q},$$X_{f_{p}(2q)},$$Y_{f_{p}(3q)}$,...,$X_{f_{p}(pq)}=X_{p}$
consecutively, and the regions in $D_{2}^{\beta}$ are $Y_{0},X_{q},$$Y_{f_{p}(2q)},$$X_{f_{p}(3q)}$,...,$Y_{f_{p}(pq)}=Y_{p}$
consecutively.
\end{lem}
\begin{proof}
Note that $D_{i}^{\beta}$ is the regular neighborhood of the under-bridge
$U_{i}$. The region $X_{i}$ contains the arc $a_{i}a_{-i}\subset L$,
and the region $Y_{j}$ contains arc $a'_{j}a'_{-j}\subset L$. Conversely,
the point $a_{i}$ is contained in $X_{f_{p}(i)}$, and the point
$a'_{j}$ is contained in $Y_{f_{p}(j)}$. Thus since $a_{i},a'_{q-i}$
are glued together and $a_{-i},a'_{q+i}$ are glued together, $X_{f_{p}(i)}$
is adjacent to $Y_{f_{p}(q-i)}$ and $Y_{f_{p}(q+i)}$. Since $Y_{0}$
is in $D_{2}^{\beta}$ and it is adjacent to $X_{q}$, the region
$X_{q}$ is adjacent to $Y_{f_{p}(2q)}$. Inductively, we can show
in $D_{2}^{\beta}$, $Y_{f_{p}(kq)}$ is adjacent to $X_{f_{p}[(k-1)q]}$
and $X_{f_{p}[(k+1)q]}$. A similar argument applies to $D_{1}^{\beta}$.
\end{proof}

\begin{defn}
In the bigon $\phi$, denote the number of $\alpha$ arcs in $\phi$
by $n_{\alpha}(\phi)$, and denote the number $\beta$ arcs in $\phi$
by $n_{\beta}(\phi)$.
\end{defn}

Every bigon $\phi$  is uniquely determined by $n_{\alpha}(\phi),n_{\beta}(\phi)$
 and the basepoint on it.

\begin{lem}[Patterns of bigons]
\label{lem:(Pattern-of-bigons)} In the
Schubert Heegaard diagram of $b(p,q)$, suppose $\phi$ is an empty
embedded bigon of index 1 in $\pi_{2}(b_{i},b_{j})$. Then
\begin{equation*}
\begin{aligned}[c]
&(i,j)=((1-n_{\alpha})q+n_{\beta}-1,(1-n_{\alpha})q-n_{\beta}),\\
&(i,j)=(n_{\alpha}q-n_{\beta},n_{\alpha}q+n_{\beta}-1),\\
&(i,j)=(n_{\alpha}q-n_{\beta},n_{\alpha}q+n_{\beta}-1),\\
&(i,j)=((1-n_{\alpha})q+n_{\beta}-1,(1-n_{\alpha})q-n_{\beta}),
\end{aligned}
\qquad
\begin{aligned}[c]
&\text{if }w_{1}\in\phi,n_{\alpha}\text{ is odd,}\\
&\text{if }w_{1}\in\phi,n_{\alpha}\text{ is even,}\\
&\text{if }w_{2}\in\phi,n_{\alpha}\text{ is odd,}\\
&\text{if }w_{2}\in\phi,n_{\alpha}\text{ is even}.
\end{aligned}
\end{equation*}

Furthermore, given $m,n\in\mathbb{Z}$ and a basepoint $pt\in\{w_1,w_2,z_1,z_2\}$,
there exists at most one empty embedded bigon $\phi$ with $n_{\alpha}(\phi)=m$,
$n_{\beta}(\phi)=n$, and $n_{pt}(\phi)=1$.

There exists an empty embedded bigon $\phi$ of index
1 with $n_{\alpha}(\phi)=m$, $n_{\beta}(\phi)=n$, and $n_{w_{2}}(\phi)=1$
if and only if the condition $P_{1}(m,n)$ holds.

The condition $P_{1}(m,n)$ is as follows:
\begin{enumerate}
\item
either $m=1$, or if $m>1$, then the set of intervals: $[0,n-1]$
and all intervals $[f_{p}(2iq)-n+1,f_{p}(2iq)+n-1]$ with $1\leq2i\leq m-1$
are pairwise disjoint intervals in $[0,p]$;
\item
either $m=1$, or if $m>1$, then the set of intervals: all intervals
$[f_{p}((2i+1)q)-n+1,f_{p}((2i+1)q)+n-1]$ with $1\leq2i+1\leq m-1$
are also pairwise disjoint intervals in $[1,p-1]$.
\end{enumerate}

Similarly, there exists an empty embedded bigon $\phi$
of index 1 with $n_{\alpha}(\phi)=m$, $n_{\beta}(\phi)=n$ and
$n_{w_{1}}(\phi)=1$ if and only if the condition $P_{2}(m,n)$ holds.

The condition $P_{2}(m,n)$ is as follows:
\begin{enumerate}
\item
either $n=1$, or $n>1$ and the set of intervals: $[0,m-1]$ and
all intervals $[f_{p}(2iq)-m+1,f_{p}(2iq)+m-1]$ with $1\leq2i\leq n-1$
are pairwise disjoint intervals in $[0,p]$;
\item
either $n=1$, or $n>1$ and the set of intervals: all intervals $[f_{p}((2i+1)q)-m+1,f_{p}((2i+1)q)+m-1]$
with $1\leq2i+1\leq n-1$ are also pairwise disjoint intervals in
$[1,p-1]$.
\end{enumerate}

In addition, for $i=1,2$, there is a one-to-one correspondence between
the set of all the empty embedded bigons with $n_{w_{i}}=1$ and the
set of empty embedded bigons with $n_{z_{i}}=1$, where the bigon
$\phi\in\pi_{2}(b_i,b_j)$ with $n_{\alpha}=m,n_{\beta}=n$,$n_{w_{i}}=1$
is sent to the bigon $\phi'\in\pi_{2}(b_i+p,b_j+p)$ with $n_{\alpha}=m,n_{\beta}=n$,$n_{z_{i}}=1$.
\end{lem}

\begin{proof}
Suppose $\phi$ is a bigon of index 1 in $\pi_{2}(b_{i},b_{j})$ with
$n_{w_{2}}(\phi)=1$. Combining Lemma \ref{lem:Schubert Heegaard diagram}
and Lemma \ref{lem:X,Y-spine-of-bigon}, we can get the formula of $(i,j)$
out of Figure \ref{bigons} by induction on $n_{\alpha},n_{\beta}$.
The initial step is  $(i,j)=(q-1,q)$, for $n_{\alpha}=n_{\beta}=1$.
Similarly, we can show the other case where $n_{w_{1}}(\phi)=1$.

For the second part, the sufficient and necessary condition of when
there exists an empty embedded bigon is that all the regions in the
bigon are not overlapped. By Lemma \ref{lem:X,Y-spine-of-bigon}, it is not hard to
get the formulas by induction.

Finally, notice that there is a symmetry of the Heegaard Schubert
diagram which sends $b_{k}$ to $b_{k+p}$ and exchanges $w_{i}$ to $z_{i}$
for all $1\leq k\leq 2p,i=1,2$. This symmetry directly gives the one-to-one correspondence
between the bigons with $w_{i}$ and the ones with $z_{i}$.
\end{proof}

Consequently, we get an algorithm for computing $A_{\mathrm{s}}^{-}(b(p,q))$
as follows.

\begin{thm}
\label{thm:bigon}Let $L=b(p,q)$ be a two-bridge link and $\mathcal{H}$ be the Schubert Heegaard diagram.
Define functions $F_{i}:\mathbb{N}\times\mathbb{N}\rightarrow\mathbb{Z}/2p\mathbb{Z}\times\mathbb{Z}/2p\mathbb{Z},i=1,2,$
by
\begin{align*}
F_{1}(m,n)= & \begin{cases}
((1-m)q+n-1,(1-m)q-n), & \text{if }m\text{ is odd},\\
(mq-n,mq+n-1), & \text{if }m\text{ is even.}
\end{cases}\\
F_{2}(m,n)= & \begin{cases}
((1-m)q+n-1,(1-m)q-n), & \text{if }m\text{ is even},\\
(mq-n,mq+n-1), & \text{if }m\text{ is odd.}
\end{cases}
\end{align*}
The conditions $P_{i}(m,n),i=1,2$ are as in Lemma \ref{lem:(Pattern-of-bigons)}.

Then, the complex $\mathfrak{A}^{-}(\mathcal{H},+\infty,+\infty)$
is a free $\mathbb{F}[[U_{1},U_{2}]]$-module generated by $g_{0},...g_{2p-1}$
with differentials
\[
\partial g_{i}=\sum_{j=0}^{2p-1}(\lambda_{i+p,j+p}+\mu_{i+p,j+p})g_{j}+\sum_{j=0}^{2p-1}\lambda_{i,j}U_{1}g_{j}+\sum_{j=0}^{2p-1}\mu_{i,j}U_{2}g_{j},
\]
where the coefficients $\lambda_{i,j},\mu_{i,j}\in{\mathbb{Z}/2\mathbb{Z}}$
are determined by the following equations
\begin{align*}
\lambda_{i,j} & =\#\{(m,n)\in\mathbb{N}\times\mathbb{N}|1\leq m,n\leq p,F_{1}(m,n)=(i,j),P_{1}(m,n)\text{ is true}\}\pmod2,\\
\mu_{i,j} & =\#\{(m,n)\in\mathbb{N}\times\mathbb{N}|1\leq m,n\leq p,F_{2}(m,n)=(i,j),P_{2}(m,n)\text{ is true}\}\pmod2.
\end{align*}
\end{thm}
\begin{rem}
One can get an algorithm of $O(p^{2})$ time complexity for computing
$A_{\mathrm{s}}^{-}(b(p,q))$. To compute $A_{\mathrm{s}}^{-}(b(p,q))$,
we only need to know $A_{+\infty,+\infty}^{-}(b(p,q))$, which is
determined by all the counting of bigons, i.e. those $\lambda_{i,j}$'s
 and $\mu_{i,j}$'s. Computing
$\lambda_{i,j}$'s and computing $\mu_{i,j}$'s are similar. In order
to get all the $\lambda_{i,j}$'s, one can nest two loops. The outer
loop is indexed by $n\geq1$, and the inner loop is indexed by $m\geq1$
with a test condition $P_{1}(m,n)$. When $P_{1}(m,n)$ is true, we change
the value of $\lambda_{F_{1}(m,n)}$ by $1\ (\mathrm{mod}\ 2)$ and keep
running the inner loop; when $P_{1}(m,n)$ is false, we stop the inner
loop and go back to the outer loop.

Let us estimate the time complexity. Switching the $\alpha$ and $\beta$ roles
converts the Heegaard diagram of $b(p,q)$ to its mirror $b(p,-q)$. Thus, we assume $0<q<p.$ First, when $P_{1}(m,n)$ is
true and $n>q$, $m$ must be $1$, as otherwise the second part
of $P_{1}(m,n)$ would imply $f_{p}(q)-n+1=q-n+1\ge1$. Thus we force
the outer loop stop when $n=q+1$. Computing the other $\lambda_{F_{1}(m,n)}$'s
with $(m,n)=(1,n),n>p$ can be done within $O(p)$ operations. Second, if
$P_{1}(m,n)$ is true, then $m\leq(p+1)/n$. This is because the first part of $P_{1}(m,n)$ implies that
there are $m$ pieces of open intervals $(f_{p}(2iq)-n+\frac{1}{2},f_{p}(2iq)+n-\frac{1}{2})$
pairwise disjoint in $(-\frac{1}{2},p+\frac{1}{2})$.
Thus, at the $n^{\text{th}}$ step of the outer loop, the inner
loop stops within $\left\lfloor(p+1)/n\right\rfloor$ steps. Finally, testing $P_{1}(m,n)$ can be done
within $2m$ steps. In fact, when $m$ is even, we can check if the
new interval $[f_{p}(mq)-n+1,f_{p}(mq)+n-1]$ is disjoint from the
other $\frac{m}{2}-1$ intervals in the first part of $P_{1}(m,n)$
(which are already ordered in the previous step). If it is disjoint
from the other intervals, we put it in the correct position in the
order. This is done within $2m$ operations. It is similar when $m$
is odd. Thus the time complexity is of the order
\[
T(p,q)=(\sum_{n=1}^{q}\sum_{m=1}^{\left\lfloor(p+1)/n\right\rfloor}2m)+O(p)\leq\left(\sum_{n=1}^{q}[(\frac{p+1}{n})^{2}+\frac{p+1}{n}]\right)+O(p).
\]
Since $n\leq q\leq p$, $(p+1)/n\geq1$, thus
\[
T(p,q)\leq2[\sum_{n=1}^{q}\frac{(p+1)^{2}}{n^{2}}]+O(p)=O(p^{2}).
\]
\end{rem}

\section{Link surgery formula}
In this section, we review the link surgery formula of Manolescu-Ozsv\'{a}th
for two-component links with basic diagrams. In Section 4.1, we review
some algebra on \emph{hyperboxes of chain complexes} and introduce
\emph{twisted gluing} of squares of chain complexes. In Section 4.2, we express
the link surgery formula for a two-component link as a twisted gluing of certain squares of chain complexes
derived from the link. These squares are elaborated in Section 4.3, by using
\emph{primitive systems of hyperboxes}. The primitive systems of hyperboxes
are generalizations of the \emph{basic systems of hyperboxes}
used in \cite{link_surgery}.  One can consult \cite{link_surgery}
for the full generality of link surgery formula with general Heegaard
diagrams. We assume that the reader is  familiar with Heegaard Floer
homology \cite{OS_HF1,OS_HF2,OS_Mixed_inv,OS_link_FLoer}.

Throughout, $\overrightarrow{L}=\overrightarrow{L_{1}}\cup\overrightarrow{L_{2}}$
will be an oriented link in $S^{3}$, and $\overrightarrow{M}$ will
denote an oriented sublink of $\overrightarrow{L}$ which may not
have the induced orientation from $\overrightarrow{L}$ on each component.

\subsection{Hyperboxes of chain complexes. }

\subsubsection{Hyperboxes of chain complexes.}
\begin{defn}[Hyperbox]
An $n$-dimensional \emph{hyperbox} of size ${\bf d}=(d_{1},...,d_{n})\in\mathbb{Z}_{\geq0}^{n}$
is the subset
\[
\mathbb{E}({\bf d})=\{(\varepsilon_{1},...,\varepsilon_{n})\in\mathbb{Z}_{\geq0}^{n}:0\leq\varepsilon_{i}\leq d_{i}\}.
\]
If $\mathbb{E}({\bf d})=\{0,1\}^{n},$ then $\mathbb{E}({\bf d})$
is called a \emph{hypercube}, denoted by $\mathbb{E}_{n}$.
\end{defn}

\begin{defn}[Hyperbox of chain complexes]
\label{defn:hyperbox of chain complexes}
Let $R$ be an $\mathbb{F}$-algebra.
An $n$-dimensional \emph{hyperbox of chain complexes} of size ${\bf d}\in\mathbb{Z}_{\geq0}^{n}$
is a collection of $\mathbb{Z}$-graded $R$-modules
\[
(C^{\varepsilon})_{\varepsilon\in\mathbb{E}(\text{{\bf d}})},C^{\varepsilon}=\underset{*\in\mathbb{Z}}{\bigoplus}C_{*}^{\varepsilon},
\]
together with a collection of $R$-linear maps
\[
D_{\varepsilon^{0}}^{\varepsilon}:C_{*}^{\varepsilon^{0}}\rightarrow C_{*-1+||\varepsilon||}^{\varepsilon^{0}+\varepsilon},
\]
one map for each $\varepsilon^{0}\in\mathbb{E}({\bf d})$ and $\varepsilon\in\mathbb{E}_{n}$
such that $\varepsilon^{0}+\varepsilon\in\mbox{\ensuremath{\mathbb{E}}}({\bf d})$.
The maps are required to satisfy the relations
\[
\sum_{\varepsilon'\leq\varepsilon}D_{\varepsilon^{0}+\varepsilon'}^{\varepsilon-\varepsilon'}\circ D_{\varepsilon^{0}}^{\varepsilon'}=0,
\]
for all $\varepsilon^{0}\in\mathbb{E}({\bf d}),\varepsilon\in\mathbb{E}_{n}$
such that $\varepsilon^{0}+\varepsilon\in\mathbb{E}({\bf d})$.
\end{defn}

By abuse of notation, we sometimes let $D^{\varepsilon}$ stand for
any of its map $D_{\varepsilon^{0}}^{\varepsilon}.$ Note that a hypercube
of chain complexes $H$ gives rise to a total complex of the hypercube
$\text{Tot}(H)$.

\begin{example}[1-dimensional hyperboxes]
\label{eg:dim-1 hyperbox}
A 1-dimensional hyperbox of chain complexes is a sequence of chain
complexes $C_{n}$, together with a sequence of chain maps $f_{n}:C^{(n-1)}\rightarrow C^{(n)}$.

\[
\xymatrix{C^{(0)}\ar[r]^{f_{1}} & C^{(1)}\ar[r]^{f_{2}} & C^{(2)}\ar[r]^{f_{3}} & \cdots\ar[r] & C^{(n-1)}\ar[r]^{f_{n}} & C^{(n)}.}
\]
The total complex of a 1-dimensional hypercube of chain complexes
can be regarded as a mapping cone. Therefore, we also call a 1-dimensional
hyperbox of chain complexes a \emph{sequence of chain complexes}.

\end{example}

\begin{example} [2-dimensional hyperboxes]
\label{eg:dim-2 hyperbox}
A \emph{square of chain complexes} is a 2-dimensional hypercube of
chain complexes:
\[
\begin{array}{cc}
\xyR{1.5pc}\xyC{3pc}\xymatrix{C^{(0,0)}\ar[r]^{D_{(0,0)}^{(1,0)}}\ar[d]_{D_{(0.0)}^{(0,1)}}\ar[dr]|-{D_{(0,0)}^{(1,1)}} & C^{(1,0)}\ar[d]^{D_{(1,0)}^{(0,1)}}\\
C^{(0,1)}\ar[r]_{D_{(0,1)}^{(1,0)}} & C^{(1,1)}
}
= & \xyR{1.5pc}\xyC{3pc}\xymatrix{C^{(0,0)}\ar[r]^{D^{(1,0)}}\ar[d]_{D^{(0,1)}}\ar[dr]|-{D^{(1,1)}} & C^{(1,0)}\ar[d]^{D^{(0,1)}}\\
C^{(0,1)}\ar[r]_{D^{(1,0)}} & C^{(1,1)}.
}
\end{array}
\]

Here $D^{(1,0)},D^{(0,1)}$ are chain maps, and $D^{(1,1)}$ is a
chain homotopy between $D^{(0,1)}\circ D^{(1,0)}$ and $D^{(1,0)}\circ D^{(0,1)}$.
We can regard the total complex of this square as the mapping cone
of

\[
\xyC{3pc}\xymatrix{\text{cone}(D_{(0,0)}^{(1,0)})\ar[rr]^{D_{(0,0)}^{(0,1)}+D_{(1,0)}^{(0,1)}+D_{(0,0)}^{(1,1)}} &  & \text{cone}(D_{(0,1)}^{(1,0)}).}
\]
A \emph{rectangle of chain complexes} is a 2-dimensional hyperbox
of chain complexes. It consists of squares of chain complexes. A rectangle
of chain complexes of size $(m,1)$ can also be regarded as a sequence
of mapping cones, i.e. a size $(m)$ 1-dimensional hyperbox of mapping
cones.
\end{example}

Let $R=(C,D)$ be a hyperbox of chain complexes of size $(d_{1},d_{2},\dots,d_{n})$.
Fixing $1\leq i\leq n$, for any integer $0\leq l\leq d_{i}$, we have a hyperbox $R^{\varepsilon_{i}=l}=(C^{\varepsilon_{i}=l},D^{\varepsilon_{i}=l})$
of size $(d_{1},...,d_{i-1},0,d_{i+1},...,d_{n})$, which consists of
 the chain complexes $C^{(\varepsilon_{1},...,\varepsilon_{n})}$ with $\varepsilon_{i}=l$.
The differentials $D^{\varepsilon_{i}=l}$ consist of all the
differentials $D_{\varepsilon^{0}}^{\varepsilon}$ of $(C,D)$ inside
$R^{\varepsilon_{i}=l}$.

\begin{rem}
In general, a hyperbox of chain complexes is not a chain complex.
But a hypercube is a chain complex considered as the total complex,
and it can also be regarded as a mapping cone in many ways.
\end{rem}

\subsubsection{Compression.}
From a hyperbox of chain complexes $H=((C^{\varepsilon})_{\varepsilon\in\mathbb{E}(\text{d})},(D^{\varepsilon})_{\varepsilon\in\mathbb{E}_{n}})$,
we can obtain a hypercube of chain complexes $\hat{H}=(\hat{C}^{\varepsilon},\hat{D}^{\varepsilon})_{\varepsilon\in\mathbb{E}_{n}}$,
thus generating a total complex $\text{Tot}(\hat{H})$. The process
of turning $H$ into $\hat{H}$ is called \emph{compression}.

\begin{example}[Compression of 1-dimensional hyperboxes]
Let $R$ be a hyperbox of dimension $1$,
see Example \ref{eg:dim-1 hyperbox}. The compression $\hat{R}$ is the
mapping cone of the composition of the maps $f_1,\dots,f_n$
\[
\xymatrix{C^{(0)}\ar[r]^{f_{n}\circ\cdots\circ f_{1}} & C^{(n)}.}
\]
\end{example}

\begin{example}[Compression of 2-dimensional hyperboxes]
Consider a rectangle of chain
complexes $R$ of size $(n,1)$:
\[
\xymatrix{C^{(0,0)}\ar[r]^{f_{1}}\ar[d]^{k_{0}}\ar[dr]^{H_{1}} & C^{(1,0)}\ar[r]^{f_{2}}\ar[d]^{k_{1}}\ar[dr]^{H_{2}} & C^{(2,0)}\ar[r]^{f_{3}}\ar[d]^{k_{2}}\ar[dr]^{H_{3}} & \cdots\ar[r]^{f_{n}}\ar[dr]^{H_{n}} & C^{(n,0)}\ar[d]^{k_{n}}\\
C^{(0,1)}\ar[r]^{g_{1}} & C^{(1,1)}\ar[r]^{g_{2}} & C^{(2,1)}\ar[r]^{g_{3}} & \cdots\ar[r]^{g_{n}} & C^{(n,1)}.
}
\]

As in \ref{eg:dim-2 hyperbox}, we can regard this rectangle as a 1-dimensional
hyperbox of mapping cones $\text{cone}(k_{i}),i=0,1,\ldots,n$.  The compression
of 1-dimensional hyperboxes induces the compression of rectangles
of chain complexes as follows
\[
\xyC{3pc}\xymatrix{C^{(0,0)}\ar[r]^{f_{n}\circ\cdots\circ f_{1}}\ar[d]_{k_{0}}\ar[dr]^{\hat{H}} & C^{(n,0)}\ar[d]^{k_{n}}\\
C^{(0,1)}\ar[r]_{g_{n}\circ\cdots\circ g_{1}} & C^{(n,1)},
}
\]
where
\[
\hat{H}=\sum_{i=1}^{n}f_{1}\circ\cdots\circ f_{i-1}\circ H_{i}\circ g_{i+1}\circ\cdots\circ g_{n}.
\]

Similarly, we can compress a rectangle of chain complexes of size
$(1,m)$. For a rectangle of chain complexes $R$ of size $(n,m)$,
we can decompose this rectangle into a union of $n$ vertical rectangles
of size $(1,m)$. We first compress all of these $n$ vertical rectangles,
and thus get a rectangle $R'$ of size $(n,1)$. Then we keep compressing
$R'$ and get a square of chain complexes $\hat{R}$. Alternatively,
we can also first compress every row, and then compress the
column. So every ordering of the coordinate axes gives a different way to
compress the rectangle.
\end{example}

For higher dimensional hyperbox, the compression is defined similarly
by induction, once we fix an order of the coordinate axes. Let us
describe this procedure using the language of composing \emph{chain maps of
hyperboxes}. One can check that it is the same as the compression
by means of the \emph{algebra of songs} introduced in \cite{link_surgery}.

Let ${^{0}{H}}=\bigl(({^{0}{C^{\varepsilon}}})_{\varepsilon\in\mathbb{E}({\bf d})},({^{0}{D^{\varepsilon}}})_{\varepsilon\in\mathbb{E}_{n}}\bigr),\ \ {^{1}{H}}=\bigl(({^{1}{C^{\varepsilon}}})_{\varepsilon\in\mathbb{E}({\bf d})},({^{1}{D^{\varepsilon}}})_{\varepsilon\in\mathbb{E}_{n}}\bigr)$
be two hyperboxes of chain complexes, having the same size ${\bf d}\in\mathbb{Z}_{\geq0}^{n}$.
Let $({\bf d},1)\in\mathbb{Z}_{\geq0}^{n+1}$ be the sequence obtained
from ${\bf d}$ by adding $1$ at the end.
\begin{defn}[Chain maps of hyperboxes]
A chain map $F:{^{0}{H}}\rightarrow{^{1}{H}}$
is a collection of linear maps
\[
F_{\varepsilon^{0}}^{\varepsilon}:{^{0}{C_{*}^{\varepsilon^{0}}}}\rightarrow{^{1}{C_{*+\Vert\varepsilon\Vert}^{\varepsilon^{0}+\varepsilon}}},
\]
satisfying
\[
\sum_{\varepsilon'\le\varepsilon}(D_{\varepsilon^{0}+\varepsilon'}^{\varepsilon-\varepsilon'}\circ F_{\varepsilon^{0}}^{\varepsilon'}+F_{\varepsilon^{0}+\varepsilon'}^{\varepsilon-\varepsilon'}\circ D_{\varepsilon^{0}}^{\varepsilon'})=0,
\]
for all $\varepsilon^{0}\in\mathbb{E}({\bf d}),\varepsilon\in\mathbb{E}_{n}$
such that $\varepsilon^{0}+\varepsilon\in\mathbb{E}({\bf d}).$
\end{defn}
In other words, a chain map between the hyperboxes ${^{0}{H}}$ and
${^{1}{H}}$ is an $(n+1)$-dimensional hyperbox of chain complexes,
of size $({\bf d},1)$, such that the sub-hyperbox corresponding to
$\varepsilon_{n+1}=0$ is ${^{0}{H}}$ and the one corresponding to
$\varepsilon_{n+1}=1$ is ${^{1}{H}}$. The maps $F$ are those maps
$D$ in the new hyperbox that increase $\varepsilon_{n+1}$ by 1.
Direct computations show the associativity $(F\circ G)\circ H=F\circ(G\circ H).$

For a $n$-dimensional hyperbox $H$ of size ${\bf d}=(d_{1},...,d_{n})$,
we fix an order of the axes, say, the increasing order $1,2,...,n$.
The hyperbox $H$ can be decomposed into
$d_{n}$ pieces of hyperboxes of size $(d_{1},...,d_{n-1},1)$, which
is a chain map $F_{i}:H^{\varepsilon_{n+1}=i-1}\rightarrow H^{\varepsilon_{n+1}=i}$.
Thus the composition $F_{d_{n}}\circ\cdots\circ F_{1}$ is a hyperbox
of size $(d_{1},...,d_{n-1},1)$, and we call it the compression along
the $n^{\text{th}}$-axis $\text{Comp}_{n}(H)$. If we keep doing
compressions for the other axes, then we get the compression $\hat{H}=\text{Comp}_{1}\circ\cdots\circ\text{Comp}_{n}(H).$

\subsubsection{Gluing of squares.}

In the link surgery formula, an algebraic operation occurs, which
we could call a twisted gluing of hypercubes. It consists in
 repeatedly  gluing mapping cones $A\xrightarrow{f}B,A\xrightarrow{g}B$
to get a new mapping cone $A\xrightarrow{f+g}B$. In this section we describe
this operation in detail for the case of two-component links. We call it
the \emph{twisted gluing of framed product squares}.

\begin{rem}
In Section 4.1.1, a hypercube of chain complexes requires a $\mathbb{Z}$-grading on it.
However, after gluing of hypercubes, it does not always admit a $\mathbb{Z}$-grading, but admits a $\mathbb{Z}/2\mathbb{Z}$-grading.
Now we only require a $\mathbb{Z}/2\mathbb{Z}$-grading on each chain complex sitting at
a vertex in the hypercube.
\end{rem}

\begin{defn}[Gluing of squares]
Suppose there are four squares of chain complexes
$R_{i,j}=(C_{i,j}^{\varepsilon},D_{i,j}^{\varepsilon}),i,j=0,1$ as
listed below,
\[
R_{i,j}:\xyR{1pc}\xyC{4pc}\xymatrix{C_{i,j}^{(0,0)}\ar[r]^{D_{i,j}^{(1,0)}}\ar[d]_{D_{i,j}^{(0,1)}}\ar[dr]|-{D_{i,j}^{(1,1)}} & C_{i,j}^{(1,0)}\ar[d]^{D_{i,j}^{(0,1)}}\\
C_{i,j}^{(0,1)}\ar[r]_{D_{i,j}^{(1,0)}} & C_{i,j}^{(1,1)}.
}
\]
 The squares $\{R_{i,j}\}_{i,j}$ are called \emph{gluable}, if $C_{0,0}^{\varepsilon}=C_{0,1}^{\varepsilon}=C_{1,0}^{\varepsilon}=C_{1,1}^{\varepsilon}$
for all $\varepsilon\in\mathbb{E}_{2}$ and $D_{i,0}^{(1,0)}=D_{i,1}^{(1,0)}=D_{i}^{(1,0)},D_{0,j}^{(0,1)}=D_{1,j}^{(0,1)}=D_{j}^{(0,1)}$
for all $i,j=0,1$. Then we can define $R=(C^{\varepsilon},D^{\varepsilon})$
to be the \emph{gluing} of $R_{i,j}$'s as below, where
we suppress the subscripts $i,j$ of $C_{i,j}^{\varepsilon}$ . One
can check that $R$ is a square of chain complexes.
\begin{equation*}
R:=\xyR{2pc}\xyC{5pc}\xymatrix{C^{(0,0)}\ar[r]^{D_{0}^{(1,0)}+D_{1}^{(1,0)}}\ar[d]_{D_{0}^{(0,1)}+D_{1}^{(0,1)}}\ar[dr]|-{\sum\limits_{i,j}D_{i,j}^{(1,1)}} & C^{(1,0)}\ar[d]^{D_{0}^{(0,1)}+D_{1}^{(0,1)}}\\
C^{(0,1)}\ar[r]_{D_{0}^{(1,0)}+D_{1}^{(1,0)}} & C^{(1,1)}.
}
\end{equation*}
\end{defn}

\begin{defn} [Framed product square]
A \emph{$\mathbb{Z}^{2}$-product square of chain complexes} is
a direct product of squares of chain complexes $$R=\prod_{\text{s}\in\mathbb{Z}^{2}}R_{\text{s}},$$
 where $R_{\text{s}}$ is a square of chain complexes for all $\text s\in \mathbb{Z}^2$. We call $\text s$ the \emph{coordinate} of any
 element $x\in R_{\text s}$. The function
 $$\mathcal{F}:R\to \mathbb{Z}^2,\quad \mathcal{F}(x)=\text s,\forall \text s \in R_{\text s},$$
 is called the \emph{framing} of $R$. In order to denote the framing, we write a \emph{framed product square} as a pair
 \[
(R,\mathcal{F})=\prod_{\text{s}\in\mathbb{Z}^{2}}(C_{\text{s}}^{\varepsilon},D_{\text s}^{\varepsilon}).
\]

We can shift the framing $\mathcal{F}$ by a set of vectors in $\mathbb{Z}^{2}$,
${\mathbf V}=\{v^{\varepsilon}\}_{\varepsilon\in\mathbb{E}_{2}}$, to get a new framing $\mathcal{F}^{\mathbf V},$
such that
$$\forall x\in C_{\text{s}}^{\varepsilon},\quad \mathcal{F}^{\mathbf V}(x)=\mathcal{F}(x)+v^{\varepsilon}.$$
We call the new framed product square $(R,\mathcal{F}^{\mathbf V})$  the \emph{shifted square of $R$ by $\mathbf V$},
and simply denote it by $R[{\bf V}]$. Thus, we can write $R[\mathbf V]=\prod_{\text s\in \mathbb{Z}^2}(\tilde{C}_{\text s}^{\varepsilon},\tilde{C}_{\text s}^{\varepsilon})$, where $\tilde{C}_{\text{s}+v^{\varepsilon}}^{\varepsilon}=C_{\text{s}}^{\varepsilon},\forall\varepsilon\in\mathbb{E}_{2},\forall\text{s}\in\mathbb{Z}^{2}.$
\end{defn}

\begin{defn}[Framed gluable]

Let $(R_{i,j},\mathcal{F}_{i,j})$ with $i,j=0,1$
be a set of framed product squares of chain complexes.
The set of four squares $\{R_{i,j}\}_{i,j}$ is called \emph{framed
gluable}, if $\{R_{i,j}\}_{i,j}$ are gluable as squares of chain complexes and all the framings
$\mathcal{F}_{i,j},\forall i,j$ are the same. Then, the result is called the \emph{framed gluing}
of $(R_{i,j},\mathcal{F}_{i,j})$'s.

\end{defn}

\begin{defn}[Twisted gluing]
Let $(R_{i,j},\mathcal{F}_{i,j})$ with $i,j=0,1$ be a set of
framed product squares of chain complexes. For any matrix
$\Lambda=(\Lambda_{1},\Lambda_{2})\in\mathbb{Z}^{2\times2}$,
let
\[
{\mathbf{V}}_{i,j}(\Lambda)=\{v^{\varepsilon}=\Lambda\cdot(i\varepsilon_{1},j\varepsilon_{2})^{t}\}_{\varepsilon=(\varepsilon_{1},\varepsilon_{2})\in\mathbb{E}_{2}}, \quad \forall i,j=0,1.
\]
Then there are four shifted squares $R_{i,j}[{\mathbf V}_{i,j}(\Lambda)]$, with $i,j=0,1$.
As long as these four shifted squares $\{R_{i,j}[{\mathbf V}_{i,j}(\Lambda)]\}_{i,j=0,1}$
are framed gluable, we define the \emph{$\Lambda$-twisted gluing
of $\{R_{i,j}\}_{i,j}$} to be the framed gluing of $\{R_{i,j}[{\mathbf V}_{i,j}(\Lambda)]\}_{i,j}$,
denoted by $R^{\Lambda}$. See Figure
\ref{twisted_gluing} for an example of twisted gluing.
\end{defn}

\begin{figure}

\centering

\includegraphics[scale=0.7]{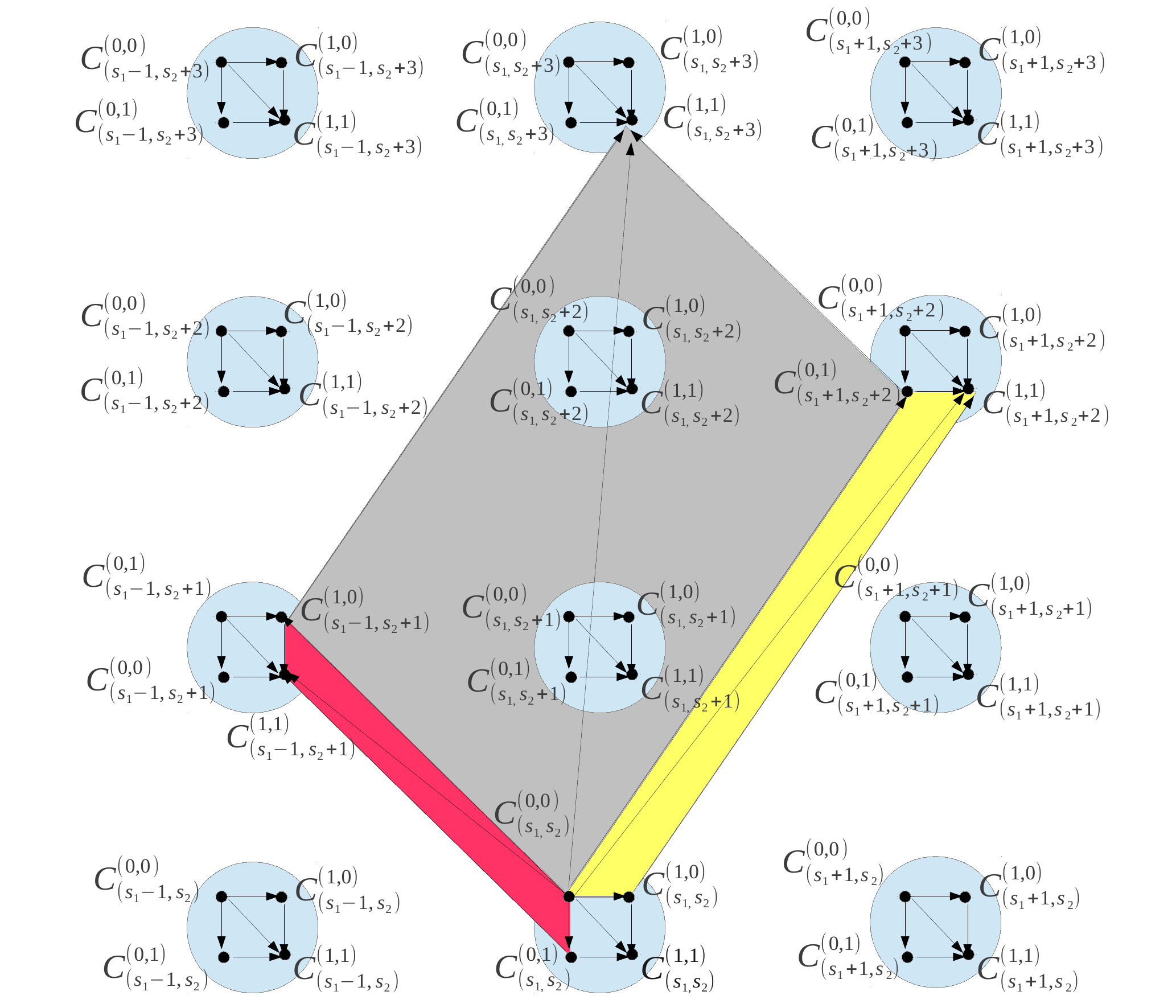}

\caption{\textbf{An example of twisted gluing}. This is an example of $\left(\protect\begin{array}{cc}
-1 & 1\protect\\
1 & 2
\protect\end{array}\right)$-twisted gluing of four squares $\{R_{i,j}=\prod_{\text{s}\in\mathbb{Z}^{2}}R_{\text{s},i,j}\}_{i,j=0,1}$,
where $R_{\text{s},i,j}=(C_{{\mathrm s},i,j}^{\varepsilon},D_{\text{s,}i,j}^{\varepsilon})$. Since $C_{{\mathrm s},i,j}^{\varepsilon}$ is identified with some $C_{{\mathrm s'},0,0}^{\varepsilon}$, we omit the subscripts $i,j$ in the picture.
Every shaded circle encloses a factor $R_{\text{s},0,0}$
of the $\mathbb{Z}^{2}$-product square $R_{0,0}$
with some $\text{s}\in\mathbb{Z}^{2}.$ The yellow parallelogram indicates
the $D$-maps of $R_{0,1}$, which is shifted by the vector
$\Lambda_{2}$; whereas the red parallelogram indicates the $D$-maps
of $R_{1,0}$, which is shifted by the vector $\Lambda_{1}.$
The gray parallelogram indicates all the maps of the square $R_{1,1}$,
which is shifted by using both $\Lambda_{1}$ and $\Lambda_{2}$.}

\label{twisted_gluing}
\end{figure}

\begin{example}[Twisted gluing in link surgery formula]
\label{(Twisted-gluing-in}
Suppose for any $(i,j)\in \mathbb{E}_2$,
$$R_{i,j}=\prod_{\text{s}\in\mathbb{Z}^{2}}R_{\text{s},i,j}$$
with $R_{\text{s},i,j}=(C_{\text{s},i,j}^{\varepsilon},D_{\text{s},i,j}^{\varepsilon})$
is a framed product square of chain complexes with the natural framing $\mathcal{F}_{i,j}(x)=\text{s},\forall x\in C_{\text{s},i,j}$.
Let $\overrightarrow{L}$ be a two-component link and $\mathrm{lk}$
be the linking number. Given any surgery framing matrix
\[
\Lambda=\left(\begin{array}{cc}
\lambda_{1} & \mathrm{lk}\\
\mathrm{lk} & \lambda_{2}
\end{array}\right),
\]
as long as the following identities hold for all $\mathrm{s}\in\mathbb{Z}^2$,
 \begin{equation}
\begin{aligned}
&C_{\text{s},0,0}^{(0,0)}=C_{\text{s},0,1}^{(0,0)}=C_{\text{s},1,0}^{(0,0)}=C_{\text{s},1,1}^{(0,0)},\\
&C_{\text{s},0,0}^{(1,0)}=C_{\text{s},0,1}^{(1,0)}=C_{\text{s}-\Lambda_{1},1,0}^{(1,0)}=C_{\text{s}-\Lambda_{1},1,1}^{(1,0)},\\
&C_{\text{s},0,0}^{(0,1)}=C_{\text{s}-\Lambda_{2},0,1}^{(0,1)}=C_{\text{s},1,0}^{(0,1)}=C_{\text{s}-\Lambda_{2},1,1}^{(0,1)},\\
&C_{\text{s},0,0}^{(1,1)}=C_{\text{s}-\Lambda_{2},0,1}^{(1,1)}=C_{\text{s}-\Lambda_{1},1,0}^{(1,1)}=C_{\text{s}-\Lambda_{1}-\Lambda_{2},1,1}^{(1,1)},
\end{aligned}
\label{eq:gluable}
\end{equation}
the shifted squares $\{R_{i,j}[{\mathbf V}_{i,j}(\Lambda)]\}_{i,j}$
are framed gluable. Thus we define the twisted gluing
of squares $R^{\Lambda}.$
\end{example}

\begin{rem}
The twisted glued square $R^{\Lambda}$ no longer
decomposes as a $\mathbb{Z}^{2}$ direct product. However, it decomposes
as a direct sum $\bigoplus_{\mathfrak{u}\in\mathbb{Z}^{2}/\Lambda}R^{\Lambda}(\mathfrak{u}),$
where $\Lambda$ is viewed as a lattice spanned by $\Lambda_{1},\Lambda_{2}.$
The equivalence classes $\mathbb{Z}^{2}/\Lambda$ correspond to the $\text{Spin}^{c}$
structures over the surgery manifold on a link in $S^{3}.$
\end{rem}

\subsection{Link surgery formula for a two-component link $\protect\overrightarrow{L}=\protect\overrightarrow{L_{1}}\cup\protect\overrightarrow{L_{2}}$.}

In order to denote the orientations of the sublinks, we use $\pm$ signs
to denote the positive and negative orientations, where the positive
orientation is the induced orientation from $\overrightarrow{L}$
and the negative orientation is the opposite orientation. Let
$L=+L_{1}\cup+L_{2}$, that is, $L$ has the same orientation as $\overrightarrow{L}$.

The link surgery formula is the total complex of a square of chain
complexes: the chain complexes at the vertices are the generalized Floer complexes
described in Section 2.2, and the maps in the square are defined by
means of \emph{complete systems of hyperboxes} (see section 4.3 for the definition).
From a complete system of hyperboxes of $L$, we get four sets of squares
of chain complexes $R_{\text{s},i,j}$, where $\text{s}\in\mathbb{H}(L),i,j\in\{0,1\}$:

\begin{equation}
R_{\text{s},0,0}:\xymatrix{\mathfrak{A}^{-}(\mathcal{H}^{L},\text{s})\ar[d]_{\Phi_{\text{s}}^{+L_{2}}}\ar[r]^{\Phi_{\text{s}}^{+L_{1}}}\ar[dr]|-{\Phi_{\text{s}}^{+L_{1}\cup+L_{2}}} & \mathfrak{A}^{-}(\mathcal{H}^{L_{2}},\psi^{+L_{1}}(\text{s}))\ar[d]^{\Phi_{\text{\ensuremath{\psi^{+L_{1}}}(\text{s})}}^{+L_{2}}}\\
\mathfrak{A}^{-}(\mathcal{H}^{L_{1}},\psi^{+L_{2}}(\text{s}))\ar[r]_{\Phi_{\text{\ensuremath{\psi^{+L_{2}}}(\text{s})}}^{+L_{1}}} & \mathfrak{A}^{-}(\mathcal{H}^{\emptyset},\psi^{+L_{1}\cup+L_{2}}(\text{s}));
}
\label{eq:surgery formula1}
\end{equation}

\begin{equation}
R_{\text{s},1,0}:\xymatrix{\mathfrak{A}^{-}(\mathcal{H}^{L},\text{s})\ar[d]_{\Phi_{\text{s}}^{+L_{2}}}\ar[r]^{\Phi_{\text{s}}^{-L_{1}}}\ar[dr]|-{\Phi_{\text{s}}^{-L_{1}\cup+L_{2}}} & \mathfrak{A}^{-}(\mathcal{H}^{L_{2}},\psi^{-L_{1}}(\text{s}))\ar[d]^{\Phi_{\text{\ensuremath{\psi^{-L_{1}}}(\text{s})}}^{+L_{2}}}\\
\mathfrak{A}^{-}(\mathcal{H}^{L_{1}},\psi^{+L_{2}}(\text{s}))\ar[r]_{\Phi_{\psi^{+L_{2}}(\text{s})}^{-L_{1}}} & \mathfrak{A}^{-}(\mathcal{H}^{\emptyset},\psi^{-L_{1}\cup+L_{2}}(\text{s}));
}
\end{equation}

\begin{equation}
R_{\text{s},0,1}:\xymatrix{\mathfrak{A}^{-}(\mathcal{H}^{L},\text{s})\ar[d]_{\Phi_{\text{s}}^{-L_{2}}}\ar[r]^{\Phi_{\text{s}}^{+L_{1}}}\ar[dr]|-{\Phi_{\text{s}}^{+L_{1}\cup-L_{2}}} & \mathfrak{A}^{-}(\mathcal{H}^{L_{2}},\psi^{+L_{1}}(\text{s}))\ar[d]^{\Phi_{\text{\ensuremath{\psi^{+L_{1}}}(\text{s})}}^{-L_{2}}}\\
\mathfrak{A}^{-}(\mathcal{H}^{L_{1}},\psi^{-L_{2}}(\text{s}))\ar[r]_{\Phi_{\text{\ensuremath{\psi^{-L_{2}}}(\text{s})}}^{+L_{1}}} & \mathfrak{A}^{-}(\mathcal{H}^{\emptyset},\psi^{+L_{1}\cup-L_{2}}(\text{s}));
}
\end{equation}

\begin{equation}
R_{\text{s},1,1}:\xymatrix{\mathfrak{A}^{-}(\mathcal{H}^{L},\text{s})\ar[d]_{\Phi_{\text{s}}^{-L_{2}}}\ar[r]^{\Phi_{\text{s}}^{-L_{1}}}\ar[dr]|-{\Phi_{\text{s}}^{-L_{1}\cup-L_{2}}} & \mathfrak{A}^{-}(\mathcal{H}^{L_{2}},\psi^{-L_{1}}(\text{s}))\ar[d]^{\Phi_{\psi^{-L_{1}}(\text{s})}^{-L_{2}}}\\
\mathfrak{A}^{-}(\mathcal{H}^{L_{1}},\psi^{-L_{2}}(\text{s}))\ar[r]_{\Phi_{\text{\ensuremath{\psi^{-L_{2}}}(\text{s})}}^{-L_{1}}} & \mathfrak{A}^{-}(\mathcal{H}^{\emptyset},\psi^{-L_{1}\cup-L_{2}}(\text{s})),
}
\label{eq:surgery formual4}
\end{equation}
where $\psi^{\overrightarrow{M}}$ is defined in Equation \eqref{eq:psi}.
Thus, we have four framed product squares with the natural framings
\[
R_{i,j}=\prod_{\text{s}\in\mathbb{H}(L)}R_{\text{s},i,j}.
\]

\begin{defn}[Link surgery formula]
For any surgery framing matrix $\Lambda$, the shifted squares $\{R_{i,j}[{\bf V}_{i,j}(\Lambda)]\}_{i,j}$
are framed gluable according to Equations \eqref{eq:gluable}. The link surgery
formula for the framed link $(L,\Lambda)$ is the total complex
of the $\Lambda$-twisted gluing of $\{R_{i,j}\}_{i,j}$ as follows,

\[
(\mathcal{C}^{-}(\mathcal{H},\Lambda),\mathcal{D}^{-}):=\xyR{1pc}\xyC{2pc}\xymatrix{\underset{\text{s}\in\mathbb{H}(L)}{\prod}\mathfrak{A}^{-}(\mathcal{H}^{L},\text{s})\ar[dd]_{\Phi^{+L_{2}}+\Phi^{-L_{2}}}\ar[rr]^{\Phi^{+L_{1}}+\Phi^{-L_{1}}}\ar[ddrr]|-{\begin{array}{c}
\Phi^{+L_{1}\cup-L_{2}}+\Phi^{-L_{1}\cup-L_{2}}\\
+\Phi{}^{+L_{1}\cup+L_{2}}+\Phi^{-L_{1}\cup+L_{2}}
\end{array}} &  & \underset{\text{s}\in\mathbb{H}(L)}{\prod}\mathfrak{A}^{-}(\mathcal{H}^{L_{2}},\psi^{+L_{1}}(\text{s}))\ar[dd]^{\Phi^{+L_{2}}+\Phi^{-L_{2}}}\\
\\
\underset{\text{s}\in\mathbb{H}(L)}{\prod}\mathfrak{A}^{-}(\mathcal{H}^{L_{1}},\psi^{+L_{2}}(\text{\text{s}}))\ar[rr]_{\Phi^{+L_{1}}+\Phi^{-L_{1}}\ } &  & \underset{\text{s}\in\mathbb{H}(L)}{\prod}\mathfrak{A}^{-}(\mathcal{H}^{\emptyset},\psi^{+L_{1}\cup+L_{2}}(\text{s})),
}
\]
where $\Phi^{\circ}=\prod_{\text{s}\in\mathbb{H}(L)}\Phi_{\text{s}}^{\circ}$ with
$\circ=\pm L_{1},\pm L_2,\pm L_{1}\cup\pm L_{2}.$

The map
\[
\Phi_{\text{s}}^{\overrightarrow{M}}:\mathfrak{A}^{-}(\mathcal{H}^{L},\text{s})\rightarrow\mathfrak{A}^{-}(\mathcal{H}^{L-M},\psi^{\overrightarrow{M}}(\text{s}))
\]
is defined by
\begin{equation}
\Phi_{\text{s}}^{\overrightarrow{M}}=D_{p^{\overrightarrow{M}}(\text{s})}^{\overrightarrow{M}}\circ\mathcal{I}_{\text{s}}^{\overrightarrow{M}}.\label{eq:phi}
\end{equation}
We will spell out the constructions of $D_{p^{\overrightarrow{M}}(\text{s})}^{\overrightarrow{M}}$
and $\mathcal{I}_{\text{s}}^{\overrightarrow{M}}$ in the next sections
by using \emph{primitive systems of hyperboxes}.
\end{defn}

For the $\Lambda$-twisted gluing of squares, there is a direct sum splitting of the complex
\[
\mathcal{C}^{-}(\mathcal{H},\Lambda)=\bigoplus_{\mathfrak{u}\in\mathbb{H}(L)/H(L,\Lambda)}\mathcal{C}^{-}(\mathcal{H},\Lambda,\mathfrak{u}),
\]
where we identify $\mathbb{H}(L)/H(L,\Lambda)$ with $\text{Spin}^{c}(S_{\Lambda}^{3}(L))$.

\begin{thm}[Manolescu-Ozsv\'{a}th Link Surgery Theorem, Theorem 7.7 of \cite{link_surgery} for two-component links]\label{(Manolescu-Ozsvth-Link-Surgery} Fix
a primitive system of hyperboxes $\mathcal{H}$ for an oriented two-component
link $\overrightarrow{L}$ in $S^{3}$, and fix a framing $\Lambda$
of $L$. Then for any $\mathfrak{u}\in\text{Spin}^{c}(S_{\Lambda}^{3}(L))\cong\mathbb{H}(L)/H(L,\Lambda)$,
there is an isomorphism of $\mathbb{F}[[U]]$-modules
\[
H_{*}(\mathcal{C}^{-}(\mathcal{H},\Lambda,\mathfrak{u}),\mathcal{D}^{-})\cong{\bf HF}_{*}^{-}(S_{\Lambda}^{3}(L),\mathfrak{u}),
\]
where $\mathbb{F}[[U]]=\mathbb{F}[[U_{1},U_{2}]]/(U_{1}-U_{2})$.
\end{thm}

Here, we let $\bf{HF}^-$ denote the completion of $HF^-$ with respect to the maximal
ideal $(U)$ in the ring $\mathbb{F}[U]$. Since completion is an exact functor, $\bf{HF}^-$
can be regarded as the homology of the complex  ${\bf{CF}}^{-}=CF^{-}\otimes_{\mathbb{F}[ U ]}\mathbb{F}[[ U ]]$,
where $\mathbb{F}[[U]]$ is the completion of $\mathbb{F}[U]$. When $\text s$ is a torsion $\mathrm{Spin}^c$
structure of a 3-manifold $M$, if
\[
{\bf{HF}}^-(M,\text s)=\mathbb{F}[[U]]\oplus T
\]
with $T$ a torsion $\mathbb{F}[[U]]$-module, then
\[
HF^-(M,\text s)=\mathbb{F}[U]\oplus T.
\]
For more details, see Section 2 in \cite{link_surgery}.

\begin{rem}
The link surgery theorem states that all the $U_{i}$-actions are
the same in the homology of the surgery complex.
\end{rem}

\begin{rem}
Although all the squares in Equations \eqref{eq:surgery formula1} to \eqref{eq:surgery formual4} posses
 $\mathbb{Z}$-gradings, the surgery complex $\mathcal{C}^{-}(\mathcal{H},\Lambda)$ does not always have
a $\mathbb{Z}$-grading after the twisted gluing. In \cite{link_surgery}, an absolute grading was also given
 in Section 7.4, which is the same as the absolute grading on Heegaard Floer homology of the surgery manifold.
\end{rem}

\subsection{Inclusion maps and destabilization maps.}

\subsubsection{Inclusion maps.}

In the link surgery formula, we need a set of chain maps $\mathcal{I}_{\text{s}}^{\overrightarrow{M}}$
in (\ref{eq:phi}) which are called \emph{inclusion maps}. Here, we define the inclusion maps
for all links with arbitrary number of components. In the knot case, the inclusion maps correspond to
the maps $h_{s}$ and $v_{s}$ from \cite{knot_surgery}.

\begin{defn}
Let $\overrightarrow{M}$ be an oriented sublink of $\overrightarrow{L}$.
Define
\begin{align*}
I_{+}(\overrightarrow{L},\overrightarrow{M}) & =\{i:\overrightarrow{L}\text{ and }\overrightarrow{M}\text{ share the same orientation on }L_{i}\};\\
I_{-}(\overrightarrow{L},\overrightarrow{M}) & =\{i:\overrightarrow{L}\text{ and }\overrightarrow{M}\text{ have different orientations on }L_{i}\}.
\end{align*}
A projection map $p^{\overrightarrow{M}}:\overline{\mathbb{H}}(L)\rightarrow\overline{\mathbb{H}}(L)$
is defined component-wisely as follows:
\[
p_{i}^{\overrightarrow{M}}(s)=\begin{cases}
+\infty & \text{if}\ i\in I_{+}(\overrightarrow{L},\overrightarrow{M})\\
-\infty & {\bf \text{if}}\ i\in I_{-}(\overrightarrow{L},\overrightarrow{M})\\
s & \text{otherwise}
\end{cases}.
\]
\end{defn}

\begin{defn}[Inclusion maps]
Suppose $\overrightarrow{M}\subset\overrightarrow{L}$
is an oriented sublink, and $\text{s}=(s_{1},s_{2})\in\mathbb{H}(L)$
satisfies $s_{i}\neq\mp\infty$ for those $i\in I_{\pm}(\overrightarrow{L},\overrightarrow{M})$.
Let $\mathcal{H}$ be a Heegaard diagram of $L$. The  \emph{inclusion map}
$\mathcal{I}_{\text{s}}^{\overrightarrow{M}}:\mathfrak{A}^{-}(\mathcal{H}^{L},\text{s})\rightarrow\mathfrak{A}^{-}(\mathcal{H}^{L},p^{\overrightarrow{M}}(\text{s}))$
is defined by the formula
\[
\mathcal{I}_{\text{s}}^{\overrightarrow{M}}(x)={\displaystyle \underset{i\in I_{+}(\overrightarrow{L},\overrightarrow{M})}{{\displaystyle \prod}}U_{\tau_{i}}^{\max(A_{i}(x)-s_{i},0)}\underset{i\in I_{-}(\overrightarrow{L},\overrightarrow{M})}{{\displaystyle {\displaystyle \prod}}}U_{\tau_{i}}^{\max(s_{i}-A_{i}(x),0)}}x.
\]
One can verify this is a chain map.
\end{defn}

\begin{example}
Suppose $\overrightarrow{L}$ is a two-component link and $\mathcal H$ is a basic Heegaard diagram
 of $L$. We have the following inclusion maps
\begin{itemize}
\item
$\mathcal{I}_{s_{1},s_{2}}^{+L_{1}}:A_{s_{1},s_{2}}^{-}\rightarrow A_{+\infty,s_{2}}^{-},\quad
\mathcal{I}_{s_{1},s_{2}}^{-L_{1}}:A_{s_{1},s_{2}}^{-}\rightarrow A_{-\infty,s_{2}}^{-}.$

\item

$\mathcal{I}_{s_{1},s_{2}}^{+L_{2}}:A_{s_{1},s_{2}}^{-}\rightarrow A_{s_1,+\infty}^{-},\quad
\mathcal{I}_{s_{1},s_{2}}^{-L_{2}}:A_{s_{1},s_{2}}^{-}\rightarrow A_{s_1,-\infty}^{-}.$

\item
$\mathcal{I}_{s_{1},s_{2}}^{+L_{1}\cup+L_{2}}:A_{s_{1},s_{2}}^{-}\rightarrow A_{+\infty,+\infty}^{-},\quad
 \mathcal{I}_{s_{1},s_{2}}^{-L_{1}\cup+L_{2}}:A_{s_{1},s_{2}}^{-}\rightarrow A_{-\infty,+\infty}^{-},\quad\\
 \mathcal{I}_{s_{1},s_{2}}^{+L_{1}\cup-L_{2}}:A_{s_{1},s_{2}}^{-}\rightarrow A_{+\infty,-\infty}^{-},\quad
 \mathcal{I}_{s_{1},s_{2}}^{-L_{1}\cup-L_{2}}:A_{s_{1},s_{2}}^{-}\rightarrow A_{-\infty,-\infty}^{-}.$
\end{itemize}
\end{example}

\subsubsection{Destabilization maps.}

Let $\overrightarrow{L}=\overrightarrow{L_{1}}\cup\overrightarrow{L_{2}}.$
 We set
\[
J(+L_{i})=\{(s_{1},s_{2})\in\overline{\mathbb{H}}(L)|s_{i}=+\infty\},\quad J(-L_{i})=\{(s_{1},s_{2})\in\overline{\mathbb{H}}(L)|s_{i}=-\infty\}.
\]

Now let $\overrightarrow{M}\subset L$ be an oriented sublink, and let $\{\overrightarrow{M_{i}}\}_{i}$
be all the oriented components of $\overrightarrow{M}$. Define
\[
J(\overrightarrow{M})=\underset{i}{\bigcap}J(\overrightarrow{M_{i}}).
\]
For  $s\in J(\overrightarrow{M})$,
there is a \emph{destabilization map}
\[
D_{\text{s}}^{\overrightarrow{M}}:\mathfrak{A}^{-}(\mathcal{H}^{L},\text{s})\rightarrow\mathfrak{A}^{-}(\mathcal{H}^{L-M},\psi^{\overrightarrow{M}}(\text{s})),
\]
which gives rise to the map $D_{p^{\overrightarrow{M}}(\text{s})}^{\overrightarrow{M}}$
in (\ref{eq:phi}). Note that $p^{\overrightarrow{M}}(s)\in J(\overrightarrow{M})$ for
any $\text{s}\in\mathbb{H}(L)$. In the knot case, the destabilization map corresponds
to the map identifying $C\{i>0\}$ and $C\{j>0\}$. We will give
the definition in the next section.

\begin{example}
Let $\text{s}=(s_{1},s_{2})\in \mathbb{H}(L)$ and $\overrightarrow{M}=\pm L_{1},$
then $p^{\pm L_{1}}(\text{s})=(\pm\infty,s_{2}).$ The destabilization map
\[
D_{\pm\infty,s_{2}}^{\pm L_{1}}:\mathfrak{A}^{-}(\mathcal{H}^{L},\pm\infty,s_{2})\rightarrow\mathfrak{A}^{-}(\mathcal{H}^{L_{2}},s_{2}-\frac{\mathrm{lk}(+L_{1},\pm L_{2})}{2})
\]
is a chain homotopy equivalence.

If we consider sublinks $\overrightarrow{M}=\pm L_{1}\cup\pm L_{2}$,
then we will get destabilization maps from $\mathfrak{A}^{-}(\mathcal{H}^{L},\pm\infty,\pm\infty)$
to $\mathfrak{A}^{-}(\mathcal{H}^{\emptyset},0)$, namely,
\begin{eqnarray*}
D_{+\infty,+\infty}^{+L_{1}\cup+L_{2}}: & \mathfrak{A}^{-}(\mathcal{H}^{L},+\infty,+\infty)\rightarrow\mathfrak{A}^{-}(\mathcal{H}^{\emptyset},0),\\
D_{-\infty,+\infty}^{-L_{1}\cup+L_{2}}: & \mathfrak{A}^{-}(\mathcal{H}^{L},-\infty,+\infty)\rightarrow\mathfrak{A}^{-}(\mathcal{H}^{\emptyset},0),\\
D_{+\infty,-\infty}^{+L_{1}\cup-L_{2}}: & \mathfrak{A}^{-}(\mathcal{H}^{L},+\infty,-\infty)\rightarrow\mathfrak{A}^{-}(\mathcal{H}^{\emptyset},0),\\
D_{-\infty,-\infty}^{-L_{1}\cup-L_{2}}: & \mathfrak{A}^{-}(\mathcal{H}^{L},-\infty,-\infty)\rightarrow\mathfrak{A}^{-}(\mathcal{H}^{\emptyset},0).
\end{eqnarray*}
\end{example}

\subsubsection{Primitive system of hyperboxes.}

In \cite{link_surgery}, \emph{complete system of hyperboxes} is defined
in order to define the destabilization maps.

\begin{defn}[Complete pre-system of hyperboxes]
A \emph{complete pre-system of hyperboxes}
$\mathcal{H}$ representing the link $\overrightarrow{L}$ consists
of a collection of hyperboxes of Heegaard diagrams, subject to certain
compatibility conditions as follows. For each pair of subsets $M\subseteq L'\subseteq L$,
and each orientation $\overrightarrow{M}$ on $M$, the complete
pre-system assigns a hyperbox $\mathcal{H}^{\overrightarrow{L'},\overrightarrow{M}}$
for the pair $(\overrightarrow{L'},\overrightarrow{M})$, where $\overrightarrow{L'}$
has the induced orientation from $\overrightarrow{L}$. Moreover,
the hyperbox $\mathcal{H}^{\overrightarrow{L'},\overrightarrow{M}}$
is required to be compatible with both $\mathcal{H}^{\overrightarrow{L'},\overrightarrow{M'}}$
and $\mathcal{H}^{\overrightarrow{L'}-\overrightarrow{M'},\overrightarrow{M}-\overrightarrow{M'}}$.
\end{defn}

In the above definition, there is some compatibility condition we have not spelled out. A
 \emph{complete system of hyperboxes} is a complete pre-system with
some additional conditions regarding the surface isotopies connecting
those hyperboxes.  In a complete system of hyperboxes, every
hyperbox of Heegaard diagrams induces a hyperbox of generalized Floer complexes. Instead of
explaining these compatibility conditions, we give a special
complete system of hyperboxes for two-component links satisfying these conditions, which
illustrates the main idea. They are called \emph{primitive system of hyperboxes}.

When the sublink $\overrightarrow{L'}$ has the induced orientation
from $\overrightarrow{L}$, we simply denote it by $L'$. Thus,
we use notation $\mathcal{H}^{L',\overrightarrow{M}}=\mathcal{H}^{\overrightarrow{L'},\overrightarrow{M}}$.
In a complete pre-system $\mathcal{H}$, we have four zero dimensional hyperboxes of Heegaard diagrams, $\mathcal{H}^{L,\emptyset},\mathcal{H}^{L_{1},\emptyset},\mathcal{H}^{L_{2},\emptyset},\mathcal{H}^{\emptyset,\emptyset}$,
where $\mathcal{H}^{L,\emptyset}$ is a Heegaard diagram of $L$,
$\mathcal{H}^{L_{i},\emptyset}$ is a Heegaard diagram of $L_{i}$,
and $\mathcal{H}^{\emptyset,\emptyset}$ is a Heegaard diagram of
$S^{3}$. We denote the four Heegaard diagrams simply by $\mathcal{H}^{L},\mathcal{H}^{L_{1}},\mathcal{H}^{L_{2}},\mathcal{H}^{\emptyset}.$

Given a basic Heegaard diagram $\mathcal{H}^{L}=(\Sigma,\boldsymbol{\alpha},\boldsymbol{\beta},\{w_{1},w_{2}\},\{z_{1},z_{2}\})$
of $L$,
from Equation (\ref{eq:A_s}), we see $\mathfrak{A^{-}}(\mathcal{H}^{L},(+\infty,s_{2}))$
is counting $n_{w_{1}}(\phi)$ without using $z_{1}$, thus as the
same as deleting $z_{1}$. Moreover $\mathcal{H}^{L_{2}}=(\Sigma,\boldsymbol{\alpha},\boldsymbol{\beta},\{w_{1},w_{2}\},\{z_{2}\})$
is a Heegaard diagram of $L_{2}$ with one free basepoint
$w_{1}$. We call this diagram the \emph{reduction} of $\mathcal{H}^{L}$
at $+L_{1}$, denoted by $r_{+L_{1}}(\mathcal{H}^{L})$; see \cite{link_surgery}
Definition 4.17. Hence, we have an identification between $\mathfrak{A}^{-}(\mathcal{H}^{L},(+\infty,s_{2}))$
and $\mathfrak{A}^{-}(r_{+L_{1}}(\mathcal{H}^{L}),s_{2}-\frac{\mathrm{lk}(+L_{1},+L_{2})}{2})$.
Similarly, we define $r_{-L_{1}}(\mathcal{H}^{L})$ to be the diagram
obtained from $\mathcal{H}^{L}$ by deleting $w_{1}$ and relabeling
$z_{1}$ as $w_{1}$. We have an identification between $\mathfrak{A}^{-}(\mathcal{H}^{L},(-\infty,s_{2}))$
and $\mathfrak{A}^{-}(r_{-L_{1}}(\mathcal{H}^{L}),s_{2}-\frac{\mathrm{lk}(-L_{1},+L_{2})}{2})$,
since $\mathfrak{A}^{-}(\mathcal{H}^{L},(-\infty,s_{2}))$ uses basepoints
$\{z_{1},w_{2}\}\subset\mathcal{H}^{L}$.

Moreover, the diagrams $r_{-L_{1}}(\mathcal{H}^{L})$ and $r_{+L_{1}}(\mathcal{H}^{L})$
are related by Heegaard moves, for they represent the same knot
$L_{2}$. By definition, there is an arc $c$ in $\Sigma-\boldsymbol{\alpha}$
connecting $w_{1}$ and $z_{1}$, so we can move $z_{1}$ along $c$
to $w_{1}$, by a sequence of Heegaard moves. Moving a basepoint to
cross some $\beta$-curve can be done by a sequence of handleslides
and isotopies of $\beta$-curves, stabilizations, and destabilization
followed by a surface isotopy. However, if we need stabilizations/destabilizations,
we can modify the original Heegaard diagram $\mathcal{H}^{L}$ by
these stabilizations in the beginning. Thus, we can always get a diagram
$\tilde{\mathcal{H}}^{L}$, such that there is a sequence of Heegaard
moves only of handleslides and isotopies of $\beta$-curves together
with some surface isotopy from $r_{-L_{1}}(\mathcal{\tilde{H}}^{L})$
to $r_{+L_{1}}(\tilde{\mathcal{H}}^{L})$. In sum, there is an surface
isotopy $h:\Sigma\rightarrow\Sigma$ supported in a small neighborhood
of $c$ (so $h$ fixes other basepoints and all the $\alpha$-curves),
such that $h(w_{1})=z_{1}$ and $h(r_{+L_{1}}(\tilde{\mathcal{H}}^{L}))$
is strongly equivalent to $r_{-L_{1}}(\tilde{\mathcal{H}}^{L})$ via
handleslides and isotopies of $\beta$-curves.

\begin{defn}[Primitive Heegaard diagrams]
For any basic Heegaard diagram $\mathcal{H}$
of an oriented link $\overrightarrow{L}=\overrightarrow{L_{1}}\cup\overrightarrow{L_{2}}$,
there are surface isotopies $h_{i}^{\mathcal{H}}:\Sigma\rightarrow\Sigma$
supported in a small neighborhood of the arc $c_{i}$ connecting $w_{i}$
and $z_{i}$ in $\Sigma-\boldsymbol{\alpha}$, such that $h_{i}^{\mathcal{H}}(w_{i})=z_{i}$
and $h_{i}^{\mathcal{H}}$ preserves the $\alpha$-curves and the
other basepoints. They are unique up to isotopy. The basic Heegaard
diagram $\mathcal{H}$ is called \emph{primitive}, if it is admissible and $r_{-L_{i}}(\mathcal{H})$
is strongly equivalent to $h_{i}^{\mathcal{H}}(r_{+L_{i}}(\mathcal{H}))$
for both $i=1,2$.
\end{defn}
From the above discussion, we can get the following lemma.

\begin{lem}
\label{lem:primitive diagram}Let $L$ be an oriented two-component
link, and let $\mathcal{H}$ be a basic admissible Heegaard diagram of $L$. Then there
is an index one/two stabilization turning $\mathcal{H}$ into a primitive Heegaard
diagram  $\mathcal{\tilde{H}}.$
\end{lem}

Fixing a primitive Heegaard diagram $\mathcal{H}^{L}$ for an oriented two-component
 link $L$, we can get two sequences of strongly equivalent Heegaard diagrams
$\bar{\mathcal{H}}^{L,-L_{i}}$:
\begin{align}\label{eq:H^L,-L_1}
\bar{\mathcal{H}}^{L,-L_{1}}:r_{-L_{1}}(\mathcal{H}^{L})=\bar{\mathcal{H}}^{L,-L_{1}}(\emptyset)\rightarrow  \bar{\mathcal{H}}^{L,-L_{1}}(L_{1})=h_{1}^{\mathcal{H}}(r_{+L_{1}}(\mathcal{H}^{L})),\\
\bar{\mathcal{H}}^{L,-L_{2}}:r_{-L_{2}}(\mathcal{H}^{L})=\bar{\mathcal{H}}^{L,-L_{2}}(\emptyset)\rightarrow  \bar{\mathcal{H}}^{L,-L_{2}}(L_{2})=h_{2}^{\mathcal{H}}(r_{+L_{2}}(\mathcal{H}^{L})).
\end{align}

These induce another two sequences of strongly equivalent Heegaard
diagrams $\bar{\mathcal{H}}^{L_{i},-L_{i}}$:
\begin{align}
\bar{\mathcal{H}}^{L_{1},-L_{1}}:r_{-L_{1}}\big(r_{+L_{2}}(\mathcal{H}^{L})\big)=\bar{\mathcal{H}}^{L_{1},-L_{1}}(\emptyset)\rightarrow \bar{\mathcal{H}}^{L_{1},-L_{1}}(L_{1})=h_{1}^{\mathcal{H}}\big(r_{+L_{1}}(r_{+L_{2}}(\mathcal{H}^{L}))\big),
\\
\bar{\mathcal{H}}^{L_{2},-L_{2}}:r_{-L_{2}}\big(r_{+L_{1}}(\mathcal{H}^{L})\big)=\bar{\mathcal{H}}^{L_{2},-L_{2}}(\emptyset)\rightarrow  \bar{\mathcal{H}}^{L_{2},-L_{2}}(L_{2})=h_{2}^{\mathcal{H}}(r_{+L_{2}}\big(r_{+L_{1}}(\mathcal{H}^{L}))\big),
\label{eq:H^L_1,-L_1}
\end{align}
together with a square of strongly equivalent Heegaard diagrams $\bar{\mathcal{H}}^{L,-L_{1}\cup-L_{2}}:$

\begin{equation}
\label{eq:H^L,-L}
\xymatrix{r_{-L_{1}\cup-L_{2}}(\mathcal{H})= \bar{\mathcal{H}}^{L_{1}\cup L_{2},-L_{1}\cup-L_{2}}(\emptyset)\ar[r]\ar[d]\ar[rd] &  \bar{\mathcal{H}}^{L_{1}\cup L_{2},+L_{1}\cup-L_{2}}(L_{1})=h_{1}^{\mathcal{H}}(r_{+L_{1}\cup-L_{2}}(\mathcal{H})) \ar[d]\\
h_{2}^{\mathcal{H}}(r_{-L_{1}\cup+L_{2}}(\mathcal{H}))=\bar{\mathcal{H}}^{L_{1}\cup L_{2},-L_{1}\cup-L_{2}}(L_{2})  \ar[r] & \bar{\mathcal{H}}^{L_{1}\cup L_{2},-L_{1}\cup-L_{2}}(L) =h_{1}^{\mathcal{H}}\circ h_{2}^{\mathcal{H}}(r_{+L_{1}\cup+L_{2}}(\mathcal{H})).
}
\end{equation}

These almost produce a complete system of hyperbox $\bar{\mathcal{H}}$
except for the admissibility of $\bar{\mathcal{H}}$. We call this
system a \emph{primitive almost complete system of hyperbox}.

\begin{defn}[Primitive almost complete system of hyperbox]
Given a primitive Heegaard diagram $\mathcal{H}^{L}=(\Sigma,\boldsymbol{\alpha},\boldsymbol{\beta},\{w_{1},w_{2}\},\{z_{1},z_{2}\})$
of $\overrightarrow{L}=\overrightarrow{L_{1}}\cup\overrightarrow{L_{2}}$,
there exists a \emph{primitive almost complete system of hyperbox}
$\bar{\mathcal{H}}$ associated to $\mathcal{H}^{L}$ consisting of

\begin{itemize}
\item four 0-dimensional hyperboxes of Heegaard diagrams:
\[
\bar{\mathcal{H}}^{L}=\mathcal{H},\ \mathcal{\bar{H}}^{L_{1}}=r_{+L_{2}}(\mathcal{H}),\ \mathcal{\bar{H}}^{L_{2}}=r_{+L_{1}}(\mathcal{H}),\ \mathcal{\bar{H}}^{\emptyset}=r_{+L_{1}\cup+L_{2}}(\mathcal{H});
\]

\item eight 1-dimensional hyperboxes of Heegaard diagrams:
\[
\mathcal{\bar{H}}^{L,\pm L_{i}},\quad \bar{\mathcal{H}}^{L_{i},\pm L_{i}},\forall i=1,2,
\]
where $\mathcal{\bar{\mathcal{H}}}^{L,+L_{i}},\bar{\mathcal{H}}^{L_{i},+L_{i}}$
are trivial hyperboxes, i.e. just a Heegaard diagram, and $\bar{\mathcal{H}}^{L,-L_{i}},$$\bar{\mathcal{H}}^{L_{i},-L_{i}}$
are described in Equations from \eqref{eq:H^L,-L_1} to \eqref{eq:H^L_1,-L_1};

\item four 2-dimensional hyperboxes of Heegaard diagrams: one trivial
hyperbox $\bar{\mathcal{H}}^{L,+L_{1}\cup+L_{2}}$, two degenerate
hyperboxes:
\[
\bar{\mathcal{H}}^{L,+L_{1}\cup-L_{2}}=\bar{\mathcal{H}}^{L_{2},-L_{2}},\ \ \bar{\mathcal{H}}^{L,-L_{1}\cup+L_{2}}=\bar{\mathcal{H}}^{L_{1},-L_{1}},
\]
and a square of strongly equivalent Heegaard diagrams $\mathcal{H}^{L,-L_{1}\cup-L_{2}}$,
which is described in Equation (\ref{eq:H^L,-L}). \end{itemize}\end{defn}

\begin{defn}[Primitive system of hyperboxes]
Given a primitive diagram $\mathcal{H}^{L}$ and the induced primitive
almost complete systems of hyperboxes $\mathcal{\bar{H}}$, if the
admissibility of $\mathcal{\bar{H}}$ is not satisfied, we can enlarge
the hyperbox in $\mathcal{\bar{H}}$  to  achieve admissibility, thus
getting a complete system of hyperboxes. We call the result a \emph{primitive
system of hyperboxes}  $\mathcal{H}$.
\end{defn}

Indeed, if $\bar{\mathcal{H}}^{L,-L_{1}}$
is not admissible,  i.e. $(\Sigma,\boldsymbol{\alpha},\boldsymbol{\beta},\boldsymbol{\beta}',{\bf w},{\bf z})$
is not admissible,  then we can insert an isotopy of $\boldsymbol{\beta}$, namely  $\boldsymbol{\beta}''$,  such that both $(\Sigma,\boldsymbol{\alpha},\boldsymbol{\beta},\boldsymbol{\beta}'',{\bf w},{\bf z})$ and $(\Sigma,\boldsymbol{\alpha},\boldsymbol{\beta}'',\boldsymbol{\beta}',{\bf w},{\bf z})$ are admissible.
Suppose $\{D_1,\dots,D_{m}\}$ is a basis of the $\mathbb{Q}$-vector space of the periodic domains in $(\Sigma,\boldsymbol{\alpha},\boldsymbol{\beta},\boldsymbol{\beta}',{\bf w},{\bf z})$
with only positive multiplicities . Let $D_1^{c}$ be the union of
all the regions which are not in $D_1$. Then $D_1^{c}\neq\emptyset$, since $n_{\bf w}D_1=0$.   As  $(\Sigma,\boldsymbol{\alpha},\boldsymbol{\beta},{\bf w},{\bf z})$ and $(\Sigma,\boldsymbol{\alpha},\boldsymbol{\beta}',{\bf w},{\bf z})$
 are both admissible, the boundary of $D_1$ must contain a $\beta$-curve and a ${\beta}'$-curve. Thus there is a $\beta$-arc
$b$ and a ${\beta}'$-arc $b'$ on $D_1\cap D_1^{c}$. So we can find a path
$\gamma$ in $D_1$ connecting $b$ to $b'$,
and then do a finger move
of the $\beta$-curve containing $b$ along $\gamma$ to
get negative multiplicities for $D_1$ (see \cite{Sakar-Wang} for the definition of finger move). Similarly we deal with the other $D_i$'s. Finally, the new $\boldsymbol{\beta}$ in the above process is chosen to be $\boldsymbol{\beta}''$. Similar arguments work for the case
of the square $\bar{\mathcal{H}}^{L,-L_{1}\cup-L_{2}}.$

Therefore to achieve admissibility, we can enlarge the square of Heegaard
diagram $\mathcal{\bar{H}}^{L,-L_{1}\cup-L_{2}}:$
\[
\xyR{1pc}\xyC{1pc}\xymatrix{(\Sigma,\boldsymbol{\alpha},\boldsymbol{\beta}_{11},{\bf w},{\bf z})\ar[r]\ar[d]\ar[rd] & (\Sigma,\boldsymbol{\alpha},\boldsymbol{\beta}_{13},{\bf w},{\bf z})\ar[d]\\
(\Sigma,\boldsymbol{\alpha},\boldsymbol{\beta}_{31},{\bf w},{\bf z})\ar[r] & (\Sigma,\boldsymbol{\alpha},\boldsymbol{\beta}_{33},{\bf w},{\bf z})
}
\]
into $\mathcal{H}^{L,-L_{1}\cup-L_{2}}:$
\[
\xyR{1pc}\xyC{1pc}\xymatrix{(\Sigma,\boldsymbol{\alpha},\boldsymbol{\beta}_{11},{\bf w},{\bf z})\ar[r]\ar[d]\ar[rd] & (\Sigma,\boldsymbol{\alpha},\boldsymbol{\beta}_{12},{\bf w},{\bf z})\ar[r]\ar[d]\ar[rd] & (\Sigma,\boldsymbol{\alpha},\boldsymbol{\beta}_{13},{\bf w},{\bf z})\ar[d]\\
(\Sigma,\boldsymbol{\alpha},\boldsymbol{\beta}_{21},{\bf w},{\bf z})\ar[r]\ar[d]\ar[rd] & (\Sigma,\boldsymbol{\alpha},\boldsymbol{\beta}_{22},{\bf w},{\bf z})\ar[r]\ar[d]\ar[rd] & (\Sigma,\boldsymbol{\alpha},\boldsymbol{\beta}_{23},{\bf w},{\bf z})\ar[d]\\
(\Sigma,\boldsymbol{\alpha},\boldsymbol{\beta}_{31},{\bf w},{\bf z})\ar[r] & (\Sigma,\boldsymbol{\alpha},\boldsymbol{\beta}_{32},{\bf w},{\bf z})\ar[r] & (\Sigma,\boldsymbol{\alpha},\boldsymbol{\beta}_{33},{\bf w},{\bf z}).
}
\]

In order to send every hyperbox of Heegaard diagrams $\mathcal{H}^{\overrightarrow{L},\overrightarrow{M}}$
to a hyperbox of chain complexes $\mathfrak{A}^{-}(\mathcal{H}^{\overrightarrow{L},\overrightarrow{M}},\text{s})$,
we need a set of $\Theta$ chain elements. We call the choice of these $\Theta$-elements a \emph{filling} of the
hyperboxes Heegaard diagrams. Let us explain $\Theta$-elements case by case.

For $\mathcal{H}^{L,-L_{1}}$, we have a sequence of strongly equivalent
Heegaard diagrams of $\overrightarrow{L}-L_{1}$,
\[
\mathcal{H}^{L,-L_{1}}:(\Sigma,\boldsymbol{\alpha},\boldsymbol{\beta}_{1},{\bf w},{\bf z})\rightarrow(\Sigma,\boldsymbol{\alpha},\boldsymbol{\beta}_{2},{\bf w},{\bf z})\rightarrow(\Sigma,\boldsymbol{\alpha},\boldsymbol{\beta}_{3},{\bf w},{\bf z}).
\]
We choose a cycle element $\Theta_{\beta_{1},\beta_{2}}$ representing
the maximal degree element in the homology of $\mathfrak{A}^{-}(\mathbb{T}_{\beta_{1}},\mathbb{T}_{\beta_{2}},0)$.
Then we define a chain homotopy equivalence $D_{\beta_{1},\beta_{2}}:\mathfrak{A}^{-}(\mathbb{T_{\alpha}},\mathbb{T}_{\beta_{1}},s)\rightarrow\mathfrak{A}^{-}(\mathbb{T}_{\alpha},\mathbb{T}_{\beta_{2}},s)$
by using triangle maps $D_{\beta_{1},\beta_{2}}({\bf x})=f_{\alpha\beta_{1}\beta_{2}}({\bf x}\otimes\Theta_{\beta_{1},\beta_{2}}).$
Similarly, we get a chain homotopy equivalence $D_{\beta_{2},\beta_{3}}:\mathfrak{A}^{-}(\mathbb{T_{\alpha}},\mathbb{T}_{\beta_{2}},s)\rightarrow\mathfrak{A}^{-}(\mathbb{T}_{\alpha},\mathbb{T}_{\beta_{3}},s)$
by choosing $\Theta_{\beta_{2},\beta_{3}}\in\mathfrak{A}^{-}(\mathbb{T}_{\beta_{2}},\mathbb{T}_{\beta_{3}},0)$.
Thus, $D^{-L_{1}}=D_{\beta_{2}\beta_{3}}\circ D_{\beta_{1}\beta_{2}}$
is also a chain homotopy equivalence. Let us put a subscript on $D^{-L_{1}}$
for labeling the $\mathrm{Spin}^{c}$ structure. Since $\mathfrak{A}^{-}(\mathcal{H}^{L},(+\infty,s_{2}))=\mathfrak{A}(r_{+L_{1}}(\mathcal{H}^{L}),s_{2}-\frac{\mathrm{lk}(+L_{1},+L_{2})}{2})$,
$\mathfrak{A}^{-}(\mathcal{H}^{L},(-\infty,s_{2}))=\mathfrak{A}^{-}(r_{-L_{1}}(\mathcal{H}^{L}),s_{2}-\frac{\mathrm{lk}(-L_{1},+L_{2})}{2})$,
we write
\[
D_{-\infty,s_{2}}^{-L_{1}}:\mathfrak{A}^{-}(\mathcal{H}^{L},(-\infty,s_{2}))\rightarrow\mathfrak{A}^{-}\bigg(\mathcal{H}^{L},\big(+\infty,s_{2}+\mathrm{\mathrm{lk}}(+L_{1},+L_{2})\big)\bigg),
\]
or simply
\[
D_{-\infty,s_{2}}^{-L_{1}}:A_{-\infty,s_{2}}^{-}\rightarrow A_{+\infty,s_{2}+\mathrm{lk}}^{-}.
\]
Similarly we define $D_{s_{1},-\infty}^{-L_{2}}:A_{s_{1},-\infty}^{-}\rightarrow A_{s_{1}+\mathrm{lk},+\infty}^{-}.$

For the 2-dimensional hyperbox of Heegaard diagrams $\mathcal{H}^{L,-L_{1}\cup-L_{2}}$,
we can get a square of chain complexes. Let us first look at the upper
left quarter of $\mathcal{H}^{L,-L_{1}\cup-L_{2}}:$

\[
\xyR{2pc}\xyC{3pc}\xymatrix{(\Sigma,\boldsymbol{\alpha},\boldsymbol{\beta}_{11},{\bf w},{\bf z})\ar[r]^{\Theta_{\beta_{11}\beta_{12}}}\ar[d]_{\Theta_{\beta_{11}\beta_{21}}}\ar[rd]^{\Theta_{\beta_{11}\beta_{22}}} & (\Sigma,\boldsymbol{\alpha},\boldsymbol{\beta}_{12},{\bf w},{\bf z})\ar[d]^{\Theta_{\beta_{12}\beta_{22}}}\\
(\Sigma,\boldsymbol{\alpha},\boldsymbol{\beta}_{21},{\bf w},{\bf z})\ar[r]_{\Theta_{\beta_{21}\beta_{22}}} & (\Sigma,\boldsymbol{\alpha},\boldsymbol{\beta}_{22},{\bf w},{\bf z}).
}
\]
In the above the diagram, the $\Theta$-elements on the edges are
arbitrary cycles representing the maximal degree elements in their
homology groups. Let $c=f_{\beta_{11}\beta_{12}\beta_{22}}(\Theta_{\beta_{11}\beta_{12}}\otimes\Theta_{\beta_{12}\beta_{22}})+f_{\beta_{11}\beta_{21}\beta_{22}}(\Theta_{\beta_{11}\beta_{21}}\otimes\Theta_{\beta_{21}\beta_{22}})$.
The equation
\begin{align*}
\partial(f_{\beta_{11}\beta_{12}\beta_{22}}(\Theta_{\beta_{11}\beta_{12}}\otimes\Theta_{\beta_{12}\beta_{22}})+f_{\beta_{11}\beta_{21}\beta_{22}}(\Theta_{\beta_{11}\beta_{21}}\otimes\Theta_{\beta_{21}\beta_{22}})) & =\\
f_{\beta_{11}\beta_{12}\beta_{22}}((\partial\Theta_{\beta_{11}\beta_{12}})\otimes\Theta_{\beta_{12}\beta_{22}})+f_{\beta_{11}\beta_{21}\beta_{22}}(\Theta_{\beta_{11}\beta_{21}}\otimes(\partial\Theta_{\beta_{21}\beta_{22}})) & =0
\end{align*}
shows that $c$ is a cycle in $\mathfrak{A}_{\mu}^{-}(\mathbb{T}_{\beta_{11}},\mathbb{T}_{\beta_{22}},0)$,
where $\mu$ equals to the maximal degree of the homology of $\mathfrak{A}^{-}(\mathbb{T}_{\beta_{11}},\mathbb{T}_{\beta_{22}},0)$.
Since the curves $\beta_{**}$ are all strongly equivalent, up to
chain homotopy equivalences, we can only consider the case when they are
all small Hamiltonian isotopies of each other. By Lemma 9.7 in \cite{OS_HF1},
we can see that $f_{\beta_{11}\beta_{12}\beta_{22}}(\Theta_{\beta_{11}\beta_{12}}\otimes\Theta_{\beta_{12}\beta_{22}}),f_{\beta_{11}\beta_{21}\beta_{22}}(\Theta_{\beta_{11}\beta_{21}}\otimes\Theta_{\beta_{21}\beta_{22}})$
both represent the maximal degree element in the homology of $CF(\mathbb{T}_{\beta_{11}},\mathbb{T}_{\beta_{22}})$.
Thus, $c$ is $0$ in the homology. So there is a $\Theta_{\beta_{11},\beta_{22}}$
such that $\partial\Theta_{\beta_{11},\beta_{22}}=c$, where $\partial=f_{\beta_{11},\beta_{22}}.$  In sum,
\[
f_{\beta_{11}\beta_{12}\beta_{22}}(\Theta_{\beta_{11}\beta_{12}}\otimes\Theta_{\beta_{12}\beta_{22}})+f_{\beta_{11}\beta_{21}\beta_{22}}(\Theta_{\beta_{11}\beta_{21}}\otimes\Theta_{\beta_{21}\beta_{22}})=f_{\beta_{11}\beta_{22}}(\Theta_{\beta_{11}\beta_{22}}).
\]
From the quadratic $A_{\infty}$ associativity Equation \eqref{eq:A-infinity}, we have a square of chain complexes
\[
{\xyR{1.5pc}\xymatrix{{\mathfrak{A}^{-}(\mathbb{T}_{\alpha},\mathbb{T}_{\beta_{11}},\text{s})}\ar[r]^{D_{\beta_{11}\beta_{12}}}\ar[d]_{{D_{\beta_{11}\beta_{21}}}}\ar[rd]^{{D_{\beta_{11}\beta_{22}}}} & {\mathfrak{A}^{-}(\mathbb{T}_{\alpha},\mathbb{T}_{\beta_{12}},\text{s})\ar[d]^{{D_{\beta_{12}\beta_{22}}}}}\\
{\mathfrak{A}^{-}(\mathbb{T}_{\alpha},\mathbb{T}_{\beta_{21}},\text{s})}\ar[r]_{\mathit{D_{\beta_{21}\beta_{22}}}} & {\mathfrak{A}^{-}(\mathbb{T}_{\alpha},\mathbb{T}_{\beta_{22}},\text{s}),}
}
}
\]
where $D_{\beta_{11}\beta_{22}}(x)=f_{\alpha\beta_{11}\beta_{12}\beta_{22}}(x\otimes\Theta_{\beta_{11}\beta_{12}}\otimes\Theta_{\beta_{12}\beta_{22}})+f_{\alpha\beta_{11}\beta_{21}\beta_{22}}(x\otimes\Theta_{\beta_{11}\beta_{21}}\otimes\Theta_{\beta_{21}\beta_{22}})+f_{\alpha\beta_{11}\beta_{22}}(x\otimes\Theta_{\beta_{11}\beta_{22}}),$
and
\begin{align*}
D_{\beta_{11}\beta_{12}}(x)=f_{\alpha\beta_{11}\beta_{12}}(x\otimes\Theta_{\beta_{11}\beta_{12}}), & \ \ D_{\beta_{12}\beta_{22}}(x)=f_{\alpha\beta_{12}\beta_{22}}(x\otimes\Theta_{\beta_{12}\beta_{22}}),\\
D_{\beta_{11}\beta_{21}}(x)=f_{\alpha\beta_{11}\beta_{21}}(x\otimes\Theta_{\beta_{11}\beta_{21}}), & \ \ D_{\beta_{11}\beta_{22}}(x)=f_{\alpha\beta_{11}\beta_{22}}(x\otimes\Theta_{\beta_{11}\beta_{22}}).
\end{align*}
Similarly, we can choose other $\Theta$-elements on $\mathcal{H}^{L,-L_{1}\cup-L_{2}},$
and get a rectangle of chain complexes of size $(2,2)$. We denote
it by $\mathfrak{A}^{-}(\mathcal{H}^{L,-L_{1}\cup-L_{2}},\text{s}).$

\begin{defn}
We define the destabilization map $D^{-L_{1}\cup-L_{2}}$ to be the diagonal
map in the compression of $\mathfrak{A}^{-}(\mathcal{H}^{L,-L_{1}\cup-L_{2}},\text{s})$:
\[
\mathit{\xyR{1pc}\xymatrix{\mathit{\mathfrak{A}^{-}(r_{-L_{1}\cup-L_{2}}(\mathcal{H}^{L}),\text{s})}\ar[r]\ar[d]\ar[rd] & \mathit{\mathfrak{A}^{-}(r_{-L_{1}\cup+L_{2}}(\mathcal{H}^{L}),\text{s})}\ar[d]\\
\mathit{\mathfrak{A}^{-}(r_{+L_{1}\cup-L_{2}}(\mathcal{H}^{L}),\text{s})}\ar[r] & \mathit{\mathfrak{A}^{-}(r_{+L_{1}\cup+L_{2}}(\mathcal{H}^{L}),\text{s}).}
}
}
\]
Since $\mathfrak{A}^{-}(r_{-L_{1}\cup-L_{2}}(\mathcal{H}^{L}),\text{s})=\mathfrak{A}^{-}(\mathcal{H}^{L},-\infty,-\infty)$,
we denote it by $D_{-\infty,-\infty}^{-L_{1}\cup-L_{2}}.$
\end{defn}

As all the other hyperboxes of Heegaard diagrams are trivial, the
following identities hold
\[
D_{+\infty,s_{2}}^{+L_{1}}=id,\ D_{s_{1},+\infty}^{+L_{2}}=id,\ D_{-\infty,+\infty}^{-L_{1}\cup+L_{2}}=0,\ D_{+\infty,-\infty}^{+L_{1}\cup-L_{2}}=0,\ D_{+\infty,+\infty}^{+L_{1}\cup+L_{2}}=0.
\]

Now we can build all the rectangles of chain complexes as follows,
where $\mathrm{lk}=\mathrm{lk}(\overrightarrow{L_{1}},\overrightarrow{L_{2}})$.
\begin{align}
\label{eq:surgery square}
\begin{array}{c}
\mathit{\mathfrak{R}_{\text{s},1,1}:=}\\
\xyC{2pc}\xyR{2pc}\mathit{\xymatrix{\mathit{A_{s_{1},s_{2}}^{-}}\ar[d]_{\mathit{I_{s_{1},s_{2}}^{-L_{1}}}}\ar[r]^{\mathit{I_{s_{1},s_{2}}^{-L_{2}}}}\ar[dr]|-{\mathit{I_{s_{1},s_{2}}^{-L_{1}\cup-L_{2}}}} & \mathit{A_{s_{1},-\infty}^{-}}\ar[d]^{\mathit{I_{s_{1},-\infty}^{-L_{1}}}}\ar[r]^{\mathit{D_{s_{1},-\infty}^{-L_{2}}}} & \mathit{A_{s_{1+\mathrm{lk}},+\infty}^{-}\ar[d]^{I_{s_{1}+\mathrm{lk},+\infty}^{-L_{1}}}}\\
\mathit{A_{-\infty,s_{2}}^{-}}\ar[d]_{\mathit{D_{-\infty,s_{2}}^{-L_{1}}}}\ar[r]_{\mathit{I_{-\infty,s_{2}}^{-L_{2}}}} & \mathit{A_{-\infty,-\infty}^{-}}\ar[d]_{\mathit{D_{-\infty,-\infty}^{-L_{1}}}}\ar[r]^{\mathit{D_{-\infty,-\infty}^{-L_{2}}}}\ar[dr]|-{\mathit{D_{-\infty,-\infty}^{-L_{1}\cup-L_{2}}}} & A_{-\infty,+\infty}^{-}\ar[d]^{D_{-\infty,+\infty}^{-L_{1}}}\\
\mathit{A_{+\infty,s_{2}+\mathrm{lk}}^{-}}\ar[r]_{\quad\mathit{I_{+\infty,s_{2}+\mathrm{lk}}^{-L_{2}}}} & \mathit{A_{+\infty,-\infty}^{-}}\ar[r]_{\mathit{D_{+\infty,-\infty}^{-L_{2}}}} & A_{+\infty,+\infty}^{-};
}
}
\end{array} & \begin{array}{c}
\mathit{\mathfrak{R}_{\text{s},0,1}:=}\\
\xyC{2pc}\xyR{2pc}\mathit{\xymatrix{\mathit{A_{s_{1},s_{2}}^{-}}\ar[d]_{\mathit{I_{s_{1},s_{2}}^{+L_{1}}}}\ar[r]^{\mathit{I_{s_{1},s_{2}}^{-L_{2}}}}\ar[dr]|-{\mathit{I_{s_{1},s_{2}}^{+L_{1}\cup-L_{2}}}} & \mathit{A_{s_{1},-\infty}^{-}}\ar[d]^{\mathit{I_{s_{1},-\infty}^{+L_{1}\cup-L_{2}}}}\ar[r]^{\mathit{D_{s_{1},-\infty}^{-L_{2}}}} & \mathit{A_{s_{1}+\mathrm{lk},+\infty}^{-}}\ar[d]^{\mathit{I_{s_{1},+\infty}^{+L_{1}}}}\\
\mathit{A_{+\infty,s_{2}}^{-}}\ar[d]_{\mathit{id}}\ar[r]_{\mathit{I_{+\infty,s_{2}}^{-L_{2}}}} & \mathit{A_{+\infty,-\infty}^{-}}\ar[d]_{\mathit{id}}\ar[r]_{\mathit{D_{+\infty,-\infty}^{-L_{2}}}} & \mathit{A_{+\infty,+\infty}^{-}}\ar[d]^{\mathit{id}}\\
\mathit{A_{+\infty,s_{2}}^{-}}\ar[r]_{\mathit{I_{+\infty,s_{2}}^{-L_{2}}}} & \mathit{A_{+\infty,-\infty}^{-}}\ar[r]_{\mathit{D_{+\infty,-\infty}^{-L_{2}}}} & \mathit{A_{+\infty,+\infty}^{-};}
}
}
\end{array}\\
\mathit{\begin{array}{c}
\mathit{\mathit{\mathfrak{R}_{\text{s},1,0}:=}}\\
\xyC{2pc}\xyR{2pc}\mathit{\xymatrix{\mathit{A_{s_{1},s_{2}}^{-}}\ar[d]_{\mathit{I_{s_{1},s_{2}}^{-L_{1}}}}\ar[r]^{\mathit{I_{s_{1},s_{2}}^{+L_{2}}}}\ar[dr]|-{\mathit{I_{s_{1},s_{2}}^{-L_{1}\cup+L_{2}}}} & \mathit{A_{s_{1},+\infty}^{-}}\ar[d]^{\mathit{I_{s_{1},+\infty}^{-L_{1}}}}\ar[r]^{\mathit{id}} & \mathit{A_{s_{1},+\infty}^{-}}\ar[d]^{\mathit{I_{s_{1},+\infty}^{-L_{1}}}}\\
\mathit{A_{-\infty,s_{2}}^{-}}\ar[d]_{\mathit{D_{-\infty,s_{2}}^{-L_{1}}}}\ar[r]_{\mathit{I_{-\infty,s_{2}}^{+L_{2}}}} & \mathit{A_{-\infty,+\infty}^{-}}\ar[d]^{\mathit{D_{-\infty,+\infty}^{-L_{1}}}}\ar[r]^{\mathit{id}} & \mathit{A_{-\infty,+\infty}^{-}}\ar[d]^{\mathit{D_{-\infty,+\infty}^{-L_{1}}}}\\
\mathit{A_{+\infty,s_{2}+\mathrm{lk}}^{-}}\ar[r]_{\mathit{I_{+\infty,s_{2}}^{+L_{2}}}} & \mathit{A_{+\infty,+\infty}^{-}}\ar[r]_{\mathit{id}} & \mathit{A_{+\infty,+\infty}^{-};}
}
}
\end{array}} & \begin{array}{c}
\mathit{\mathit{\mathfrak{R}_{\text{s},0,0}:=}}\\
\xyC{2pc}\xyR{2pc}\mathit{\xymatrix{\mathit{A_{s_{1},s_{2}}^{-}}\ar[d]_{\mathit{I_{s_{1},s_{2}}^{+L_{1}}}}\ar[r]^{\mathit{I_{s_{1},s_{2}}^{+L_{2}}}}\ar[dr]|-{\mathit{I_{s_{1},s_{2}}^{+L_{1}\cup+L_{2}}}} & \mathit{A_{s_{1},+\infty}^{-}}\ar[d]^{\mathit{I_{s_{1},+\infty}^{+L_{1}}}}\ar[r]^{\mathit{id}} & \mathit{A_{s_{1},+\infty}^{-}}\ar[d]^{\mathit{I_{s_{1},+\infty}^{+L_{1}}}}\\
\mathit{A_{+\infty,s_{2}}^{-}}\ar[d]_{\mathit{id}}\ar[r]_{\mathit{I_{+\infty,s_{2}}^{+L_{2}}}} & \mathit{A_{+\infty,+\infty}^{-}}\ar[d]^{\mathit{id}}\ar[r]^{\mathit{id}} & \mathit{A_{+\infty,+\infty}^{-}}\ar[d]^{\mathit{id}}\\
\mathit{A_{+\infty,s_{2}}^{-}}\ar[r]_{\mathit{I_{+\infty,s_{2}}^{+L_{2}}}} & \mathit{\mathit{A_{+\infty,+\infty}^{-}}}\ar[r]_{\mathit{\mathit{id}}} & \mathit{\mathit{A_{+\infty,+\infty}^{-}.}}
}
}
\end{array}
\end{align}
The squares $R_{\text{s},i,j}$'s used in Equations (\ref{eq:surgery formula1})
to (\ref{eq:surgery formual4}) are defined to be the compressions
of $\mathfrak{R}_{\text{s},i,j}$'s.

\section{Applying the surgery formula to two-bridge links}

In this section, we show some algebraic rigidity results for the chain maps
between certain chain complexes up to chain homotopy. This provides a
 way to determine the destabilization
maps in the surgery complex of two-bridge links up to chain homotopy.
Using these maps to replace the original maps in the surgery complex,
we construct a \emph{perturbed surgery complex}. We
further show that it has the same homology as the original one. Based on the
perturbed surgery formula, we give an algorithm for computing the homology of surgeries
on two-bridge links.

\subsection{Algebraic rigidity results.}

There is a short exact sequence in the Exercise 3.6.1 in \cite{Weibel_homological}
as follows. Suppose $P_{*},Q^{*}$ are (co)chain complexes of $R$-modules,
and $P,d(P)=\mathrm{Im}(d)$ are both projective $R$-modules. Then
there is an exact sequence
\[
0\rightarrow\prod_{p+q=n-1}\mathrm{Ext}_{R}^{1}(H_{p}(P_{*}),H^{q}(Q^{*}))\rightarrow H_{n}(\text{Hom}(P_{*},Q^{*}))\rightarrow\prod_{p+q=n}\mathrm{Hom}_{R}(H_{p}(P_{*}),H^{q}(Q^{*}))\rightarrow0.
\]
For completeness, we give a proof here adapted to the setting of $\mathbb{Z}/2\mathbb{Z}$-graded
chain complexes.

\begin{lem}
\label{lem:H*(Hom(P,Q))}
Let $P_{*},Q_{*}$ be $\mathbb{Z}/2\mathbb{Z}$-graded chain complexes
of $R$-modules. If $P$ and $d(P)=\mathrm{Im}(d)$ are projective modules, then
there is a short exact sequence for any $n,p,q\in\mathbb{Z}/2\mathbb{Z}$,
\begin{equation}
0\rightarrow\bigoplus_{p+q=n+1}\mathrm{Ext}_{R}^{1}(H_{p}(P),H_{q}(Q))\rightarrow H_{n}(\mathrm{Hom}(P,Q))\rightarrow\bigoplus_{p+q=n}\mathrm{Hom}_{R}(H_{p}(P),H_{q}(Q))\rightarrow0.\label{eq:universal coefficient}
\end{equation}
\end{lem}

\begin{proof}
First, all the indices $n,p,q,i,j$ are in $\mathbb{Z}/2\mathbb{Z}.$
Since $d(P)$ is projective, the short exact sequence $0\rightarrow Z^{P}\rightarrow P\rightarrow d(P)\rightarrow0$
splits, thus giving that $P=d(P)\oplus Z^{P}$. Thereby, $Z^{P}$
is projective and thus $\text{Ext}_{R}^{1}(Z^{P},M)=0$ for all $R$-module
$M$. Also, by $\text{Ext}_{R}^{1}(d(P),M)=0,\forall M$ we get an
exact sequence
\[
0\rightarrow\text{Hom}_{R}(d(P_{p}),Q_{q})\rightarrow\text{Hom}_{R}(P_{p},Q_{q})\rightarrow\text{Hom}_{R}(Z_{p}^{P},Q_{q})\rightarrow0.
\]
These assemble to a short exact sequence of chain complexes
\begin{equation}
0\rightarrow\bigoplus_{p+q=n}\text{Hom}_{R}(d(P_{p}),Q_{q})\rightarrow(\text{Hom}_{R}(P,Q))_{n}\rightarrow\bigoplus_{p+q=n}\text{Hom}_{R}(Z_{p}^{P},Q_{q})\rightarrow0.\label{eq:S.E.S of Hom}
\end{equation}
 Actually, it is not hard to check the following commuting diagram
\[
\xymatrix{0\ar[r] & \bigoplus\limits_{p+q=n}\text{Hom}_{R}(d(P_{p}),Q_{q})\ar[r]\ar[d]_{d} & (\text{Hom}_{R}(P,Q))_{n}\ar[d]_{d}\ar[r] & \bigoplus\limits_{p+q=n}\text{Hom}_{R}(Z_{p}^{P},Q_{q})\ar[d]_{d}\ar[r] & 0\\
0\ar[r] & \bigoplus\limits_{p+q=n+1}\text{Hom}_{R}(d(P_{p}),Q_{q})\ar[r] & (\text{Hom}_{R}(P,Q))_{n+1}\ar[r] & \bigoplus\limits_{p+q=n+1}\text{Hom}_{R}(Z_{p}^{P},Q_{q})\ar[r] & 0.
}
\]
Since $dP$ is projective, the short exact sequence $0\rightarrow d(Q_{j-1})\rightarrow Z_{j}^{Q}\rightarrow H_{j}(Q)\rightarrow0$
gives a short exact sequence
\[
0\rightarrow\text{Hom}_{R}(dP_{i},d(Q_{j-1}))\rightarrow\text{Hom}_{R}(dP_{i},Z_{j}^{Q})\rightarrow\text{Hom}_{R}(dP_{i},H_{j}(Q))\rightarrow0.
\]
Furthermore, the differential in $\text{Hom}_{R}(dP_{i},Q_{j})$ is
$d_{Q}$, so from the above exact sequence it follows that $H_{n}(\text{Hom}_{R}(d(P),Q))=\bigoplus_{p+q=n}\text{Hom}_{R}(d(P_{p}),H_{q}(Q)),p,q,n\in\mathbb{Z}/2\mathbb{Z}.$
Since $Z^{P}$ is projective, similarly we have
\[
H_{n}(\text{Hom}_{R}(Z^{P},Q))=\bigoplus_{p+q=n}\text{Hom}_{R}(Z_{p}^{P},H_{q}(Q)).
\]

Thus the long exact sequence of homology from Equation \eqref{eq:S.E.S of Hom}
is
\begin{eqnarray}
\label{eq:L.E.S of Homology}
\cdots\rightarrow\bigoplus_{p+q=n}\text{Hom}_{R}(Z_{p+1}^{P},H_{q}(Q)) & \xrightarrow{\partial_{n+1}} & \bigoplus_{p+q=n}\text{Hom}_{R}(d(P_{p}),H_{q}(Q))\rightarrow H_{n}(\text{Hom}_{R}(P,Q))\nonumber \\
\rightarrow\bigoplus_{p+q=n}\text{Hom}_{R}(Z_{p}^{P},H_{q}(Q)) & \xrightarrow{\partial_{n}} & \bigoplus_{p+q=n}\text{Hom}_{R}(d(P_{p+1}),H_{q}(Q))\rightarrow\cdots.
\end{eqnarray}
A diagram chasing shows that the connecting morphism $\partial_{*}:\text{Hom}(Z_{*}^{P},H_{*}(Q))\rightarrow\text{Hom}(d(P_{*}),H_{*}(Q))$
is the restriction.

Hence, the short exact sequence $0\rightarrow dP_{i+1}\rightarrow Z_{i}^{P}\rightarrow H_{i}(P)\rightarrow0$
can produce the exact sequence
\begin{eqnarray*}
0 & \rightarrow & \text{Hom}_{R}(H_{p}(P),H_{q}(Q))\rightarrow\text{Hom}_{R}(Z_{p}^{P},H_{q}(Q))\xrightarrow{\partial_{p+q}}\text{Hom}_{R}(dP_{p+1},H_{q}(Q))\\
 & \rightarrow & \text{Ext}_{R}^{1}(H_{p}(P),H_{q}(Q))\rightarrow\text{Ext}_{R}^{1}(Z_{p}^{P},H_{q}(Q))=0,
\end{eqnarray*}
thus $\text{Ker}(\partial_{p+q})\cong\text{Hom}_{R}(H_{p}(P),H_{q}(Q))$
and $\text{Coker}(\partial_{p+q})\cong\text{Ext}_{R}^{1}(H_{p}(P),H_{q}(Q))$.
Finally, the exact sequence in Equation (\ref{eq:universal coefficient})
comes from Equation (\ref{eq:L.E.S of Homology}).
\end{proof}

Let $(C_{*},\partial_{*})$ be a chain complex of $\mathbb{F}$-vector
spaces, with $U_{1},U_{2}$-actions which drop the $\mathbb{Z}$-grading
by 2. Consider $C$ as a $\mathbb{F}[[U_{1},U_{2}]]$-module. Even
though the $U_{1},U_{2}$-actions do not preserve the $\mathbb{Z}$-grading,
we will still call $C$ a \emph{complex of $\mathbb{F}[[U_{1},U_{2}]]$-modules}.

\begin{prop}

\label{prop:homotopy equivalence}Let $A,B$ be complexes of $\mathbb{F}[[U_{1},U_{2}]]$-modules
with $U_{1},U_{2}$-actions dropping grading by 2, and $A,d(A)$ are
both free $\mathbb{F}[[U_{1},U_{2}]]$-modules. Suppose $H_{*}(A)\cong H_{*}(B)\cong\mathbb{F}[[U_{1},U_{2}]]/(U_{1}-U_{2})$,
precisely, $H_{2k}(A)\cong H_{2k}(B)\cong\mathbb{F}$ for all $k\leq0$
and $H_{i}(A)=H_{i}(B)=0$ otherwise, where $U_{i}\cdot H_{2k}(A)=H_{2k-2}(A),U_{i}\cdot H_{2k}(B)=H_{2k-2}(B)$
for both $i=1,2$. If $F,G:A\rightarrow B$ are both quasi-isomorphisms
as $\mathbb{F}[[U_{1},U_{2}]]$-modules, then $F$ and $G$ are homotopic
as $\mathbb{F}[[U_{1},U_{2}]]$-modules.

\end{prop}

\begin{proof}
First, the $\mathbb{Z}$-grading of $A,B$ induces a $\mathbb{Z}/2\mathbb{Z}$-grading
on both $A$ and $B$, and $U_{1},U_{2}$-action preserves the induced
$\mathbb{Z}/2\mathbb{Z}$-grading, thus we regard $A,B$ as $\mathbb{Z}/2\mathbb{Z}$-graded
chain complexes of $\mathbb{F}[[U_{1},U_{2}]]$-modules. In order
to distinguish these two gradings, we put brackets on the numbers
to represent $\mathbb{Z}/2\mathbb{Z}$-gradings. Hence we have $H_{[0]}(A)=H_{[0]}(B)=\mathbb{F}[[U_{1},U_{2}]]/(U_{1}-U_{2}),H_{[1]}(A)=H_{[1]}(B)=0$.

By Lemma \ref{lem:H*(Hom(P,Q))}, we have
\begin{align*}
0 & \rightarrow\bigoplus_{[p]\in\mathbb{Z}/2\mathbb{Z}}\text{Ext}_{\mathbb{F}[[U_{1},U_{2}]]}^{1}(H_{[p+1]}(A_{*}),H_{[p]}(B_{*}))\rightarrow H_{[0]}(\text{Hom}(A_{*},B_{*}))\\
 & \rightarrow\bigoplus_{[p]\in\mathbb{Z}/2\mathbb{Z}}\text{Hom}_{\mathbb{F}[[U_{1},U_{2}]]}(H_{[p]}(A_{*}),H_{[p]}(B_{*}))\rightarrow0,
\end{align*}
thus $H_{[0]}(\text{Hom}(A_{*},B_{*}))=\text{Hom}_{\mathbb{F}[[U_{1},U_{2}]]}(\mathbb{F}[[U_{1},U_{2}]]/(U_{1}-U_{2}),\mathbb{F}[[U_{1},U_{2}]]/(U_{1}-U_{2}))=\mathbb{F}[[U_{1},U_{2}]]/(U_{1}-U_{2}).$
Since $H_{[0]}(\text{Hom}(A_{*},B_{*}))$ is the group of chain homotopy
equivalence classes of chain maps from $A$ to $B$, this means the
chain maps from $A$ to $B$ are classified by their action on homology.
Since $F$ and $G$ are both quasi-isomorphisms, they are homotopic
as $\mathbb{F}[[U_{1},U_{2}]]$-modules.

Let $H:A\rightarrow B$ be any homotopy such that $F-G=H\partial+\partial H,H\cdot U_{i}=U_{i}\cdot H$.
Then $H$ shifts the $\mathbb{Z}/2\mathbb{Z}$-grading by 1, thus
shifting the $\mathbb{Z}$-grading by odd numbers. Thus, let $H=\sum_{i\in\mathbb{Z}}H_{2i+1}$,
where $H_{2i+1}:A_{*}\rightarrow B_{*+2i+1}$. Since
\[
F-G=H\partial+\partial H=\sum_{i\in\mathbb{Z}}(H_{2i+1}\partial+\partial H_{2i+1})
\]
preserves the original $\mathbb{Z}$-grading, we have $H_{2i+1}\partial+\partial H_{2i+1}=0,\forall i\neq0.$
So we can replace the homotopy $H$ by $H_{1}:A_{*}\rightarrow B_{*+1}$,
thus being a chain homotopy of the original $\mathbb{Z}$-graded chain
complexes.\end{proof}

\begin{rem}
\label{rem:cx_of_the_unknot}
The prototype of the complexes in the previous Proposition is the simplest Heegaard Floer chain complex of the unknot in $S^3$. Let $C^{\text u}$ be the chain complex of $\bb{F}[[U_1,U_2]]$-modules generated by ${\bf{x}},{\bf{y}}$ with differential $\partial{{\bf{x}}}=(U_1-U_2){\bf{y}}$, where ${\bf{y}},{\bf{x}}$ are of gradings $0,-1$ respectively.
\end{rem}

\begin{cor}
\label{cor:homotopy equilvence}
Let $A,B$ be complexes of $\bb{F}[[U_1,U_2]]$-modules with $U_1,U_2$-actions dropping grading by 2. Suppose
$A$ is chain homotopy equivalent to the complex $C^{\text u}$ as $\mathbb{F}[[U_1,U_2]]$-modules, and $H_*(B)\cong \bb{F}[[U_1,U_2]]/(U_1-U_2)$
as in Proposition \ref{prop:homotopy equivalence}. Then for any quasi-isomorphisms $F,G:A\to B$ of $\bb{F}[[U_1,U_2]]$-modules,
$F$ and $G$ are chain homotopic as $\bb{F}[[U_1,U_2]]$-modules.
\end{cor}

\begin{proof}
Let $h_1:A\to C^{\text u},h_2:C^{\text u}\to A$ be the chain homotopy equivalences, such that $h_1\circ h_2 \simeq id_{C^{\text u}}, h_2\circ h_1 \simeq id_A.$
Then by Proposition \ref{prop:homotopy equivalence}, $F\circ h_2$ is homotopic to $G\circ h_2$ as $\bb{F}[[U_1,U_2]]$-modules. Hence, $F\circ h_2\circ h_1,G\circ h_2\circ h_1$ are homotopic as $\bb{F}[[U_1,U_2]]$-modules, and thus so are $F$ and $G$.
\end{proof}

\begin{prop}
\label{prop:homotopies on the diagonal}Let $A_{*},B_{*}$ be complexes
of $\mathbb{F}[[U]]$-modules with $U$-action dropping grading by
2, and $A,B$ are both free $\bb{F}[[U]]$-modules. Suppose $H_{*}(A)=H_{*}(B)=\mathbb{F}[[U]]$, precisely, $H_{2k}(A)\cong H_{2k}(B)\cong\mathbb{F}$
for all $k\leq0$ and $H_{i}(A)=H_{i}(B)=0$ otherwise, where $U\cdot H_{2k}(A)=H_{2k-2}(A),U\cdot H_{2k}(B)=H_{2k-2}(B)$. If $F,G:A\to B$ are both quasi-isomorphisms of $\bb{F}[[U]]$-modules, then $F,G$ are chain homotopic as maps of
$\bb{F}[[U]]$-modules.

Moreover, if $H,K$ are both chain homotopies as homomorphisms of $\mathbb{F}[[U]]$-modules
between any two chain maps $f,g:A\rightarrow B$, i.e. $H\partial+\partial H=K\partial+\partial K=f-g$,
then $H-K=\partial T+T\partial$, for some $\mathbb{F}[[U]]$-module
homomorphism $T:A_{*}\rightarrow B_{*+2}$.
\end{prop}

\begin{proof}
First, we regard $A,B$ as $\mathbb{Z}/2\mathbb{Z}$-graded complexes
of $\mathbb{F}[[U]]$-modules. Since a P.I.D. is hereditary, every submodule of
 a free module over a P.I.D. is a projective module. See \cite{Weibel_homological} Definition
 4.2.10 and Exercise 4.2.6. Thus, $d(A),d(B)$ are both projective $\bb{F}[[U]]$-modules.
  Applying Lemma \ref{lem:H*(Hom(P,Q))}, we can compute
\begin{align*}
H_{[0]}(\mathrm{Hom}(A,B))&=\mathrm{Hom}(\bb{F}[[U]],\bb{F}[[U]])=\bb{F}[[U]],\\ H_{[1]}(\mathrm{Hom}(A,B))&=\mathrm{Ext}_{\mathbb{F}[[U]]}^{1}(\mathbb{F}[[U]],\mathbb{F}[[U]])=0.
\end{align*}

The first identity implies that the quasi-isomorphisms $F,G$ are chain homotopic as $\bb{Z}/2\bb{Z}$-graded complexes via $H$.
In order to get a homotopy between $F$ and $G$ preserving the $\bb{Z}$-grading, we decompose $H=\sum_{i\in \bb{Z}} H_{2i+1}$, where $H_{2i+1}:A_*\to B_{*+2i+1}$. Then similarly to Proposition \ref{prop:homotopy equivalence}, the map $H_1$ is also a chain homotopy between $F,G$.

Since $\partial(H-K)+(H-K)\partial=0$, the second identity implies that $H-K\in Z_{[1]}(\text{Hom}(A,B))=B_{[1]}(\text{Hom}(A,B))$.
This means there is a homomorphism of $\mathbb{F}[[U]]$-modules $T:A\rightarrow B$
preserving the $\mathbb{Z}/2\mathbb{Z}$-grading, such that $H-K=\partial T+T\partial.$
Thus, the map $T$ can be decomposed as $T=\sum_{i\in\mathbb{Z}}T_{2i}$,
where $T_{2i}:A_{*}\rightarrow B_{*+2i}$. From the fact that $H-K=\sum_{i\in\mathbb{Z}}\partial T_{2i}+T_{2i}\partial$
maps $A_{n}$ into $B_{n+1}$, it follows that $\partial T_{2i}+T_{2i}\partial:A_{*}\rightarrow B_{*+2i-1}$
vanish for all $i\neq1$. Thus $T=T_{2}:A_{*}\rightarrow B_{*+2}$.
\end{proof}

\begin{cor}
\label{cor:standard model of unknot}Suppose the complexes $A_{*}$
and $B_{*}$ are as in Proposition \ref{prop:homotopies on the diagonal}.
Then, $A_{*}$ and $B_{*}$ are chain homotopy equivalent as $\bb{F}[[U]]$-modules.
\end{cor}

\begin{proof}

From the proof of Proposition \ref{prop:homotopies on the diagonal} we see $H_{[0]}(\text{Hom}(A_{*},B_{*}))=\mathrm{Hom}(H_{[0]}(A),H_{[0]}(B))=\bb{F}[[U]],$
which implies that there exists a quasi-isomorphism $h:A\rightarrow B$
as $\mathbb{Z}/2\mathbb{Z}$-graded chain complex of $\mathbb{F}[[U]]$-modules.
Decompose $h$ as $h=h_{0}+h_{1}$ such that for all $a\in A_{n}$,
$h_{0}(a)\in B_{n},h_{1}(a)\in\bigoplus_{i\neq0}B_{n+2i}$. Then $h\partial_{A}=\partial_{B}h$
implies that $h_{0}\partial_{A}=\partial_{B}h_{0},h_{1}\partial_{A}=\partial_{B}h_{1}$,
so $h_{0}$ is also a chain map preserving the $\mathbb{Z}$-grading.
Since $Uh=hU$ and the $U$-action drops the $\bb{Z}$-grading by
2, we have $h_{0}U+Uh_{0}=0$. In addition, $h_{0}$ is also
a quasi-isomorphism. Hence, on the homology level, $h_{0}(1)=1\in\mathbb{F}[[U]].$

Similarly, we have another quasi-isomorphism $g_{0}:B_{*}\rightarrow A_{*}$
preserving the $\mathbb{Z}$-grading, such that $g_0\partial_{B}=\partial_{A}g_{0},g_{0}U=Ug_{0}$.
Then, $g_{0}h_{0}:A_{*}\rightarrow A_{*}$ is a quasi-isomorphism.
From Proposition \ref{prop:homotopies on the diagonal}, it follows that $g_{0}h_{0}-id_{A}=\partial H+H\partial$,
where $H$ is chain homotopy of $\mathbb{Z}$-graded complexes commuting
with the $U$-action. Similarly, we have that $h_{0}g_{0}$ is homotopic
to $id_{B}.$
\end{proof}

\subsection{Destabilization maps.}

In the link surgery formula, the part of polygon counts for defining
the destabilization maps is difficult to read off from the Heegaard
diagram. However, in the case of two-bridge links we
can use the algebraic rigidity result to avoid the difficulty.

In a primitive system of hyperboxes, all the destabilization maps
we need are listed below:
\begin{eqnarray*}
D_{-\infty,s_{2}}^{-L_{1}}: & \mathfrak{A}(\mathcal{H}^{L},-\infty,s_{2})\rightarrow\mathfrak{A}^{-}(\mathcal{H}^{L},+\infty,s_{2}+\mathrm{lk}(L_{1},L_{2})),\\
D_{s_{1},-\infty}^{-L_{2}}: & \mathfrak{A}(\mathcal{H}^{L},s_{1},-\infty)\rightarrow\mathfrak{A}^{-}(\mathcal{H}^{L_{2}},s_{1}+\mathrm{lk}(L_{1},L_{2}),+\infty),\\
D_{-\infty,-\infty}^{-L_{1}\cup-L_{2}}: & \mathfrak{A}^{-}(\mathcal{H}^{L},-\infty,-\infty)\rightarrow\mathfrak{A}^{-}(\mathcal{H}^{L},+\infty,+\infty).
\end{eqnarray*}

Second, notice that all the domains and targets of these maps have
homology $\mathbb{F}[[U_{1},U_{2}]]/U_{1}-U_{2}$, which is isomorphic to
$\mathbb{F}[[U_1]]$ as an $\mathbb{F}[[U_1]]$-module. By Proposition
\ref{prop:homotopies on the diagonal}, we can substitute $D_{-\infty,s_{2}}^{-L_{1}},D_{s_{1},-\infty}^{-L_{2}}$
by any $\mathbb{F}[[U_1]]$-linear homotopy equivalence, since they are all homotopic as homomorphisms
of $\mathbb{F}[[U_{1}]]$-modules. We can also substitute the diagonal maps $D_{-\infty,-\infty}^{-L_{1}\cup-L_{2}}$
by any $\mathbb{F}[[U_{1}]]$-linear homotopy shifting grading by
1, since they are homotopic up to higher $\mathbb{F}[[U_{1}]]$-linear homotopy. We will show an invariance theorem
of the surgery square under perturbations of the edge maps and the
diagonal maps in the next sections.

\subsection{Perturbed surgery complex for two-bridge links.}

The rigidity results in  Section 5.1 allow us to perturb
the edge and diagonal maps up to homotopies, in the surgery square
for two-bridge links. However, in order to obtain a square of chain
complexes, we still need some more modifications of the square.

First, suppose we have a hypercube $(C^{\varepsilon},D^{\varepsilon})$.
If we change $D{}_{\varepsilon^{0}}^{\varepsilon}$ to $D'{}_{\varepsilon^{0}}^{\varepsilon}=D{}_{\varepsilon^{0}}^{\varepsilon}+\Delta D{}_{\varepsilon^{0}}^{\varepsilon}$
for all $\varepsilon$ with $\Vert\varepsilon\Vert>0$, then in order
to have a hypercube again, we need to have
\begin{eqnarray*}
\sum_{\varepsilon'\leq\varepsilon}(D_{\varepsilon^{0}+\varepsilon'}^{\varepsilon-\varepsilon'}+\Delta D{}_{\varepsilon^{0}+\varepsilon'}^{\varepsilon-\varepsilon'})\circ(D_{\varepsilon^{0}}^{\varepsilon'}+\Delta D{}_{\varepsilon^{0}}^{\varepsilon}) & = & 0,\\
\sum_{\varepsilon'\leq\varepsilon}\Delta D{}_{\varepsilon^{0}+\varepsilon'}^{\varepsilon-\varepsilon'}\circ D_{\varepsilon^{0}}^{\varepsilon'}+D_{\varepsilon^{0}+\varepsilon'}^{\varepsilon-\varepsilon'}\circ\Delta D{}_{\varepsilon^{0}}^{\varepsilon}+\Delta D{}_{\varepsilon^{0}+\varepsilon'}^{\varepsilon-\varepsilon'}\circ\Delta D{}_{\varepsilon^{0}}^{\varepsilon} & = & 0.
\end{eqnarray*}
This formula provides a necessary condition to inductively perturb
the maps from edges to the longest diagonal. Based on the above principles,
we get the following procedures to construct the perturbed surgery
square.

Suppose $\mathcal{H}$ be a primitive system of hyperboxes of a two-bridge
link $L$ and consider Equation (\ref{eq:surgery square}). Now we
choose an arbitrary $\mathbb{F}[[U_{1}]]$-linear quasi-isomorphism $\tilde{D}_{s_{1},s_{2}}^{-L_{i}}$
 for substituting $D_{s_{1},s_{2}}^{-L_{i}}$.
By Proposition \ref{prop:homotopies on the diagonal}, $\tilde{D}_{s_{1},s_{2}}^{-L_{i}}$
and $D_{s_{1},s_{2}}^{-L_{i}}$ are homotopic by a $\mathbb{F}[[U_{1}]]$-linear
homotopy $H_{s_{1},s_{2}}^{-L_{i}}$:
\[
\tilde{D}_{\text{s}}^{-L_{i}}=D_{\text{s}}^{-L_{i}}+H_{\text{s}}^{-L_{i}}\partial_{\text{s}}^{-}+\partial_{p^{L_{i}}(\text{s})}^{-}H_{\text{s}}^{-L_{i}}.
\]
 Then, we choose any $\mathbb{F}[[U_{1}]]$-linear maps $\tilde{F}_{s_{1},-\infty}^{\pm L_{1}\cup-L_{2}},\tilde{F}_{-\infty,s_{2}}^{-L_{1}\cup\pm L_{2}},\tilde{D}_{-\infty,-\infty}^{-L_{1}\cup-L_{2}}$
which are homotopies in each square of Equation (\ref{eq:surgery square}),
such that the following rectangles are hyperboxes of chain complexes:
\begin{align}
\xyC{2.5pc}\xyR{2.5pc}\xymatrix{\mathit{A_{s_{1},s_{2}}^{-}}\ar[d]_{\mathit{I_{s_{1},s_{2}}^{-L_{1}}}}\ar[r]^{\mathit{I_{s_{1},s_{2}}^{-L_{2}}}}\ar[dr]|-{\mathit{I_{s_{1},s_{2}}^{-L_{1}\cup-L_{2}}}} & \mathit{A_{s_{1},-\infty}^{-}}\ar[d]|-{\mathit{I_{s_{1},-\infty}^{-L_{1}}}}\ar[r]^{\mathit{\tilde{D}_{s_{1},-\infty}^{-L_{2}}}}\ar[dr]|-{\mathit{\tilde{F}_{s_{1},-\infty}^{-L_{1}\cup-L_{2}}}} & \mathit{A_{s_{1}+\mathrm{lk},+\infty}^{-}}\ar[d]^{\mathit{I_{s_{1}+\mathrm{lk},+\infty}^{-L_{1}}}}\\
\mathit{A_{-\infty,s_{2}}^{-}}\ar[d]_{\mathit{\tilde{D}_{-\infty,s_{2}}^{-L_{1}}}}\ar[r]_{\mathit{I_{-\infty,s_{2}}^{-L_{2}}}}\ar[dr]|-{\mathit{\tilde{F}_{-\infty,s_{2}}^{-L_{1}\cup-L_{2}}}} & \mathit{A_{-\infty,-\infty}^{-}}\ar[d]|-{\mathit{\tilde{D}_{-\infty,-\infty}^{-L_{1}}}}\ar[r]^{\mathit{\tilde{D}_{-\infty,-\infty}^{-L_{2}}}}\ar[dr]|-{\mathit{\tilde{D}_{-\infty,-\infty}^{-L_{1}\cup-L_{2}}}} & \mathit{A_{-\infty,+\infty}^{-}}\ar[d]^{\mathit{\tilde{D}_{-\infty,+\infty}^{-L_{1}}}}\\
\mathit{A_{+\infty,s_{2}+\mathrm{lk}}^{-}}\ar[r]_{\mathit{I_{+\infty,s_{2}+\mathrm{lk}}^{-L_{2}}}} & \mathit{A_{+\infty,-\infty}^{-}}\ar[r]_{\mathit{\tilde{D}_{+\infty,-\infty}^{-L_{2}}}} & \mathit{A_{+\infty,+\infty}^{-};}
}
 & \xyC{2.5pc}\xyR{2.5pc}\xymatrix{\mathit{A_{s_{1},s_{2}}^{-}}\ar[d]_{\mathit{I_{s_{1},s_{2}}^{+L_{1}}}}\ar[r]^{\mathit{I_{s_{1},s_{2}}^{-L_{2}}}}\ar[dr]|-{\mathit{I_{s_{1},s_{2}}^{+L_{1}\cup-L_{2}}}} & \mathit{A_{s_{1},-\infty}^{-}}\ar[d]|-{\mathit{I_{s_{1},-\infty}^{+L_{1}}}}\ar[r]^{\mathit{\tilde{D}_{s_{1},-\infty}^{-L_{2}}}}\ar[dr]|-{\mathit{\tilde{F}_{s_{1},-\infty}^{+L_{1}\cup-L_{2}}}} & \mathit{A_{s_{1}+\mathrm{lk},+\infty}^{-}}\ar[d]^{\mathit{I_{s_{1},+\infty}^{+L_{1}}}}\\
\mathit{A_{+\infty,s_{2}}^{-}}\ar[d]_{\mathit{id}}\ar[r]_{\mathit{I_{+\infty,s_{2}}^{-L_{2}}}} & \mathit{A_{+\infty,-\infty}^{-}}\ar[d]_{\mathit{id}}\ar[r]_{\mathit{\tilde{D}_{+\infty,-\infty}^{-L_{2}}}} & \mathit{A_{+\infty,+\infty}^{-}}\ar[d]^{\mathit{id}}\\
\mathit{A_{+\infty,s_{2}}^{-}}\ar[r]_{\mathit{I_{+\infty,s_{2}}^{-L_{2}}}} & \mathit{A_{+\infty,-\infty}^{-}}\ar[r]_{\mathit{\tilde{D}_{+\infty,-\infty}^{-L_{2}}}} & \mathit{A_{+\infty,+\infty}^{-};}
}
\label{eq:perturbed rect}\\
\xyC{2.5pc}\xyR{2.5pc}\xymatrix{\mathit{A_{s_{1},s_{2}}^{-}}\ar[d]_{\mathit{I_{s_{1},s_{2}}^{-L_{1}}}}\ar[r]^{\mathit{I_{s_{1},s_{2}}^{+L_{2}}}}\ar[dr]|-{\mathit{I_{s_{1},s_{2}}^{-L_{1}\cup+L_{2}}}} & \mathit{A_{s_{1},+\infty}^{-}}\ar[d]^{\mathit{I_{s_{1},+\infty}^{-L_{1}}}}\ar[r]^{\mathit{id}} & \mathit{A_{s_{1},+\infty}^{-}}\ar[d]^{\mathit{I_{s_{1},+\infty}^{-L_{1}}}}\\
\mathit{A_{-\infty,s_{2}}^{-}}\ar[d]_{\mathit{\tilde{D}_{-\infty,s_{2}}^{-L_{1}}}}\ar[r]_{\mathit{I_{-\infty,s_{2}}^{+L_{2}}}}\ar[dr]|-{\mathit{\tilde{F}_{-\infty,s_{2}}^{-L_{1}\cup+L_{2}}}} & \mathit{A_{-\infty,+\infty}^{-}}\ar[d]^{\mathit{\tilde{D}_{-\infty,+\infty}^{-L_{1}}}}\ar[r]^{\mathit{id}} & \mathit{A_{-\infty,+\infty}^{-}}\ar[d]^{\mathit{\tilde{D}_{-\infty,+\infty}^{-L_{1}}}}\\
\mathit{A_{+\infty,s_{2}+\mathrm{lk}}^{-}}\ar[r]_{\mathit{I_{+\infty,s_{2}}^{+L_{2}}}} & \mathit{A_{+\infty,+\infty}^{-}}\ar[r]_{\mathit{id}} & \mathit{A_{+\infty,+\infty}^{-};}
}
 & \xyC{2.5pc}\xyR{2.5pc}\mathit{\xymatrix{\mathit{A_{s_{1},s_{2}}^{-}}\ar[d]_{\mathit{I_{s_{1},s_{2}}^{+L_{1}}}}\ar[r]^{\mathit{I_{s_{1},s_{2}}^{+L_{2}}}}\ar[dr]|-{\mathit{I_{s_{1},s_{2}}^{+L_{1}\cup+L_{2}}}} & \mathit{A_{s_{1},+\infty}^{-}}\ar[d]^{\mathit{I_{s_{1},+\infty}^{+L_{1}}}}\ar[r]^{\mathit{id}} & \mathit{A_{s_{1},+\infty}^{-}}\ar[d]^{\mathit{I_{s_{1},+\infty}^{+L_{1}}}}\\
\mathit{A_{+\infty,s_{2}}^{-}}\ar[d]_{\mathit{id}}\ar[r]_{\mathit{I_{+\infty,s_{2}}^{+L_{2}}}} & \mathit{A_{+\infty,+\infty}^{-}}\ar[d]^{\mathit{id}}\ar[r]^{\mathit{id}} & \mathit{A_{+\infty,+\infty}^{-}}\ar[d]^{\mathit{id}}\\
\mathit{A_{+\infty,s_{2}}^{-}}\ar[r]_{\mathit{I_{+\infty,s_{2}}^{+L_{2}}}} & \mathit{\mathit{A_{+\infty,+\infty}^{-}}}\ar[r]_{\mathit{\mathit{id}}} & \mathit{\mathit{A_{+\infty,+\infty}^{-}.}}
}
}\nonumber
\end{align}

\begin{defn}[Perturbed surgery square]
The above rectangles in Equation \eqref{eq:perturbed rect} are called
\emph{perturbed surgery rectangles} for the two-bridge link $L$. After compressing
them, we get four sets of squares,
\[
\xyR{2pc}\xyC{4pc}\xymatrix{A_{s_{1},s_{2}}^{-}\ar[d]_{\tilde{\Phi}_{s_{1},s_{2}}^{-L_{1}}}\ar[r]^{\tilde{\Phi}_{s_{1},s_{2}}^{-L_{2}}}\ar[dr]|-{\tilde{\Phi}_{s_{1},s_{2}}^{-L_{1}\cup-L_{2}}} & A_{s_{1+\mathrm{lk}},+\infty}^{-}\ar[d]^{\tilde{\Phi}_{s_{1+\mathrm{lk}},+\infty}^{-L_{1}}}\\
A_{+\infty,s_{2}+\mathrm{lk}}^{-}\ar[r]_{\tilde{\Phi}_{+\infty,s_{2}+\mathrm{lk}}^{-L_{2}}} & A_{+\infty,+\infty}^{-};
}
\xyR{2pc}\xyC{4pc}\xymatrix{A_{s_{1},s_{2}}^{-}\ar[d]_{\tilde{\Phi}_{s_{1},s_{2}}^{-L_{1}}}\ar[r]^{\tilde{\Phi}_{s_{1},s_{2}}^{+L_{2}}}\ar[dr]|-{\tilde{\Phi}_{s_{1},s_{2}}^{-L_{1}\cup+L_{2}}} & A_{s_{1},+\infty}^{-}\ar[d]^{\tilde{\Phi}_{s_{1},+\infty}^{-L_{1}}}\\
A_{+\infty,s_{2}+\mathrm{lk}}^{-}\ar[r]_{\tilde{\Phi}_{+\infty,s_{2}+\mathrm{lk}}^{+L_{2}}} & A_{+\infty,+\infty}^{-};
}
\]

\[
\xyR{2pc}\xyC{4.5pc}\xymatrix{A_{s_{1},s_{2}}^{-}\ar[d]_{\tilde{\Phi}_{s_{1},s_{2}}^{+L_{1}}}\ar[r]^{\tilde{\Phi}_{s_{1},s_{2}}^{-L_{2}}}\ar[dr]|-{\tilde{\Phi}_{s_{1},s_{2}}^{+L_{1}\cup-L_{2}}} & A_{s_{1}+\mathrm{lk},+\infty}^{-}\ar[d]^{\tilde{\Phi}_{s_{1}+\mathrm{lk},+\infty}^{+L_{1}}}\\
A_{+\infty,s_{2}}^{-}\ar[r]_{\tilde{\Phi}_{+\infty,s_{2}}^{-L_{2}}} & A_{+\infty,+\infty}^{-};
}
\xyR{2pc}\xyC{4.5pc}\xymatrix{A_{s_{1},s_{2}}^{-}\ar[d]_{\tilde{\Phi}_{s_{1},s_{2}}^{+L_{1}}}\ar[r]^{\tilde{\Phi}_{s_{1},s_{2}}^{+L_{2}}}\ar[dr]|-{\tilde{\Phi}_{s_{1},s_{2}}^{+L_{1}\cup+L_{2}}} & A_{s_{1},+\infty}^{-}\ar[d]^{\tilde{\Phi}_{s_{1},+\infty}^{+L_{1}}}\\
A_{+\infty,s_{2}}^{-}\ar[r]_{\tilde{\Phi}_{+\infty,s_{2}}^{+L_{2}}} & A_{+\infty,+\infty}^{-}.
}
\]
After a $\Lambda$-twisted gluing of the above squares, we obtain
\emph{a perturbed surgery square} $(\mathcal{\tilde{C}}^{-}(\mathcal{H}^{L},\Lambda),\tilde{\mathcal{D}}^{-})$.
\end{defn}

\begin{rem}
In the definition, a perturbed surgery square depends on the choices
of the maps $\tilde{D}_{s_{1},s_{2}}^{-L_{i}},\tilde{F}_{s_{1},-\infty}^{\pm L_{1}\cup-L_{2}},\tilde{F}_{-\infty,s_{2}}^{\pm L_{1}\cup-L_{2}},\tilde{D}_{-\infty,-\infty}^{-L_{1}\cup-L_{2}}$.
However, we will show it is isomorphic to the original square as $\mathbb{F}[[U_{1}]]$-module.
\end{rem}

\subsection{Invariance of the perturbed surgery complex. }

Now we establish the invariance of the perturbed surgery complex for
two-bridge links under the change of edge maps and some diagonal maps
up to chain homotopies.
\begin{prop}
\label{prop:mapping cone iso.}Let $R$ be a $\mathbb{F}$-algebra. Suppose $f,g:A\rightarrow B$ be two chain
maps between two chain complexes of $R$-modules. If $f,g$ are homotopic
to each other by $f\overset{H}{\rightarrow}g$, then the mapping cones
$\mathrm{cone}(f),\mathrm{cone}(g)$ are isomorphic.
\end{prop}

\begin{proof}
We directly construct the isomorphism between the mapping cones $\text{cone}(f)$
and $\text{cone}(g)$. Define $K_{1}:\text{cone}(f)\rightarrow\text{cone}(g),K_{2}:\text{cone}(g)\rightarrow\text{cone}(f)$
by
\begin{eqnarray*}
K_{1}|_{A} & = & id_{A}\oplus H,K_{1}|_{B}=id_{B},\\
K_{2}|_{A} & = & id_{A}\oplus H,K_{2}|_{B}=id_{B}.
\end{eqnarray*}
 In fact, $K_{1}$ is a chain map, since $\forall a\in A,b\in B,$
\begin{eqnarray*}
K_{1}\partial_{f}(a)+\partial_{g}K_{1}(a) & = & K_{1}(\partial_{A}(a)+f(a))+\partial_{g}(a+H(a))\\
 & = & \partial(a)+H\partial_{A}(a)+f(a)+\partial_{A}(a)+g(a)+\partial_{B}H(a)=0,\\
K_{1}\partial_{f}(b)+\partial_{g}K_{1}(b) & = & K_{1}\partial_{B}(b)+\partial_{g}(b)=\partial_{B}(b)+\partial_{B}(b)=0.
\end{eqnarray*}
Moreover, $K_{2}K_{1}$ is $id_{\text{cone}(f)}$, since
\begin{eqnarray*}
K_{2}K_{1}(a) & = & K_{2}(a+H(a))=a+H(a)+H(a)=a,\\
K_{2}K_{1}(b) & = & K_{2}(b)=b,\forall a\in A,b\in B.
\end{eqnarray*}
\[
\xyC{2pc}\xyR{2pc}\xymatrix{A\ar[d]_{f}\ar[r]^{id_{A}}\ar[dr]|-{H} & A\ar[r]^{id_{A}}\ar[d]_{g}\ar[dr]|-{H} & A\ar[d]_{f}\\
B\ar[r]_{id_{B}} & B\ar[r]_{id_{B}} & B
}
\]

\end{proof}
There is a hyperbox version of Proposition \ref{prop:mapping cone iso.}.
\begin{defn}
A hyperbox of chain complexes $R$ is said to be \emph{isomorphic}
to another hyperbox $R'$, if there are chain maps of hyperboxes $F:R\rightarrow R',G:R'\rightarrow R$,
such that $F\circ G=id_{R'},G\circ F=id_{R}$. \end{defn}

\begin{prop}
\label{prop:replacing_maps}Let $R=((C^{\varepsilon})_{\varepsilon\in\mathbb{E}(\text{(d,}1))},(D^{\varepsilon})_{\varepsilon\in\mathbb{E}(n+1)})$
be a hyperbox of chain complexes of size $({\bf d},1)\in\mathbb{Z}_{\geq0}^{n+1}$.
If all the edge maps $D^{({\bf 0},1)}=id$, where ${\bf 0}=(0,...,0)\in\mathbb{Z}^{n}$,
then $R$ induces an isomorphism from the subhyperbox $R^{\varepsilon_{n+1}=0}$
to the subhyperbox $R^{\varepsilon_{n+1}=1}$. \end{prop}

\begin{proof}
We first show the case of hypercubes by induction, i.e., ${\bf d}=(1,...,1)\in\mathbb{Z}^{n}$.

When $n=1$, this is exactly Proposition 5.8. When $n>1$, let us
make some notations at first. There is a $(n-1)$-dimensional subhypercube
corresponding to $\varepsilon_{n}=\varepsilon_{n+1}=0$, denoted by
$R^{00}$, and there is also a $(n-1)$-dimensional subhypercube corresponding
to $\varepsilon_{n}=0,\varepsilon_{n+1}=1$, denoted by $R^{01}$.
Similarly, the subhypercube corresponding to $\varepsilon_{n}=1,\varepsilon_{n+1}=0$
is denoted by $R^{10}$, and the hypercube corresponding to $\varepsilon_{n}=\varepsilon_{n+1}=1$
is denoted by $R^{11}$. Then we can view the hypercube $R$ as the
following square of hypercubes.

\[
\xymatrix{R^{00}\ar[r]^{f}\ar[d]_{h_{1}}\ar[rd]^{H} & R^{10}\ar[d]^{h_{2}}\\
R^{01}\ar[r]_{f'} & R^{11}
}
\]
Notice that $f,f',h_{1},h_{2}$ are chain maps of hypercubes, and
$H$ is a chain homotopy of hypercubes between the chain maps. In
other words, we have
\[
h_{1}\circ D|_{R^{00}}=D|_{R^{10}}\circ h_{1},h_{2}\circ D|_{R^{01}}=D|_{R^{11}}\circ h_{2},\ H\circ D|_{R^{00}}+D|_{R^{11}}\circ H=h_{2}\circ f+f'\circ h_{1}.
\]

By induction, the $(n-1)$-dimensional subhypercube corresponding to $\varepsilon_{n}=0$
induces the isomorphism $h_{1}$. Thus, we have a chain map of hypercubes
$h_{1}^{-1}:R^{01}\rightarrow R^{00}$, such that $h_{1}h_{1}^{-1}=id_{R^{01}}$
and $h_{1}^{-1}h_{1}=id_{R^{00}}.$ Similarly, we have $h_{2}^{-1}:R^{11}\rightarrow R^{10}$
as the inverse of $h_{2}$. The hypercube $R$ induces a chain map
$h_{1}+H+h_{2}$ from the subhypercube $R^{\varepsilon_{n+1}=0}$
to the subhypercube $R^{\varepsilon_{n+1}=1}$. We show that the chain map
of hyperboxes $h_{1}+H+h_{2}:C^{\varepsilon_{n+1}=0}\rightarrow C^{\varepsilon_{n+1}=1}$
is an isomorphism, by directly constructing the inverse $h_{1}^{-1}+h_{2}^{-1}\circ H\circ h_{1}^{-1}+h_{2}^{-1}:R^{\varepsilon_{n+1}=1}\rightarrow R^{\varepsilon_{n+1}=0}$,
which is induced by the following hypercube:
\[
K=\xyR{2pc}\xyC{2pc}\xymatrix{R^{01}\ar[r]^{f'}\ar[rd]|-{\ \ h_{2}^{-1}\circ H\circ h_{1}^{-1}}\ar[d]_{h_{1}^{-1}} & R^{11}\ar[d]^{h_{2}^{-1}}\\
R^{00}\ar[r]_{f} & R^{10}.
}
\]
Here, the map $h_{2}^{-1}\circ H\circ h_{1}^{-1}$ is the
composition of maps $h_{2}^{-1},H,h_{1}^{-1}$.

The following two rectangles of hypercubes show that $h_{1}^{-1}+h_{2}^{-1}Hh_{1}^{-1}+h_{2}^{-1}$
is the inverse of $h_{1}+H+h_{2}$.

\begin{align*}
\xyR{2pc}\xyC{2pc}\xymatrix{C^{00}\ar[r]^{f}\ar[d]_{h_{1}}\ar[rd]^{H} & C^{10}\ar[d]^{h_{2}}\\
C^{01}\ar[r]^{f'}\ar[rd]|-{\ \ h_{2}^{-1}Hh_{1}^{-1}}\ar[d]_{h_{1}^{-1}} & C^{11}\ar[d]^{h_{2}^{-1}}\\
C^{00}\ar[r]_{f} & C^{10}
}
 & \xymatrix{C^{01}\ar[r]^{f'}\ar[rd]|-{\ \ h_{2}^{-1}Hh_{1}^{-1}}\ar[d]_{h_{1}^{-1}} & C^{11}\ar[d]^{h_{2}^{-1}}\\
C^{00}\ar[r]^{f}\ar[d]_{h_{1}}\ar[rd]^{H} & C^{10}\ar[d]^{h_{2}}\\
C^{01}\ar[r]^{f'} & C^{11}
}
\end{align*}

Next, to prove the general case for hyperboxes, we do induction on the size of $H$,
while fixing $n$. We claim that for any $k$ with $1\leq k\leq n-1$, if the
proposition is true for all hyperboxes $R$ of size $({\bf d},1)$ where
${\bf d}=(d_{1},...,d_{k},0,0,...,0)\in\mathbb{Z}_{\geq0}^{n}$, then
the proposition is also true for all hyperboxes $S$ of size $({\bf d}',1)$
where ${\bf d}'=(d_{1}',...,d_{k+1}',0,...,0)\in\mathbb{Z}_{\geq0}^{n}$.

Let $S^{i,j}=S^{\varepsilon_{k+1}=i,\varepsilon_{n+1}=j}$ with $i\in\{0,1,...,d_{k+1}'\},j=\{0,1\}$
be the subhyperbox corresponding to those complexes with $\varepsilon_{k+1}=i,\varepsilon_{n+1}=j$.
Thereby, the subhyperbox $S^{i,j}$ is of size ${\bf d}_{k}'=(d_{1}',...,d_{k}',0,...,0)\in\mathbb{Z}_{\geq0}^{n+1}$.
So we can regard $S^{i,j}$ as a $k$-dimensional hyperbox of size $\bar{{\bf d}}_{k}'=(d_{1}',...,d_{k}')$.
For all $\varepsilon\in\mathbb{E}(\bar{{\bf d}}_{k}')$, we denote
the chain complex of $S$ sitting at $(\varepsilon,i,0,0,...,0,j)$
by $(S^{i,j})^{\varepsilon}$.

We can decompose the hyperbox $S$ as a rectangle of hyperboxes as
the following diagram:

\[
\xyR{2pc}\xyC{2pc}\xymatrix{S^{0,0}\ar[r]^{f_{1}}\ar[d]^{h_{0}}\ar[rd]^{H_{1}} & S^{1,0}\ar[r]^{f_{2}}\ar[d]^{h_{1}}\ar[rd]^{H_{2}} & S^{2,0}\ar[r]^{f_{3}}\ar[d]^{h_{2}}\ar[rd]^{H_{3}} & \cdots\ar[r]^{f_{d_{k+1}'}}\ar[rd]^{H_{d_{k+1}'}} & S^{d_{k+1}',0}\ar[d]^{h_{d_{k+1}'}}\\
S^{0,1}\ar[r]_{f_{1}'} & S^{1,1}\ar[r]_{f_{2}'} & S^{2,1}\ar[r]_{f_{3}'} & \cdots\ar[r]_{f_{d_{k+1}'}'} & S^{d_{k+1}',1},
}
\]
where $f_{1},...,f_{d_{k+1}'},f_{1}',...,f_{d_{k+1}'}',h_{0},...,h_{d_{k+1}'}$
are chain maps of hyperboxes and $H_{1},...,H_{d_{k+1}'}$ are chain
homotopies of hyperboxes.

By the induction hypothesis, the subhyperbox $S^{\varepsilon_{k+1}=j},j\in\{0,1,...,d_{k+1}'\}$
is of size $(d_{1}',...,d_{k}',0,...,0,1)$, and thereby induces the
isomorphism of hyperboxes $h_{j}:S^{j,0}\rightarrow S^{j,1}$. Let
the inverse of $h_{j}$ be $h_{j}^{-1}:S^{j,1}\rightarrow S^{j,0}$.
We define a set of homotopies of hyperboxes $h_{j}^{-1}\circ H_{j}\circ h_{j-1}^{-1}:S^{j-1,1}\rightarrow S^{j,0},$
for any $j\in\{1,2,...,d_{k+1}'\}$ by the following equations, for
all $\varepsilon^{0}\in\mathbb{E}(\bar{{\bf d}}_{k}'),\varepsilon\in\mathbb{E}(k)$
such that $\varepsilon^{0}+\varepsilon\in\mathbb{E}(\bar{{\bf d}}_{k}')$,
\[
(h_{j}^{-1}\circ H_{j}\circ h_{j-1}^{-1})_{\varepsilon^{0}}^{\varepsilon}=\sum_{\{\varepsilon',\varepsilon''\in\mathbb{E}(k)|\varepsilon'\leq\varepsilon''\leq\varepsilon\}}(h_{j}^{-1})_{\varepsilon^{0}+\varepsilon''}^{\varepsilon-\varepsilon''}\circ(H_{j})_{\varepsilon^{0}+\varepsilon'}^{\varepsilon''-\varepsilon'}\circ(h_{j-1}^{-1})_{\varepsilon^{0}}^{\varepsilon'}.
\]

We simply denote $h_{j}^{-1}\circ H_{j}\circ h_{j-1}^{-1}$ by $h_{j}^{-1}H_{j}h_{j-1}^{-1}$.
From the definition of $h_{j}^{-1}H_{j}h_{j-1}^{-1}$, we can show
the associativity of compositions of maps of hyperboxes. Thus, $H_{j}D|_{S^{j-1,0}}+D|_{S^{j,1}}H_{j}=h_{j}f_{j}+f_{j}'h_{j-1}$
and $h_{j}D|_{S^{j,0}}=D|_{S^{j,1}}h_{j}$ implies that
\[
h_{j}^{-1}\circ H_{j}\circ h_{j-1}^{-1}\circ D|_{S^{j-1,1}}+D|_{S^{j,0}}\circ h_{j}^{-1}\circ H_{j}\circ h_{j-1}^{-1}=f_{j}\circ h_{j-1}^{-1}+h_{j}^{-1}\circ f_{j}'.
\]

Therefore, we can construct the following the hypercube $T$
\[
\xyC{4pc}\xyR{2pc}\xymatrix{S^{0,1}\ar[r]^{f_{1}'}\ar[d]|-{h_{0}^{-1}}\ar[rd]|-{h_{1}^{-1}H_{1}h_{0}^{-1}} & S^{1,1}\ar[r]^{f_{2}'}\ar[d]|-{h_{1}^{-1}}\ar[rd]|-{h_{2}^{-1}H_{2}h_{1}^{-1}} & S^{2,1}\ar[r]^{f_{3}'}\ar[d]|-{h_{2}^{-1}}\ar[rd]|-{h_{3}^{-1}H_{3}h_{2}^{-1}\ \ } & \cdots\ar[r]^{f_{d_{k+1}'}'}\ar[rd]|-{h_{d_{k+1}'}^{-1}H_{d_{k+1}'}h_{d_{k+1}'-1}^{-1}} & S^{d_{k+1}',1}\ar[d]^{h_{d_{k+1}'}^{-1}}\\
S^{0,0}\ar[r]_{f_{1}} & S^{1,0}\ar[r]_{f_{2}} & S^{2,0}\ar[r]_{f_{3}} & \cdots\ar[r]_{f_{d_{k+1}'}} & S^{d_{k+1}',0},
}
\]
which induces a chain map from the subhyperbox $S^{\varepsilon_{n+1}=1}$
to the subhyperbox $S^{\varepsilon_{n+1}=0}$. Direct computations
show that the chain maps induced by $S$ and $T$ are the inverses
of each other.
\end{proof}

\begin{cor}
\label{cor:replacing_diagonal}On a rectangle of chain complexes $R=(C^{\varepsilon},D^{\varepsilon})$,
if we change the diagonal maps $D_{\varepsilon}^{(1,1)}$ by a higher
homotopy, i.e. $D'{}_{\varepsilon}^{(1,1)}=D_{\varepsilon}^{(1,1)}+H_{\varepsilon}^{(1,1)}D_{\varepsilon}^{\varepsilon}+D_{\varepsilon+(1,1)}^{\varepsilon+(1,1)}H_{\varepsilon}^{(1,1)}$
with $H_{\varepsilon}^{(1,1)}:C_{*}^{\varepsilon}\rightarrow C_{*+2}^{\varepsilon+(1,1)}$,
then the new rectangle $R'=(C^{\varepsilon},D'^{\varepsilon})$ is
isomorphic to $R$.
\end{cor}

\
\begin{thm}
\label{thm:invariance of perturbed surg. formula}
Suppose $L$ is an oriented two-bridge link with the framing $\Lambda$.
For any $\mathbb{F}[[U_{1}]]$-linear quasi-isomorphisms $\tilde{D}_{s_{1},s_{2}}^{-L_{i}}$
 and $\mathbb{F}[[U_{1}]]$-linear
homotopies $\tilde{F}_{s_{1},s_{2}}^{-L_{1}\cup\pm L_{2}},\tilde{F}_{s_{1},s_{2}}^{\pm L_{1}\cup-L_{2}},$
$\tilde{D}_{-\infty,-\infty}^{-L_{1}\cup-L_{2}}$, the perturbed surgery
complex $(\mathcal{\tilde{C}}^{-}(\mathcal{H}^{L},\Lambda),\tilde{\mathcal{D}}^{-})$
is isomorphic to the original surgery complex in \cite{link_surgery}
as $\mathbb{F}[[U_{1}]]$-module. By imposing the $U_{2}$-action
to be the same as $U_{1}$-action, the $\mathbb{F}[[U_{1},U_{2}]]$-module
$H_{*}(\mathcal{\tilde{C}}^{-}(\mathcal{H}^{L},\Lambda),\tilde{\mathcal{D}}^{-})$
is isomorphic to the homology ${\bf HF}^{-}(S_{\Lambda}^{3}(L))$.
Furthermore, this isomorphism preserves the absolute grading.
\end{thm}

\begin{proof}
First, we restrict our scalars to $\bb{F}[[U_1]].$ By Proposition \ref{prop:replacing_maps}, the following cubes show
that the top square in each cube is isomorphic to the bottom one.
\[
\xyR{2pc}\xyC{2pc}\xymatrix{ & C\ar[ddd]|-{id}\ar[rrr]|-{f_{2}} &  &  & D\ar[ddd]|-{id}\\
A\ar[ddd]|-{id}\ar[rrr]|-{f_{3}}\ar[ur]|-{f_{1}+K\partial_{A}+\partial_{C}K\quad}\ar[urrrr]|-{H+f_{2}K}\ar[ddr]|-{K} &  &  & B\ar[ddd]|-{id}\ar[ur]|-{f_{4}}\\
\\
 & C\ar[rrr]|-{f_{2}} &  &  & D\\
A\ar[rrr]|-{f_{3}}\ar[ur]|-{f_{1}}\ar[urrrr]|-{H} &  &  & B\ar[ur]|-{f_{4}}
} \quad
\xyR{2pc}\xyC{2pc}\xymatrix{ & C\ar[ddd]|-{id}\ar[rrr]|-{f_{2}} &  &  & D\ar[ddd]|-{id}\\
A\ar[ddd]|-{id}\ar[rrr]|-{f_{3}}\ar[ur]|-{f_{1}}\ar[urrrr]|-{H+Kf_{3}} &  &  & B\ar[ddd]|-{id}\ar[ur]|-{\quad f_{4}+K\partial_{B}+\partial_{D}K}\ar[ddr]|-{K}\\
\\
 & C\ar[rrr]|-{f_{2}} &  &  & D\\
A\ar[rrr]|-{f_{3}}\ar[ur]|-{f_{1}}\ar[urrrr]|-{H} &  &  & B\ar[ur]|-{f_{4}}
}
\]
This means when an edge map is changed up to a $\bb{F}[[U_1]]$-linear chain homotopy in a
square $R$, we are able to change the diagonal maps correspondingly
to guarantee the new square is isomorphic to the original one as an $\bb{F}[[U_1]]$-module. By
inductions on the edges of rectangles $\mathfrak{R}_{{\mathrm s},i,j}$ in Equation \eqref{eq:surgery square},
we can show that after changing the edge maps $D_{s_{1},s_{2}}^{-L_{i}}$ by $\tilde{D}_{s_{1},s_{2}}^{-L_{i}}$
and changing some diagonal maps accordingly, the result
rectangle $\mathfrak{R}'_{{\mathrm s},i,j}$ is isomorphic to $\mathfrak{R}_{{\mathrm s},i,j}$ as $\mathbb{F}[[U_{1}]]$-modules.
In fact, we only have changed diagonal maps among the positions of $\tilde{F}_{s_{1},s_{2}}^{\pm L_{1}\cup\pm L_{2}},\tilde{F}_{s_{1},s_{2}}^{\pm L_{1}\cup\pm L_{2}},\tilde{D}_{\pm\infty,\pm\infty}^{\pm L_{1}\cup\pm L_{2}}$
in \eqref{eq:perturbed rect}, where we can keep applying the rigidity results in Proposition \ref{prop:homotopies on the diagonal}.
Thereby, Corollary \ref{cor:replacing_diagonal} implies the perturbed
rectangles in Equation \eqref{eq:perturbed rect} are isomorphic to those rectangles $\mathfrak{R}'_{{\mathrm s},i,j}$'s, and thus
isomorphic to the original $\mathfrak{R}_{{\mathrm s},i,j}$'s in Equation (\ref{eq:surgery square}) as $\mathbb{F}[[U_{1}]]$-modules.
After compressing these rectangles and gluing them together, the perturbed
surgery complex is isomorphic to the original surgery complex as an
$\mathbb{F}[[U_{1}]]$-module.

From Theorem \ref{(Manolescu-Ozsvth-Link-Surgery}, it follows that
the $U_{1},U_{2}$ actions in $H_{*}(\mathcal{C}^{-}(\mathcal{H}^{L},\Lambda),\mathcal{D}^{-})$
are the same. Thus by imposing the $U_{2}$-action as the same as
the $U_{1}$-action on the $\mathbb{F}[[U_{1}]]$-module $H_{*}(\mathcal{\tilde{C}}^{-}(\mathcal{H}^{L},\Lambda),\tilde{\mathcal{D}}^{-})$,
we get an isomorphism as $\mathbb{F}[[U_{1},U_{2}]]$-module between
$H_{*}(\mathcal{\tilde{C}}^{-}(\mathcal{H}^{L},\Lambda),\tilde{\mathcal{D}}^{-})$
and $H_{*}(\mathcal{C}^{-}(\mathcal{H}^{L},\Lambda),\mathcal{D}^{-})$.
As all the rigidity results respect the gradings, the above isomorphism
also preserves the grading.
\end{proof}

\begin{rem}
In the above theorem, the homology of the unknot is $\mathbb{F}[[U_{1},U_{2}]]/(U_{1}-U_{2})$
as an $\mathbb{F}[[U_{1},U_{2}]]$-module. There is no analogue of
the Proposition \ref{prop:homotopies on the diagonal} for homotopies
over the ring $\mathbb{F}[[U_{1},U_{2}]]$. This is why we restrict
our scalars to $\mathbb{F}[[U_{1}]]$. This idea is due to Ciprian
Manolescu.
\end{rem}

\subsection{Algorithm for computing ${\bf HF}^{-}(S_{\Lambda}^{3}(L))$ for two-bridge links.}
Let $L$ be a two-bridge link.

First, we use the algorithm in Section 3.4 to compute all the $A_{\mathrm{s}}^{-}(L)$'s.

Second, by solving linear equations, we find $\bb{F}[[U_1,U_2]]$-linear quasi-homomorphisms
\[
\tilde{D}_{-\infty,s_{2}}^{-L_{1}}:A_{-\infty,s_{2}}^{-}\rightarrow A_{-\infty,s_{2}+\mathrm{lk}}^{-},\ \ \tilde{D}_{s_{1},-\infty}^{-L_{2}}:A_{s_{1},-\infty}^{-}\rightarrow A_{s_{1}+\mathrm{lk},-\infty}^{-}.
\]
Finding chain maps is a problem of solving linear equations modulo
2, which has fast algorithm in the case of sparse matrices. In order to
find a quasi-isomorphism without computing the homology, we adopt an area filtration on the complexes as follows.
From the Schubert diagram $\mathcal{H}^{L}$, one can see that the
diagram $r_{-L_{1}}(\mathcal{H}^{L})$ is isotopic to the standard
genus-$0$ diagram of the unknot with one free basepoint and two intersection
points ${\bf x},{\bf y}$ of attaching curves. Let $C^{\text{u}}$
be the chain complex of $\mathbb{F}[[U_{1},U_{2}]]$-modules generated
by ${\bf x},{\bf y}$, with differential $\partial{\bf x}=(U_{1}-U_{2}){\bf y}$.
Thus there is a chain homotopy equivalence by counting holomorphic
triangles from $A_{\pm\infty,s_{2}}^{-}$ to $C^{\text{u}}$, denoted
by $F:A_{\pm\infty,s_{2}}^{-}\rightarrow C^{\text{u}}$.

Consider the Heegaard diagram by removing $z_{1}$, then an area filtration argument
shows this chain homotopy equivalence $F:A_{+\infty,s_{2}}^{-}\rightarrow C^{\text{u}}$
is in the form of
\begin{align*}
F(b_{0}) & ={\bf x}+\text{lower terms},\\
F(b_{-1}) & ={\bf y}+\text{lower terms,}
\end{align*}
where the lower terms are referred to the area filtration. In fact,
as long as a chain map $G:C^{\text{u}}\rightarrow A_{+\infty,s_{2}}^{-}$
is in the form of
\begin{align*}
G({\bf x}) & =b_{0}+\text{lower terms},\\
G({\bf y}) & =b_{-1}+\text{lower terms},
\end{align*}
then it is a quasi-isomorphism. This is because $F\circ G:C^{\text{u}}\rightarrow C^{\text{u}}$
is in the form of
\begin{align*}
F\circ G({\bf x}) & ={\bf x}+\text{lower terms},\\
F\circ G({\bf y}) & ={\bf y}+\text{lower terms,}
\end{align*}
which is an isomorphism of groups by Lemma 9.10 in \cite{OS_HF1}.
 In order to find an area filtration, we
can set every bigon and square to be of the same area $1$ on the
Schubert Heegaard diagram so that every periodic domain has area $0$.
We can also similarly determine $F$.

Third, we plug in all the maps $I_{\text{s}}^{\overrightarrow{M}}$
and $\tilde{D}^{-L_{i}}$ to (\ref{eq:perturbed rect}). By Corollary \ref{cor:homotopy equilvence}, these
$\tilde{D}^{-L_{i}}$'s are chain homotopic to $D^{-L_{i}}$'s as maps of $\mathbb{F}[[U_1,U_2]]$-modules.
 Thus, following the same line in the proof of Theorem \ref{thm:invariance of perturbed surg. formula}, we can find $\bb{F}[[U_1,U_2]]$-linear diagonal maps $\tilde{F}$'s and $\tilde{D}^{-L_1\cup -L_2}_{-\infty,-\infty}$
 to make those rectangles to be hyperboxes of chain complexes. Finding such maps is also a problem of
 solving linear equations.

Finally, after compressing all the rectangles and doing $\Lambda$-twisted
gluing of these squares, we obtain the perturbed surgery complex $(\mathcal{\tilde{C}}^{-}(\mathcal{H}^{L},\Lambda),\tilde{\mathcal{D}}^{-})$.
Then, we compute the homology over the polynomial ring $\mathbb{F}[[U_1,U_2]]$ and
there are several algorithms of polynomial time. This $\mathbb{F}[[U_1,U_2]]$-module might not be isomorphic to
 ${\bf HF}^{-}(S_{\Lambda}^{3}(L))$. However, by Theorem \ref{thm:invariance of perturbed surg. formula}, as an $\mathbb{F}[[U_1]]$-module, it is isomorphic to ${\bf HF}^{-}(S_{\Lambda}^{3}(L))$. So we impose the $U_2$-action on the homology of $(\mathcal{\tilde{C}}^{-}(\mathcal{H}^{L},\Lambda),\tilde{\mathcal{D}}^{-})$ to be the same as the $U_1$-action. By Theorem \ref{thm:invariance of perturbed surg. formula}, now it is isomorphic to ${\bf HF}^{-}(S_{\Lambda}^{3}(L))$ as an $\mathbb{F}[[U_1,U_2]]$-module.

We note that the surgery complex is infinitely generated over $\mathbb{F}[[U_1,U_2]]$. Hence, before
finding the perturbed surgery complex, we need to do truncations for a fixed framing
matrix $\Lambda$, as described in \cite{link_surgery} Section 8.3.
The time complexity of doing truncations is a polynomial of $\det(\Lambda).$

\begin{rem}
Indeed, in the second and third steps, we only need $\bb{F}[[U_1]]$-linear quasi-isomorphisms $\tilde{D}^{-L_i}$'s and $\bb{F}[[U_1]]$-linear diagonal maps $\tilde{F}$'s and $\tilde{D}^{-L_1\cup -L_2}$ to replace those $\bb{F}[[U_1,U_2]]$-linear maps. The reason we use $\bb{F}[[U_1,U_2]]$-linear maps is that over $\bb{F}[[U_1,U_2]]$ the module $A^-_{\mathrm s}$ is finitely generated and thus easier to use in computer programs.
\end{rem}

\begin{example}
Consider $(0,0)$ surgery on the unlink $L=L_1\cup L_2$ and look at the $(0,0)$ $\mathrm{Spin}^c$ structure $\mathrm{s_0}$. The general Floer
complexes of the unlink $A^-_{\mathrm s}(L)$'s are all $C^{\text u}$, where $C^{\text u}$ is defined in Remark \ref{rem:cx_of_the_unknot}. Since the Alexander gradings $A({\bf{x}})=A({\bf{y}})=(0,0)$, the inclusion maps $I^{\pm L_i}_{\mathrm{s_0}}$ are all the identities for $i=1,2.$
It follows that $\Phi^{+L_i}_{\mathrm s_0}$ and $\Phi^{-L_i}_{\mathrm s_0}$ are chain homotopic by Proposition \ref{prop:homotopy equivalence}. Hence, we can get the perturbed surgery complex for the  $\mathrm{Spin}^c$ structure $\mathrm{s_0}$ as follows
\[
\xymatrix{C^{\text u}\ar[r]^0\ar[d]_0\ar[dr]|-{F} & C^{\text u}\ar[d]^0\\
C^{\text u} \ar[r]_0 &C^{\text u},}
\]
where $F$ is an $\mathbb{F}[[U_1,U_2]]$-linear map shifting the gradings by $1$ satisfying $\partial F=F\partial$. Thus, $F$ can either be $0$ or the following map $f:C^{\text u}\to C^{\text u}$, where $f({\bf x})={\bf y},f({\bf y})=0.$
\begin{description}
\item[Case I]For $F=0$, the homology of the perturbed complex is $\big(\mathbb{F}[[U_1,U_2]]/(U_1-U_2)\big)^{\oplus4}$.
\item[Case II]For $F=f$, the homology is $\big(\mathbb{F}[[U_1,U_2]]/(U_1-U_2)\big)^{\oplus2}\oplus \big(\mathbb{F}[[U_1,U_2]]/(U_1-U_2)^2\big).$
\end{description}
However, as $\mathbb{F}[[U_1]]$-modules, both homology groups are isomorphic to $\mathbb{F}[[U_1]]^{\oplus 4}$. Thus, by imposing the $U_2$-action to be the same as $U_1$, we obtain the correct homology is $\big(\mathbb{F}[[U_1,U_2]]/(U_1-U_2)\big)^{\oplus 4}.$
\end{example}
\subsection{Further discussions of the perturbed surgery complex.}

Besides replacing the maps in a hypercube of chain complexes up to homotopy, we can also replace
the chain complexes sitting at the vertices in the hypercube up to chain homotopy equivalences.
Sometimes this procedure allows us to simplify the computations of the homology of a hypercube.

\begin{lem}
Let $A,\tilde{A},B,\tilde{B}$ be chain complexes. Suppose $h_{A}:A\rightarrow\tilde{A},h_{\tilde{A}}:\tilde{A}\rightarrow A$
are the chain homotopy equivalences, with $h_{A}h_{\tilde{A}}\overset{K_{\tilde{A}}}{\simeq}id_{\tilde{A}},h_{\tilde{A}}h_{A}\overset{K_{A}}{\simeq}id_{A}.$
Similarly, we have the chain homotopy equivalences $h_{B},h_{\tilde{B}}$
and the chain homotopies $K_{B},K_{\tilde{B}}$ with $h_{B}h_{\tilde{B}}\overset{K_{\tilde{B}}}{\simeq}id_{\tilde{B}},h_{\tilde{B}}h_{B}\overset{K_{B}}{\simeq}id_{B}.$
Let $f:A\rightarrow B$ be any chain map. Then, the mapping cones
$\mathrm{cone}(f)$ and $\mathrm{cone}(h_{B}fh_{\tilde{A}})$ are
chain homotopy equivalent via the chain maps $H_1:\mathrm{cone}(f)\to\mathrm{cone}(h_{B}fh_{\tilde{A}}),H_2:\mathrm{cone}(h_{B}fh_{\tilde{A}})\to \mathrm{cone}(f)$ induced by the following squares of chain
complexes
\[
H_{1}:=\xymatrix{A\ar[d]_{h_{A}}\ar[r]^{f}\ar[rd]|-{h_{B}fK_{A}} & B\ar[d]^{h_{B}}\\
\tilde{A}\ar[r]_{h_{B}fh_{\tilde{A}}} & \tilde{B};
}
H_{2}:=\xymatrix{\tilde{A}\ar[r]^{h_{B}fh_{\tilde{A}}}\ar[rd]|-{K_{B}fh_{\tilde{A}}}\ar[d]_{h_{\tilde{A}}} & \tilde{B}\ar[d]^{h_{\tilde{B}}}\\
A\ar[r]_{f} & B.
}
\]
\end{lem}

\begin{proof}
We directly compute the compositions  $H_{1}\circ H_{2}$
and $H_{2}\circ H_{1}$ to check that they are both chain homotopic to the identities.
The composition $H_{2}\circ H_{1}$ is the compression of the juxtaposition
of the two squares, which is the square
\[
H_{2}\circ H_{1}=\xymatrix{A\ar[r]^{f}\ar[d]_{h_{\tilde{A}}h_{A}}\ar[rd]|-{\tilde{F}} & B\ar[d]^{h_{\tilde{B}}h_{B}}\\
A\ar[r]_{f} & B
}
\]
where $\tilde{F}=K_{B}fh_{\tilde{A}}h_{A}+h_{\tilde{B}}h_{B}fK_{A}$.
The following cube shows that $H_{2}\circ H_{1}$ is homotopic to
$id_{\mathrm{cone}(f)}$
\[
\xymatrix{ & A\ar[ddd]|-{id_{A}}\ar[rrr]^{f}\ar[rrd]|-{\tilde{F}}\ar[ld]|-{h_{\tilde{A}}h_{A}}\ar[ldddd]|-{K_{A}}\ar[ddddrr]|-{K_{B}fK_{A}} &  &  & B\ar[ddd]|-{id_{B}}\ar[ld]|-{h_{\tilde{B}}h_{B}}\ar[ddddl]|-{K_{B}}\\
A\ar[ddd]|-{id_{A}}\ar[rrr]|-{f} &  &  & B\ar[ddd]|-{id_{B}}\\
\\
 & A\ar[ld]|-{id_{A}}\ar[rrr]|-{f} &  &  & B\ar[ld]|-{id_{B}}\\
A\ar[rrr]|-{f} &  &  & B
}
\]
where $\tilde{F}=K_{B}fh_{\tilde{A}}h_{A}+h_{\tilde{B}}h_{B}fK_{A}.$
Direct computation verifies that the above cube is a hypercube chain
complex. We only check the longest diagonal here:
\begin{align*}
K_{B}f+fK_{A}+K_{B}fh_{\tilde{A}}h_{A}+h_{\tilde{B}}h_{B}fK_{A} & =K_{B}f(\partial_{A}K_{A}+K_{A}\partial_{A})+(\partial_{B}K_{B}+K_{B}\partial_{B})fK_{A})\\
 & =K_{B}fK_{A}\partial_{A}+\partial_{B}K_{B}fK_{A}.
\end{align*}
Similarly, $H_{1}\circ H_{2}$ is chain homotopic to $id_{\mathrm{cone}(h_{B}fh_{\tilde{A}})}.$
\end{proof}

Now we generalize this lemma to a hypercube version. The proof is similar, so we omit it.

\begin{prop}
Suppose $A$ and $\tilde{A}$ are two chain homotopy equivalent $n$-dimensional
hypercubes, and so do $B$ and $\tilde{B}$. Then the $(n+1)$-dimensional hypercube $\mathrm{cone}(A\xrightarrow{f}B)$
is chain homotopy equivalent to $\mathrm{cone}(\tilde{A}\xrightarrow{h_{B}fh_{\tilde{A}}}\tilde{B})$,
where $h_{B}:B\rightarrow\tilde{B},h_{\tilde{A}}:\tilde{A}\rightarrow A$
are chain homotopy equivalences.
\end{prop}
Iterating these conjugation constructions, we have the following proposition.
\begin{prop}
Let $H=(C^{\varepsilon},D^{\varepsilon})$ be a $n$-dimensional hypercube
of chain complexes. Suppose we have that $\tilde{C}^{\varepsilon}$
is chain homotopy equivalent to $C^{\varepsilon}$ for all $\varepsilon\in\mathbb{E}_{n}$.
Then there exists a hypercube $\tilde{H}=(\tilde{C}^{\varepsilon},\tilde{D}^{\varepsilon})$
which is chain homotopy equivalent to $H$. \end{prop}

For the purpose of this paper, we give an example of the 2-dimensional case.

\begin{example}
Let $H$ be the following square of chain complexes
\[
H=\xymatrix{C_{1}\ar[r]^{f_{1}}\ar[d]_{g_{1}}\ar[rd]|-{F} & C_{2}\ar[d]^{f_{2}}\\
C_{3}\ar[r]_{g_{2}} & C_{4}.
}
\]
Then suppose we have a set of chain homotopy equivalences $h_{i}:C_{i}\to\tilde{C}_{i},\tilde{h}_{i}:\tilde{C}_{i}\to C_{i}$,
where $h_{i}\circ\tilde{h_{i}}$ is homotopic to $id_{\tilde{C}_{i}}$
via $\tilde{K}_i$ and $\tilde{h}_{i}\circ h_{i}$ is homotopic
to $id_{C_{i}}$ via $K_{i}.$ Then compressing the following rectangle,
we obtain the desired square $\tilde{H}$.

\[
\xymatrix{\tilde{C}_{1}\ar[r]^{\tilde{h}_{1}}\ar[d]_{h_{3}g_{1}\tilde{h}_{1}}\ar[rd]|-{K_{3}g_{1}\tilde{h}_{1}} & C_{1}\ar[r]^{f_{1}}\ar[d]^{g_{1}}\ar[rd]|-{F} & C_{2}\ar[d]_{f_{2}}\ar[r]^{h_{2}}\ar[rd]|-{h_{4}f_{2}K_{2}} & \tilde{C_{2}}\ar[d]^{h_{4}f_{2}\tilde{h}_{2}}\\
\tilde{C}_{3}\ar[r]_{\tilde{h}_{3}} & C_{3}\ar[r]_{g_{2}} & C_{4}\ar[r]_{h_{4}} & \tilde{C_{4}}.
}
\]
Thus,
\[
\tilde{H}=\xymatrix{\tilde{C_{1}}\ar[r]^{h_{2}f_{1}\tilde{h}_{1}}\ar[d]_{h_{3}g_{1}\tilde{h}_{1}}\ar[rd]|-{\tilde{F}} & \tilde{C}_{2}\ar[d]^{h_{4}f_{2}\tilde{h}_{2}}\\
\tilde{C}_{3}\ar[r]_{h_{4}g_{2}\tilde{h}_{3}} & \tilde{C}_{4}.
}
\]
where $\tilde{F}=h_{4}g_{2}K_{3}g_{1}\tilde{h}_{1}+h_{4}F\tilde{h}_{1}+h_{4}f_{2}K_{2}f_{1}\tilde{h}_{1}.$ \end{example}

If we know the chain homotopy types of all the $A^-_{\mathrm{s}}$'s, we can also replace the chain complexes in the perturbed surgery
complex by the conjugation construction. We will still call it the perturbed surgery complex.  This is used in simplifying the computations in the next section.
\section{Examples}

\subsection{The complexes $\widehat{CFL}(L)$ for two-bridge links $L$.}

We recall from \cite{OS_link_FLoer} that for a link $L$, the filtered
chain complex $\widehat{CFL}(L)$ is a chain complex of $S^{3}$ with
a filtration induced from $L$. More precisely, fixing a Heegaard
diagram $\mathcal{H}^{L}$ of $L\subset S^{3}$, we obtain a chain
complex of $\mathbb{F}$-modules $\widehat{CF}(\mathcal{H}^{L})$,
generated by the intersection points of $\mathbb{T_{\alpha}}$ and
$\mathbb{T}_{\beta}$ in the symmetric product. There is an Alexander
filtration on $\widehat{CF}(\mathcal{H}^{L}).$ It is shown that given
different Heegaard diagrams of $L$, $\mathcal{H}_{1}^{L}$ and $\mathcal{H}_{2}^{L}$,
there is a chain homotopy equivalence from $\widehat{CF}(\mathcal{H}_{1}^{L})$
to $\widehat{CF}(\mathcal{H}_{2}^{L})$, which preserves the Alexander
filtration. Thus, the filtered chain homotopy equivalence class of
these chain complexes is called the \emph{filtered chain homotopy
type} of $\widehat{CFL}(L)$. By abuse of notation, we also let $\widehat{CFL}(L)$
be some filtered complex in this equivalence class. Similarly, we define
the filtered chain homotopy type of $CFL^{-}(L)$, by looking at the
Alexander filtered chain complex $CF^{-}(\mathcal{H}^{L}).$

We represent $\mathbb{Z}^{2}$-filtered complexes graphically by dots
and arrows on the $x$-$y$ coordinate plane, with the dots representing
generators, the arrows representing differentials, and the coordinates
representing filtrations.

\begin{thm}[Theorem 12.1 in \cite{OS_link_FLoer}]
\label{thm:(Ozsvth-Szab)}
Suppose $\overrightarrow{L}=\overrightarrow{L_{1}}\cup\overrightarrow{L_{2}}$
is an oriented two-component alternating link. Then the filtered chain
homotopy type of $\widehat{CFL}(L)$ is determined by the following
data:

(1)the multi-variable Alexander polynomial of L, $\Delta_{L}$;

(2)the signature of L, $\sigma(L),$ and the linking number of $L$,
$\mathrm{lk}(L)$;

(3)the filtered chain homotopy type of $\widehat{CFK}(L_{1})$ and
$\widehat{CFK}(L_{2})$.
\end{thm}

In fact, for alternating two-component links, $\widehat{CFL}(L)$
is filtered chain homotopy equivalent to a simplified filtered chain
complex $\widehat{CFL}_{\mathrm{OS}}(L)$.  The simplified complex
is a direct sum of five different types of $\mathbb{Z}\oplus\mathbb{Z}$-filtered
chain complexes $B_{(d)}[i,j]$,$H_{(d)}^{l}[i,j]$,$V_{(d)}^{l}[i,j]$,$X_{(d)}^{l}[i,j]$
and $Y_{(d)}^{l}[i,j]$. These basic filtered complexes are described
in Section 12.1 of \cite{OS_link_FLoer}; the filtered complex $B_{(d)}[i,j]$
looks like a box and the others look like zigzags. See Figure \ref{B,H,V,X,Y}.

\begin{centering}
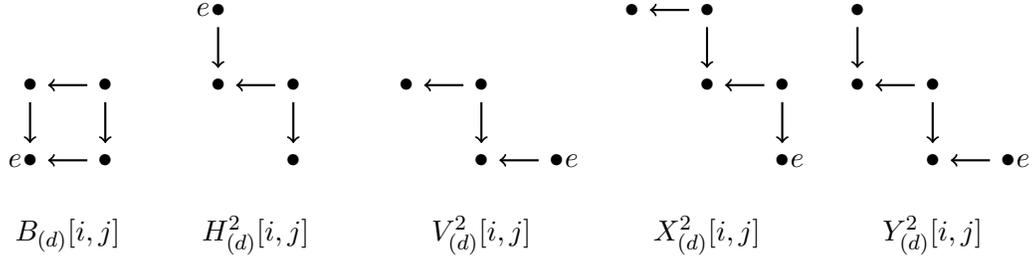
\begin{figure}
\begin{tikzpicture}[thick]
\node (b1) at ( -6,1) {$\bullet$};
\node (b2) at ( -6,0) {$\bullet$}
  edge [<-]  (b1);
\node at (-6.2,0) {$e$};
\node (b4) at ( -5,1) {$\bullet$}
  edge [->] (b1);
\node (b3) at ( -5,0) {$\bullet$}
  edge [->] (b2)
  edge [<-] (b4);
\node (b0) at (-5.5,-1) {$B_{(d)}[i,j]$};

\node (c1) at ( -3.5,2) {$\bullet$};
\node (c2) at ( -3.5,1) {$\bullet$}
  edge [<-]  (c1);
\node (c3) at ( -2.5,1) {$\bullet$}
  edge [->] (c2);
\node (c4) at ( -2.5,0) {$\bullet$}
  edge [<-] (c3);
\node at (-3.7,2) {$e$};
\node  (c0) at (-3,-1) {$H_{(d)}^2[i,j]$};

\node (d1) at ( -1,1) {$\bullet$};
\node (d2) at ( 0,1) {$\bullet$}
  edge [->]  (d1);
\node (d3) at ( 0,0) {$\bullet$}
  edge [<-] (d2);
\node (d4) at ( 1,0) {$\bullet$}
 edge [->] (d3);
\node at (1.2,0) {$e$};
\node (d0) at (0,-1) {$V^2_{(d)}[i,j]$};

\node (e1) at ( 2,2) {$\bullet$};
\node (e2) at ( 3,2) {$\bullet$}
  edge [->]  (e1);
\node (e3) at ( 3,1) {$\bullet$}
  edge [<-] (e2);
\node (e4) at ( 4,1) {$\bullet$}
  edge [->] (e3);
\node (e5) at  (4,0) {$\bullet$}
  edge [<-] (e4);
\node at (4.2,0) {$e$};
\node (e0) at (3,-1) {$X^2_{(d)}[i,j]$};

\node (f1) at ( 5,2) {$\bullet$};
\node (f2) at ( 5,1) {$\bullet$}
  edge [<-]  (f1);
\node (f3) at ( 6,1) {$\bullet$}
  edge [->] (f2);
\node (f4) at ( 6,0) {$\bullet$}
  edge [<-] (f3);
\node (f5) at  (7,0) {$\bullet$}
  edge [->] (f4);
\node at (7.2,0) {$e$};
\node (f0) at (6,-1) {$Y^2_{(d)}[i,j]$};
\end{tikzpicture}
\caption{{\bf{Examples of the}} $\mathbb{Z}^2${\bf{-filtered chain complexes $B,H,V,X,Y$.}}   The labelled dots $e$ in $B_{(d)}[i,j],H_{(d)}^l[i,j],V_{(d)}^l[i,j],X_{(d)}^l[i,j],Y_{(d)}^l[i,j]$ are of  grading $d$ and  with the filtrations $(i,j),(i,j),(i,j),(i+l,j),(i+l,j)$ respectively.}

\label{B,H,V,X,Y}
\end{figure}
\end{centering}

\begin{cor}
\label{cor:filtered_type_CFLhat_2-bri}
If L is an oriented two-bridge link, then the filtered homotopy type
of $\widehat{CFL}(L)$ is determined by $\sigma(L)$, $\mathrm{lk}(L)$
and the multi-variable Alexander polynomial $\Delta_{L}(x,y)$. More
concretely, let
\[
l=\mathrm{lk}(\overrightarrow{L})+\frac{\sigma(\overrightarrow{L})-1}{2},
\]

(1) if $l\geq0$, let $a=\frac{1-\sigma-\mathrm{lk}}{2},\ b=\frac{-1-\sigma-\mathrm{lk}}{2}$,
then we have that $\widehat{CFL}(L)$ is filtered chain homotopic
to
\[
Y_{(0)}^{l}[a,a]\oplus Y_{(-1)}^{l+1}[b,b]\oplus\underset{k}{\bigoplus}B_{(d_{k})}[i_{k},j_{k}],
\]
where those $d_{k},i_{k},j_{k}$'s are determined by the Alexander
polynomial $\Delta_{L}$;

(2) if $l<0$, then we have that the $\widehat{CFL}(L)$ is filtered
chain homotopic to
\[
X_{(0)}^{|l|}[\frac{\mathrm{lk}}{2},\frac{\mathrm{lk}}{2}]\oplus X_{(-1)}^{|l|-1}[\frac{\mathrm{lk}}{2},\frac{\mathrm{lk}}{2}]\oplus\underset{k}{\bigoplus}B_{(d_{k})}[i_{k},j_{k}],
\]
 where those $d_{k},i_{k},j_{k}$'s are determined by the Alexander
polynomial $\Delta_{L}$.
\end{cor}

\begin{example}
Let $\mathit{Wh}$ denote the Whitehead link. Since $\mathrm{lk}(\mathit{Wh})=0$,
$\sigma(\mathit{Wh})=-1$, we get $l=-1$. Notice that $\mathrm{lk}=0$
implies that the signature doesn't depend on the orientations of the
link. Thus the filtered chain homotopy type of $\widehat{CFL}(\mathit{Wh})$
is
\[
X_{(0)}^{1}[0,0]\oplus X_{(-1)}^{0}[0,0]\oplus\underset{k}{\bigoplus}B_{(d_{k})}[i_{k},j_{k}],
\]
 where those $d_{k},i_{k},j_{k}$ are determined by the Alexander
polynomial. If we consider the mirror of $\mathit{Wh}$, we have $\sigma(\overline{\mathit{Wh}})=1$.
Similarly, the filtered chain homotopy type of $\widehat{CFL}(\overline{\mathit{Wh}})$
is
\[
Y_{(0)}^{0}[0,0]\oplus Y_{(-1)}^{1}[-1,-1]\oplus\underset{k}{\bigoplus}B_{(d_{k}')}[i_{k}',j_{k}'].
\]
In the following diagram, $\widehat{CFL}(\mathit{Wh})$ and $\widehat{CFL}(\overline{\mathit{Wh}})$
are illustrated, where each dot represents a generator and each arrow
represents a differential.

\begin{equation}
\xyR{1pc}\xyC{1pc}\sigma(\overline{\mathit{Wh}})=1:\xymatrix{\bullet\ar[d] & \mbox{\ensuremath{\bullet}\ \ensuremath{\bullet}}\ar@<1ex>[d]\ar@<-1ex>[d]\ar[l] & \bullet\ar[l]\ar[d]\\
{\displaystyle \underset{{\displaystyle \bullet}}{\bullet}}\ar[d] & {\displaystyle \underset{{\displaystyle \bullet}\ {\displaystyle \bullet}}{\bullet\ \bullet}}\ar[l]\ar@<1ex>[d] & {\displaystyle \underset{{\displaystyle \bullet}}{\bullet}}\ar@<1.5ex>[l]\ar[l]\ar[d]\\
\bullet & \bullet\ \bullet\ar[l] & \bullet,\ar[l]
}
\ \ \ \ \ \ \ \ \sigma(\mathit{Wh})=-1:\xymatrix{\bullet\ar[d] & \mbox{\ensuremath{\bullet}\ \ensuremath{\bullet}}\ar@<-1ex>[d]\ar[l] & \bullet\ar[l]\ar[d]\\
{\displaystyle \underset{{\displaystyle \bullet}}{\bullet}}\ar[d] & {\displaystyle \underset{{\displaystyle \bullet}\ {\displaystyle \bullet}}{\bullet\ \bullet}}\ar[l]\ar@<1ex>[d]\ar@<1.5ex>[l]\ar@<-1ex>[d] & {\displaystyle \underset{{\displaystyle \bullet}}{\bullet}}\ar@<1.5ex>[l]\ar[d]\\
\bullet & \bullet\ \bullet\ar[l] & \bullet.\ar[l]
}
\label{eq:CFL_Hat_Whitehead}
\end{equation}
We find that it is easier to work with $\mathit{Wh}$ with $\sigma=-1$,
since all the $A_{\mathrm{s}}^{-}(\mathit{Wh})$ have homology $\mathbb{F}[[U_{1},U_{2}]]/(U_{1}-U_{2})$,
so that we can apply the rigidity results. (One can compare this to
the case of the right-handed trefoil knot versus the left-handed trefoil
knot.) \end{example}

\begin{rem}
The way of decomposing the complex $\widehat{CFL}(\overrightarrow{L})$
into direct sums of $B,H,V,X,Y$ is not canonical. We can do some
base-changes to change the above direct sum decomposition of $\widehat{CFL}$
such that the patterns of the arrows don't change.
\end{rem}

\subsection{The filtered homotopy type of $CFL^{-}(L)$ for some two-bridge links.}

\begin{figure}
\centering

\includegraphics[scale=0.5]{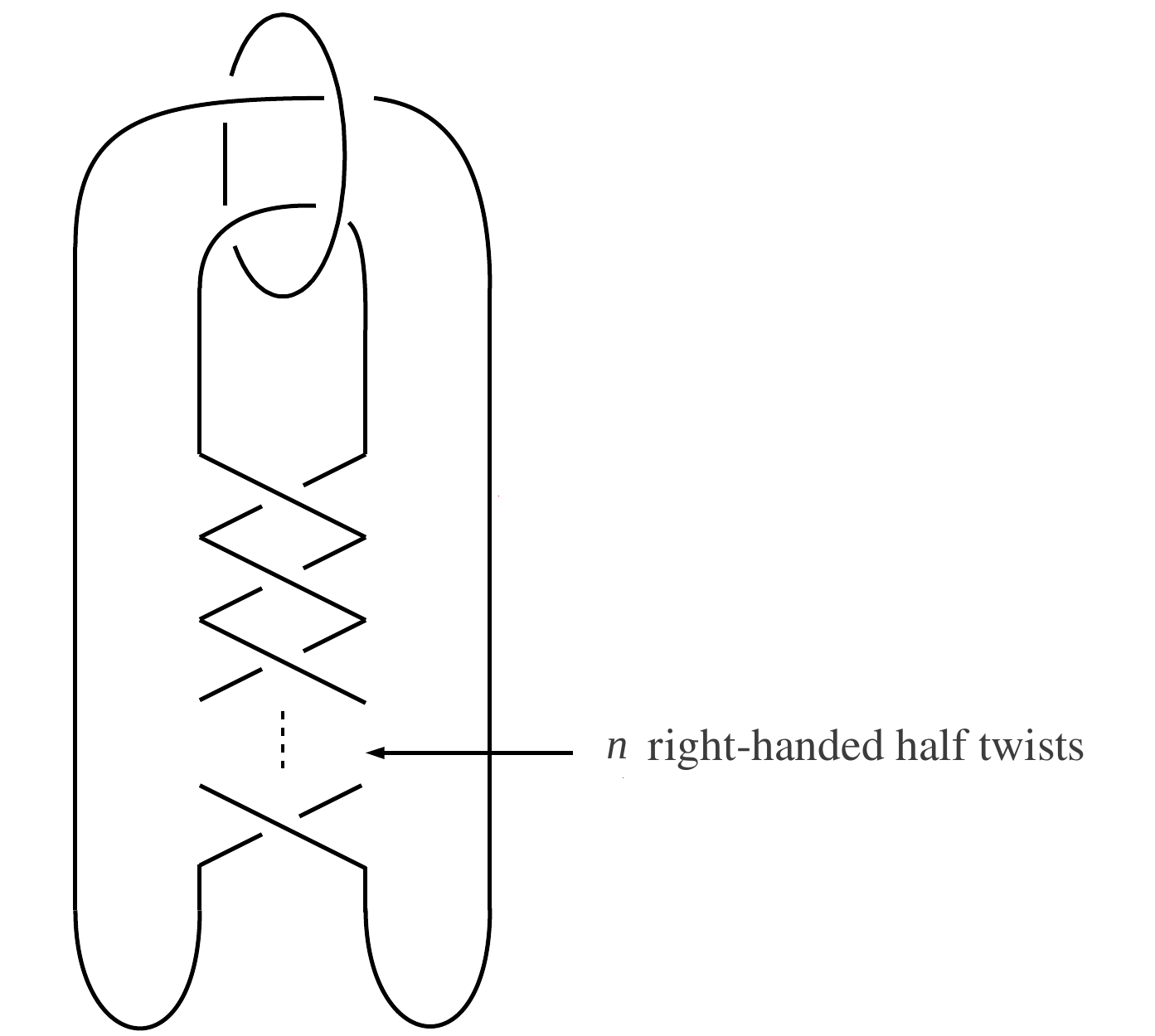}\caption{\textbf{The two-bridge links} $b(4n+4,2n+3)$\textbf{.} If $n=2k-1$,
then the linking number is 0; if $n=2k$, then the linking number
is $2$.}

\label{b(4n,2n+1)}

\end{figure}

Given a two-bridge link $L$, we can use the Schubert Heegaard diagram
to combinatorially find the filtered complex $CFL^{-}(L)$. However,
this description is too cumbersome. Instead, here we use algebraic arguments
to determine the filtered homotopy type of $CFL^{-}(L)$ in some special
examples.

Consider the Schubert Heegaard diagram $\mathcal{H}$ for $L$ and
the $\mathbb{Z}^{2}$-filtered chain complex $\widehat{CF}(\mathcal{H})$.
Then there is a filtered chain homotopy equivalence $F:\widehat{CF}(\mathcal{H})\rightarrow\widehat{CFL}_{\mathrm{OS}}(L)$.
Thus, $F$ induces an isomorphism on the homology of their associated
graded, i.e. the link Floer homology. In fact, the homology of their
associated graded are just the chain groups themselves, because $\widehat{CF}(\mathcal{H})$
and $\widehat{CFL}_{\mathrm{OS}}(L)$ are both thin (with no differentials
in their associated graded). So $F$ is an isomorphism. In other words,
we can change the basis of $\widehat{CF}(\mathcal{H})$ preserving
the filtration, such that the arrows are pruned as in the Ozsv\'{a}th-Szab\'{o}
simplified pattern. We use this new basis to consider $CF^{-}(\mathcal{H})$.
Since every bigon in $\mathcal{H}$ contains a basepoint, those $\partial_{U_{1},U_{2}}$
arrows in $CF^{-}(\mathcal{H})$ are either upward or rightward of
length 1. This property is repeatedly used later.

In this section, we show that for the two-bridge links $b(4n,2n+1)$,
the filtered homotopy type of $CFL^{-}(L)$ is determined by $\widehat{CFL}(L)$.
Since $\widehat{CFL}(L)$ can be decomposed as direct sums of $B,X,Y$'s,
our goal is to show that $CFL^{-}(L)$ can be viewed as a square of
chain complexes of  these $B,X,Y$'s.

Using continuous fractions, we can get the 4-plat presentations of
$b(4n,2n+1)$, thus providing the diagram in Figure \ref{b(4n,2n+1)}.
In addition, there is a convention issue of signs of the signature.
We adopt the convention compatible with Corollary \ref{cor:filtered_type_CFLhat_2-bri},
so that $\sigma\big(b(8k,4k+1)\big)=-1$.

\begin{prop}
\label{prop:CFL^-(b(8k,4k+1))}For the two-bridge link $L=b(8k,4k+1)$,
the filtered homotopy type of $CFL^{-}(L)$ is determined by the Alexander
polynomial, signature and linking number of $L$, or equivalently
by $\widehat{CFL}(L)$. Precisely, we have $CFL^{-}(L)=CFL^{-}(\mathit{Wh})\oplus\bigoplus_{i=1}^{k-1}(N,\partial^{-})$,
where $CFL^{-}(\mathit{Wh})$ and $(N,\partial^{-})$ are described
in Figure \ref{CFL^-(Wh)}.

\begin{figure}
\centering

\includegraphics[scale=0.7]{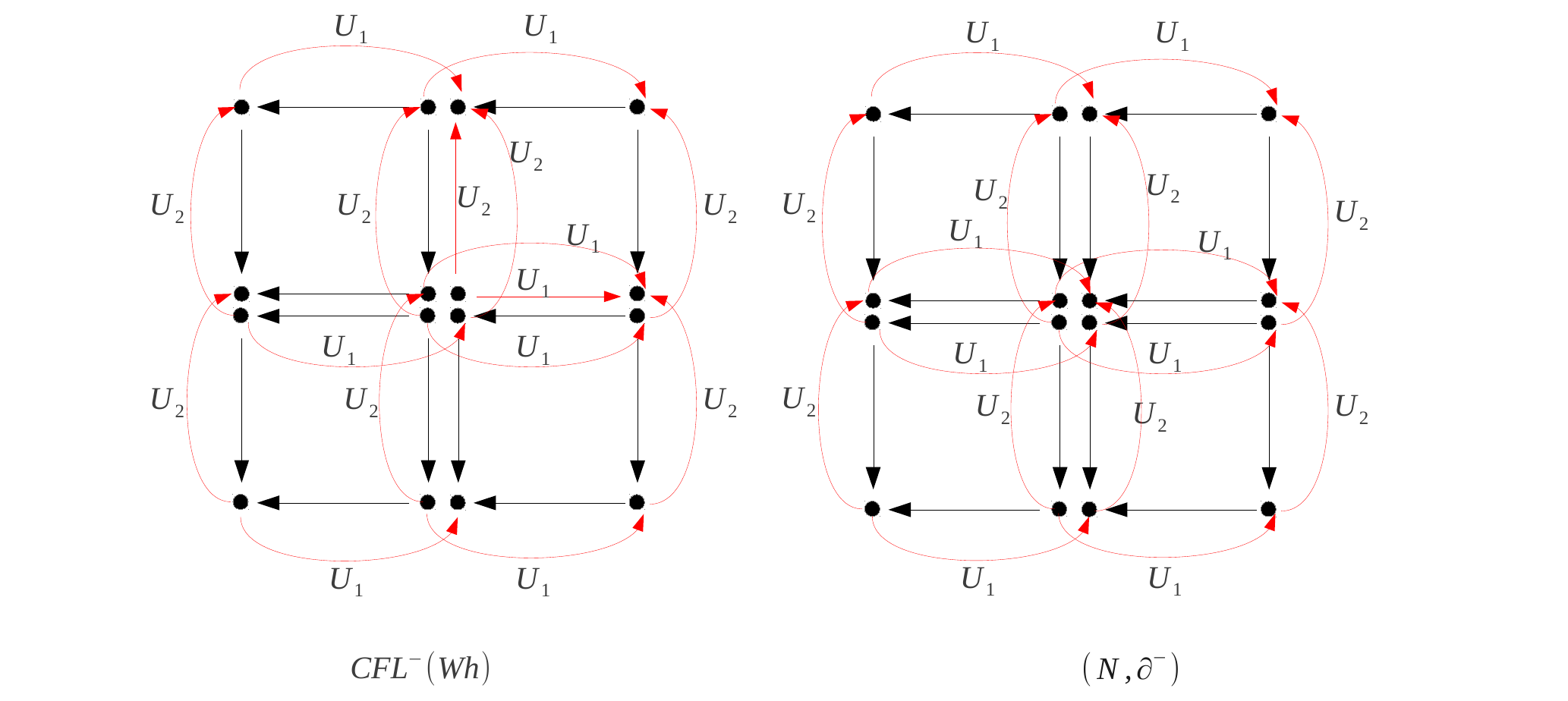}\caption{\textbf{The filtered complex $CFL^{-}(\mathit{Wh})$ and }$(N,\partial^{-})$\textbf{.}
The horizontal red arrows and vertical red arrows have $U_{1}$ and
$U_{2}$ coefficients respectively. }
\label{CFL^-(Wh)}

\end{figure}
\end{prop}

\begin{proof}
By Theorem \ref{thm:(Ozsvth-Szab)}, the Ozsv\'{a}th-Szab\'{o} simplified
complex $\widehat{CFL}_{\mathrm{OS}}(L)$ can be computed in terms
of $\Delta_{L}(x,y)=k\Delta_{\mathit{Wh}}(x,y)=k\frac{(x-1)(y-1)}{\sqrt{xy}}$,
$\mathrm{lk}=0$, and $\sigma(L)=-1$.
Then, we compute that
\[
\widehat{CFL}_{\mathrm{OS}}(L)=A\oplus B\oplus C\oplus D=\bigoplus_{i=1}^{k}A^{(i)}\oplus\bigoplus_{i=1}^{k}B^{(i)}\oplus\bigoplus_{i=1}^{k}C^{(i)}\oplus\bigoplus_{i=1}^{k}D^{(i)}.
\]
See Figure \ref{ABCD} for the filtered homotopy type of $A^{(i)},B^{(i)},C^{(i)},D^{(i)}$,
where we denote the generators in $A^{(i)},B^{(i)},C^{(i)},D^{(i)}$
by $a_{j}^{(i)},b_{j}^{(i)},c_{j}^{(i)},d_{j}^{(i)},$ $j=1,2,3,4,$
respectively.

\begin{figure}
\centering

\includegraphics[scale=0.6]{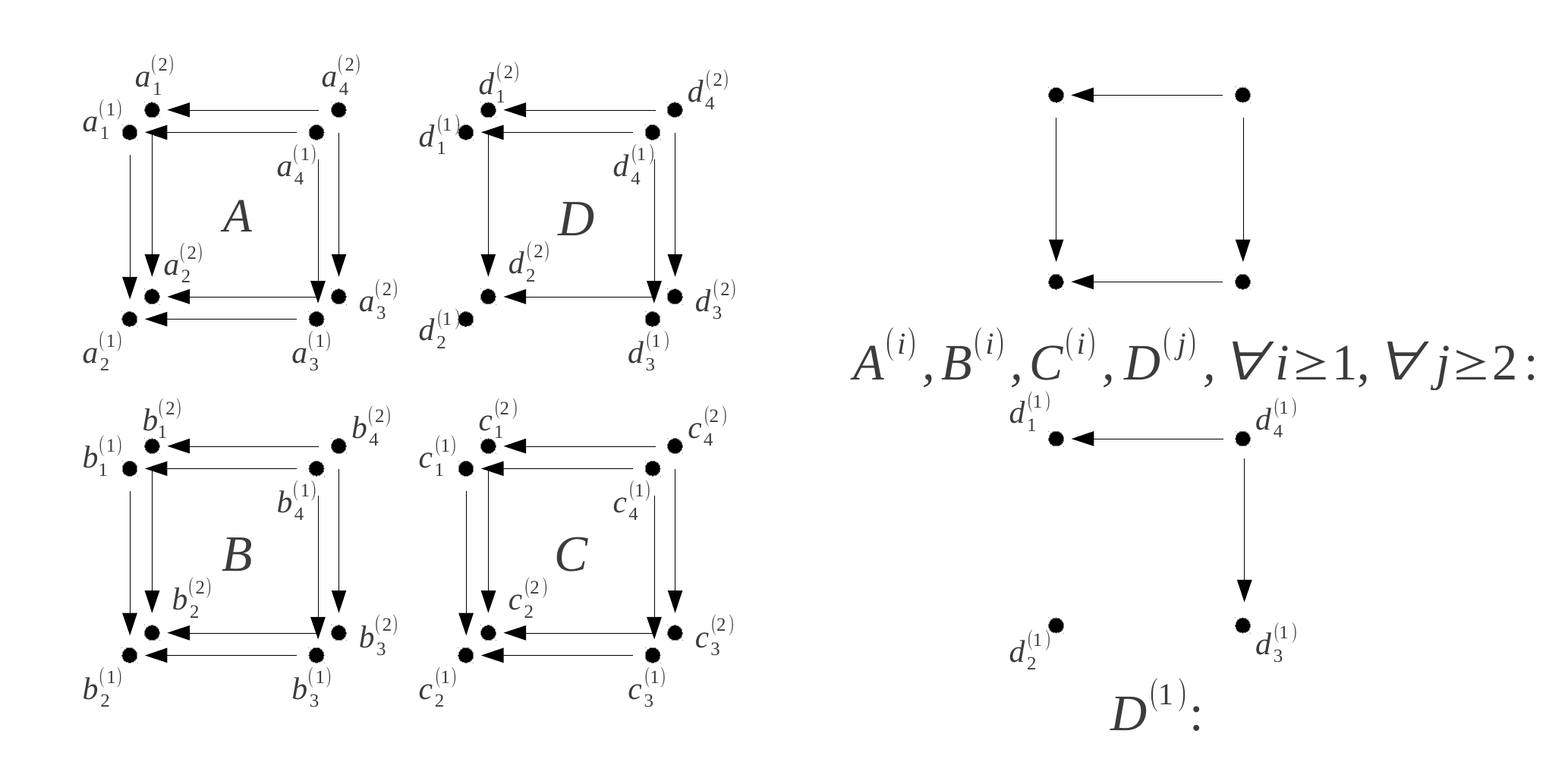}

\caption{$\widehat{CFL}_{\mathrm{OS}}(b(8k,4k\pm1))$\textbf{.} On the left
side, the figure illustrates the Alexander grading of $A,B,C,D$ summands,
where $k=2$. On the right side, it indicates the filtered homotopy
types of $(A^{(i)},\hat{\partial}),(B^{(i)},\hat{\partial}),(C^{(i)},\hat{\partial}),(D^{(i)},\hat{\partial})$,
which all have the filtered homotopy types as boxes, except for $(D^{(1)},\hat{\partial}).$ }

\label{ABCD}

\end{figure}

Given $\widehat{CFL}_{\mathrm{OS}}(L)$, let us investigate the possibilities
for $CFL^{-}(L).$ The differential $\partial^{-}$ in $CFL^{-}(L)$
decomposes into
\[
\partial^{-}=\hat{\partial}+\partial_{U_{1},U_{2}}=\partial_{A_{1}}+\partial_{A_{2}}+\partial_{U_{1}}+\partial_{U_{2}},
\]
where $\partial_{U_{1},U_{2}}(x)=\partial_{U_{1}}(x)+\partial_{U_{2}}(x)$
consists of the components in $\partial^{-}(x)$ with coefficients
of $U_{1},U_{2}$ powers, and $\hat{\partial}(x)=\partial_{A_{1}}(x)+\partial_{A_{2}}(x)$
is decomposed by the Alexander filtration. As stated before, here
$\partial_{U_{i}}$ has the form of $\partial_{U_{i}}(x)=U_{i}y$
for $i=1,2,$  i.e. the $\partial_{U_{i}}$-arrows are all of length
1. A close examination of $U_{1},U_{2}$ powers and the Alexander
filtrations in the coefficients of the following identity provides
that
\begin{align*}
0=(\partial^{-})^{2}=(\hat{\partial}+\partial_{U_{1},U_{2}})^{2}=(\partial_{A_{1}}+\partial_{A_{2}}+\partial_{U_{1}}+\partial_{U_{2}})^{2}\implies & [\hat{\partial},\partial_{U_{1},U_{2}}]=0,\ \partial_{U_{1},U_{2}}^{2}=\hat{\partial}^{2}=0\implies\\
{}[\partial_{A_{1}},\partial_{U_{1}}]=[\partial_{A_{2}},\partial_{U_{1}}]=[\partial_{A_{2}},\partial_{U_{1}}]=[\partial_{A_{2}},\partial_{U_{2}}]=[\partial_{U_{1}},\partial_{U_{2}}]=0, & \ \ \partial_{A_{1}}^{2}=\partial_{U_{1}}^{2}=\partial_{A_{2}}^{2}=\partial_{U_{2}}^{2}=0.
\end{align*}
where $[f,g]=fg+gf$.

At this point, we first consider the Whitehead link. The $\widehat{CFL}(\mathit{Wh})$
is shown by the right term in Equation (\ref{eq:CFL_Hat_Whitehead}),
and the bullets are labeled as in Figure \ref{ABCD}. By looking at
the vertical arrows only, the equations $\partial_{U_{2}}^{2}=\partial_{A_{2}}^{2}=[\partial_{A_{2}},\partial_{U_{2}}]=0$
give rise to the two possibilities of the rightmost column as follows,
according to whether $\partial_{U_{2}}(c_{4}^{(1)})$ is $0$ or not.

\begin{center}
\begin{tikzpicture}[thick]
\node (a1) at ( 0,1.5) {$d^{(1)}_4$};
\node (a2) at ( 0,.3) {$d^{(1)}_3$}
  edge [<-] (a1);
\node (b1) at ( 0,-.3) {$c^{(1)}_4$}
  edge [->,out=180,in=180,red] node [left,black] {$U_2$} (a1);
\node (b2) at ( 0,-1.5) {$c^{(1)}_3$}
  edge [<-] (b1)
  edge [->,out=180,in=180,red] node [left,black] {$U_2$} (a2);
\node (a1') at ( 6,1.5) {$d^{(1)}_4$};
\node (a2') at ( 6,.3) {$d^{(1)}_3$}
  edge [<-] (a1');
\node (b1') at ( 6,-.3) {$c^{(1)}_4$};
\node (b2') at ( 6,-1.5) {$c^{(1)}_3$}
  edge [<-] (b1');
\end{tikzpicture}
\end{center}

Consider the Heegaard diagram of $L_{1}$ obtained from $\mathcal{H}$
by deleting $w_{2}$, i.e. the reduction $r_{-L_{2}}(\mathcal{H})$.
The differentials in $\widehat{CF}(r_{-L_{2}}(\mathcal{H}))$ count
the bigons without basepoints $w_{1},z_{1},z_{2}$ on $\mathcal{H}$,
which are the same as bigons with basepoint $w_{2}$. Thus, the complex
$\widehat{CF}(r_{-L_{2}}(\mathcal{H}))$ can be obtained by ignoring
the arrows $\partial_{A_{2}}$ and setting $U_{2}=1$. So the vertical
homology of $CFL^{-}(L)$ using only the $\partial_{U_{2}}$ arrows
is the knot Floer homology of the unknot $L_{1}$, $\mathbb{F}\oplus\mathbb{F}$,
supported in the filtration $A_{1}=\tau(L_{1})+\frac{\mathrm{lk}(L)}{2}=0.$
Thus, the right-hand side in the above diagram is ruled out. A similar
argument applies to the leftmost column. In sum,
\[
\partial_{U_{2}}(b_{1}^{(1)})=U_{2}a_{1}^{(1)},\ \partial_{U_{2}}(b_{2}^{(1)})=U_{2}a_{2}^{(1)},\ \partial_{U_{2}}(c_{4}^{(1)})=U_{2}d_{4}^{(1)},\ \partial_{U_{2}}(c_{3}^{(1)})=U_{2}d_{3}^{(1)}.
\]
Together with $\partial_{A_{1}}\partial_{U_{2}}=\partial_{U_{2}}\partial_{A_{1}},$
we get
\[
\partial_{A_{1}}\partial_{U_{2}}(b_{4}^{(1)})=\partial_{U_{2}}\partial_{A_{1}}(b_{4}^{(1)})=\partial_{U_{2}}b_{1}^{(1)}=U_{2}a_{1}^{(1)}\implies\partial_{U_{2}}(b_{4}^{(1)})=U_{2}a_{4}^{(1)}\text{ or }U_{2}(a_{4}^{(1)}+d_{1}^{(1)}).
\]
Thus, $\partial_{A_{2}}\partial_{U_{2}}=\partial_{U_{2}}\partial_{A_{2}}$
implies that $\partial_{U_{2}}(b_{3}^{(1)})=a_{3}^{(1)}$ and $\partial_{U_{2}}(c_{1}^{(1)})=d_{1}^{(1)},\ \partial_{U_{2}}(c_{2}^{(1)})=0.$

\begin{center}
\begin{tikzpicture}[thick]
\node (a1) at ( 0,2) {$a^{(1)}_1$};
\node (a2) at ( 0,.3) {$a^{(1)}_2$}
  edge [<-] (a1);
\node (b1) at ( 0,-.3) {$b^{(1)}_1$}
  edge [->,out=180,in=180,red] node [left,black] {$U_2$} (a1);
\node (b2) at ( 0,-2) {$b^{(1)}_2$}
  edge [<-] (b1)
  edge [->,out=180,in=180,red] node [left,black] {$U_2$} (a2);
\node (a4) at ( 1.7,2) {$a^{(1)}_4$}
  edge [->] (a1);
\node (a3) at ( 1.7,.3) {$a^{(1)}_3$}
  edge [<-] (a4)
  edge [->] (a2);
\node (b4) at ( 1.7,-.3) {$b^{(1)}_4$}
  edge [->] (b1)
  edge [->,out=170,in=190,red] node [right,black] {$U_2$} (a4);
\node (b3) at ( 1.7,-2) {$b^{(1)}_3$}
  edge [->] (b2)
  edge [<-] (b4)
  edge [->,out=170,in=190,red] node [right,black] {$U_2$} (a3);

\node (d1) at ( 2.3,2) {$d^{(1)}_1$}
  edge [<-, out=-100, in=65, red, dashed] node [black] {$U_2$} (b4);
\node (d2) at ( 2.3,.3) {$d^{(1)}_2$}
  edge [->, red] node [right,black] {$U_2$} (d1);
\node (c1) at ( 2.3,-.3) {$c^{(1)}_1$}
  edge [->,out=10,in=-10,red] node [right,black] {$U_2$} (d1);
\node (c2) at ( 2.3,-2) {$c^{(1)}_2$}
  edge [<-] (c1);
\node (d4) at ( 4,2) {$d^{(1)}_4$}
  edge [->] (d1);
\node (d3) at ( 4,.3) {$d^{(1)}_3$}
  edge [<-] (d4);
\node (c4) at ( 4,-.3) {$c^{(1)}_4$}
  edge [->] (c1)
  edge [->,out=0,in=0,red] node [right,black] {$U_2$} (d4);
\node (c3) at ( 4,-2) {$c^{(1)}_3$}
  edge [->] (c2)
  edge [<-] (c4)
  edge [->,out=0,in=0,red] node [right,black] {$U_2$} (d3);
\end{tikzpicture}
\end{center}

Next, $\partial_{U_{2}}\partial_{A_{2}}=\partial_{A_{2}}\partial_{U_{2}}$
implies that $\partial_{U_{2}}(d_{2}^{(1)})\in U_{2}\cdot D.$ To determine $\partial_{U_{2}}(d_{2}^{(1)})$,
we consider the complex $CFL^{-}(L)\otimes_{\mathbb{F}[[U_{1},U_{2}]]}(\mathbb{F}[[U_{1},U_{2}]]/U_{1})=CFL^{-}(L)/(U_{1}\cdot CFL^{-}(L))$,
i.e. setting $U_{1}=0$. The homology of this complex can be computed from the long exact
sequence of homologies, and it is $\mathbb{F}[[U_{2}]]/U_{2}$
as an $\mathbb{F}[[U_{2}]]$-module. Meanwhile, to compute this homology we can also use
the $A_1$-filtration and kill acyclic subcomplexes and acyclic quotient complexes. Taking the vertical homology
of this complex with respect to $\partial_{A_{2}}$ leaves only $d_{2}^{(1)}$
and $d_{1}^{(1)}$. Thus, from the homology constraint we have computed, it follows that $\partial_{U_{2}}(d_{2}^{(1)})=U_{2}d_{1}^{(1)}$.
Thus, we recover all the $U_{2}$-arrows. See the above figure, where
the dashed arrow is undetermined. Similarly, we can get all the $U_{1}$-arrows.
By changing basis, $\tilde{a}_{4}^{(1)}=a_{4}^{(1)}+d_{1}^{(1)}$
and $\tilde{c}_{4}^{(1)}=c_{4}^{(1)}+d_{3}^{(1)}$, we can get rid
of the dashed arrows, which gives the picture of $CFL^{-}(\mathit{Wh})$
in Figure \ref{CFL^-(Wh)}.

When $k>1$, we follow the same line of argument, together with doing
more changes of basis to prune the arrows. First, consider the rightmost
column, i.e. $R=\mathrm{Span}_{\mathbb{F}[[U_{2}]]}\{d_{3}^{(i)},d_{4}^{(i)},c_{3}^{(i)},c_{4}^{(i)}\}_{i=1}^{k}$
with the differentials $\partial_{A_{2}}+\partial_{U_{2}}$. Assume
$\partial_{U_{2}}(c{}_{4}^{(i)})=U_{2}\cdot\sum_{m=1}^{k}\lambda_{i,m}d_{4}^{(m)}$.
Then $\partial_{U_{2}}\partial_{A_{2}}=\partial_{A_{2}}\partial_{U_{2}}$
implies that $\partial_{U_{2}}(c_{3}^{(i)})=U_{2}\cdot\sum_{m=1}^{k}\lambda_{i,m}d_{3}^{(m)}.$
So the matrix $D=(\lambda_{i,m})$ represents the differential $\partial_{U_{2}}$
in the upside down vertical complex $\big(R\otimes_{\mathbb{F}[[U_2]]}\mathbb{F}[[U_{2}]]/(U_{2}-1),\partial_{U_{2}}\big)$.
Since its homology is $0$, the matrix $D$ is invertible. In other
words, the $\partial_{U_{2}}$-arrows form an isomorphism from $\mathrm{Span}_{\mathbb{F}}\{c_{3}^{(i)},c_{4}^{(i)}\}_{i=1}^{k}$
to $\mathrm{Span}_{\mathbb{F}}\{d_{3}^{(i)},d_{4}^{(i)}\}_{i=1}^{k}$.
Thus, we can find a new basis of $C$, namely $\{\tilde{c}_{1}^{(i)},\tilde{c}_{2}^{(i)},\tilde{c}_{3}^{(i)},\tilde{c}_{4}^{(i)}\}_{i=1}^{k}$,
such that
\[
\partial_{U_{2}}(\tilde{c}_{3}^{(i)})=U_{2}\cdot d_{3}^{(i)},\ \ \partial_{U_{2}}(\tilde{c}_{4}^{(i)})=U_{2}\cdot d_{4}^{(i)},\forall1\leq i\leq k,
\]
while the pattern of the $\hat{\partial}$ is preserved. In addition,
$[\partial_{U_{2}},\partial_{A_{1}}]=0$ implies that
\begin{align*}
\partial_{U_{2}}(\tilde{c}_{1}^{(i)}) & =U_{2}\cdot d_{1}^{(i)},\ \ \partial_{U_{2}}(\tilde{c}_{2}^{(i)})=U_{2}\cdot d_{2}^{(i)},\forall2\leq i\leq k,\\
\partial_{U_{2}}(\tilde{c}_{1}^{(1)}) & =U_{2}\cdot d_{1}^{(1)},\ \ \partial_{U_{2}}(\tilde{c}_{2}^{(1)})=0.
\end{align*}
From the fact that the vertical homology of $CFL^{-}(L)$ with respect
to the differential $\partial_{A_{2}}+\partial_{U_{2}}$ is $\mathbb{F}[[U_{2}]]/U_{2}$,
it follows that $\partial_{U_{2}}(d_{2})=U_{2}\cdot d_{1}$.

We may as well keep using the notations $c_{j}^{(i)}$ for the new
basis. Applying similar arguments for the leftmost column with respect
to vertical arrows, we can change the basis of $A$ without changing
the pattern of $\hat{\partial}$, such that
\[
\partial_{U_{2}}(b_{j}^{(i)})=U_{2}a_{j}^{(i)},\forall j=1,2,\forall i=1,...,k.
\]
Then $\partial_{A_{1}}\partial_{U_{2}}=\partial_{U_{2}}\partial_{A_{1}}$
implies
\[
\partial_{U_{2}}(b_{4}^{(i)})=U_{2}a_{4}^{(i)}+\sum_{m=1}^{k}\varepsilon_{i,m}U_{2}d_{1}^{(m)},\forall i=1,...,k,\varepsilon_{i,m}\in\mathbb{F}.
\]
Thus, $\partial_{U_{2}}(b_{3}^{(i)})=U_{2}a_{3}^{(i)}+\sum_{m=2}^{k}\varepsilon_{i,m}U_{2}d_{2}^{(m)},\forall i=1,...,k.$
Do base-changes:
\[
\tilde{a}_{4}^{(i)}=a_{4}^{(i)}+\sum_{m=1}^{k}\varepsilon_{i,m}d_{1}^{(m)},\ \ \tilde{a}_{3}^{(i)}=a_{3}^{(i)}+\sum_{m=2}^{k}\varepsilon_{i,m}d_{2}^{(m)}.
\]
We can preserve the pattern of $\hat{\partial}$, such that under
the new basis (where we keep using the notations $a_{j}^{(i)}$) all
the vertical arrows are pruned as $\partial_{U_{2}}(b_{j}^{(i)})=U_{2}a_{j}^{(i)},\forall j=1,2,3,4,\forall i=1,...,k.$
Similarly, by changing the bases of $A$ and $B$ simultaneously,
we can prune the horizontal arrows in the top row, while preserving
the pattern of $\hat{\partial},\partial_{U_{2}}$ on $A$ and \textbf{$B$},
such that $\partial_{U_{1}}(b_{j}^{(i)})=U_{1}c_{j}^{(i)},\forall j=1,4,\forall i=1,...,k.$
Then $\partial_{A_{2}}\partial_{U_{1}}=\partial_{U_{1}}\partial_{A_{2}}$
implies that $\partial_{U_{1}}(b_{2}^{(i)}),\partial_{U_{1}}(b_{3}^{(i)})$
are determined.

Similarly, all the horizontal arrows from $B$ can be pruned by changing
the basis of $C$. Suppose
\[
\partial_{U_{1}}(b_{4}^{(i)})=U_{1}(\sum_{m=1}^{k}\lambda_{i,m}c_{4}^{(m)}+\sum_{m=1}^{k}\mu_{i,m}d_{3}^{(m)}).
\]
Then $\partial_{U_{1}}\partial_{U_{2}}(b_{4}^{(i)})=\partial_{U_{2}}\partial_{U_{1}}(b_{4}^{(i)})$
implies that $U_{1}U_{2}(\sum_{m=1}^{k}\lambda_{i,m}d_{4}^{(m)})=U_{1}U_{2}d_{4}^{(i)}$.
Thus, $\lambda_{i,i}=1,\lambda_{i,m}=0,\forall i=1,...,k,\forall m\neq i.$
Thus, $\partial_{U_{1}}(b_{4}^{(i)})=U_{1}(c_{4}^{(i)}+\sum_{m=1}^{k}\mu_{i,m}d_{3}^{(m)})$.
Do base-changes
\begin{align*}
\tilde{c}_{4}^{(i)}=c_{4}^{(i)}+\sum_{m=1}^{k}\mu_{i,m}d_{3}^{(m)}, & \ \ \tilde{c}_{1}^{(i)}=c_{1}^{(i)}+\sum_{m=2}^{k}\mu_{i,m}d_{2}^{(m)},\forall i=1,...,k.
\end{align*}
Under the new basis (where we keep using the notations $c_{j}^{(i)}$),
the patterns of all the $\hat{\partial}$ and $\partial_{U_{2}}$
arrows are preserved, while $\partial_{U_{1}}(b_{4}^{(i)})=c_{4}^{(i)}.$
Moreover, $[\partial_{A_{1}},\partial_{U_{1}}]=[\partial_{A_{2}},\partial_{U_{1}}]=0$
implies that $\partial_{U_{1}}b_{j}^{(i)}=U_{1}c_{j}^{(i)},\forall j=1,2,3,4,\forall i=1,...,k.$
Finally, all the arrows are as in Figure \ref{CFL^-(Wh)}. Thus, $CFL^{-}(L)$
can be viewed as a square of chain complexes of $A,B,C,D.$
\end{proof}
Similar arguments apply to the case of $L=b(8k+4,4k+3).$
\begin{prop}
\label{prop:CFL^-b(8k,4k+3)}For the two-bridge link $L=b(8k+4,4k+3)$,
the filtered homotopy type of $CFL^{-}(L)$ is determined by the filtered
homotopy type of $\widehat{CFL}(L)$ (and hence by the Alexander polynomial,
signature and linking number). Furthermore, $CFL^{-}(L)=CFL^{-}(T(2,4))\oplus\bigoplus_{i=1}^{k-1}(N,\partial^{-})$,
where $(N,\partial^{-})$ is as in Figure \ref{CFL^-(Wh)} and $CFL^{-}(T(2,4))$
is as follows.\end{prop}

\begin{center}
\begin{tikzpicture}[thick, scale=0.7, transform shape]
\node (b1) at ( -.5,-.3) {$\bullet$};
\node (b2) at ( -.5,-2.5) {$\bullet$}
  edge [->,red] node [left, black]{$U_2$} (b1);
\node (b4) at ( 1.7,-.3) {$\bullet$}
  edge [->] (b1);
\node (b3) at ( 1.7,-2.5) {$\bullet$}
  edge [<-,red] node [below, black] {$U_1$}(b2)
  edge [<-] (b4);

\node (d1) at ( 2.3,2.5) {$\bullet$}
  edge [<-, red] node [left,black] {$U_2$} (b4);
\node (d2) at ( 2.3,.3) {$\bullet$}
  edge [<-,red] node [above, black] {$U_1$}  (b1)
  edge [<-,red] node [right,black] {$U_2$}  (b3)
  edge [<-] (d1);
\node (d4) at ( 4.5,2.5) {$\bullet$}
  edge [->] (d1);
\node (d3) at ( 4.5,.3) {$\bullet$}
  edge [<-] (d4)
  edge [->] (d2)
  edge [<-,red] node [below,black] {$U_1$} (b4);

\end{tikzpicture}
\end{center}

\subsection{Computations of surgeries on $b(8k,4k+1)$.}

In this section, we compute the homology of surgeries on the two-bridge
link $b(8k,4k+1)$ and their $d$-invariants explicitly. Here, we
make a convention of the $d$-invariants of ${\bf HF}^{-}$ different
from \cite{OS_Mixed_inv}. We require that $d({\bf HF}^{-}(S^{3}))=0.$
Thus, the $d$-invariants computed here are the same as the $d$-invariants
for $HF^{+}$.

We will first compute for the Whitehead link, following three steps:
computations of $A_{\mathrm{s}}^{-}(\mathit{Wh})$, computations of
the inclusion maps $I_{\mathrm{s}}^{\overrightarrow{M}}$, and the
computations of the homology of the surgeries on $\mathit{Wh}$.

\begin{lem}
\label{lem:H_*(A_s(Wh))}$H_{*}(A_{\mathrm{s}}^{-}(\mathit{Wh}))=\mathbb{F}[[U_{1},U_{2}]]/(U_{1}-U_{2})=\mathbb{F}[[U]]$
for all $\text{s}\in\mathbb{H}(\mathit{Wh})=\mathbb{Z}^{2}$.\end{lem}
\begin{proof}
By Proposition \ref{prop:CFL^-(b(8k,4k+1))}, we can decompose $A_{+\infty,+\infty}^{-}(\mathit{Wh})$
into a square of chain complexes, i.e.
\[
\xyR{1pc}\xyC{1pc}\xymatrix{A\ar[r] & D\\
B\ar[u]\ar[r] & C,\ar[u]
}
\]
where the summands $A,B,C,D$ are described in Proposition \ref{prop:CFL^-(b(8k,4k+1))}.
Since $A_{+\infty,+\infty}^{-}$ can be viewed as a mapping cone from
$A\oplus B\oplus C$ to $D$, we get a short exact sequence of chain
complexes
\[
0\rightarrow D\overset{i}{\rightarrow}A_{+\infty,+\infty}^{-}\rightarrow A\oplus B\oplus C\rightarrow0.
\]
From the fact that $H_{*}(A\oplus B\oplus C)=0,$ it follows that
$i$ is a quasi-isomorphism. Hence $H_{*}(A_{+\infty,+\infty}^{-})=\mathbb{F}[[U]]$
and $[d_{1}]=[d_{3}]=1\in H_{*}(A_{+\infty,+\infty}^{-}).$

All the other complexes $A_{s_{1},s_{2}}^{-}(\mathit{Wh})$ can be
actually obtained by taking various reflections on $A_{+\infty,+\infty}^{-}(\mathit{Wh})$.
Note that Equation (\ref{eq:A_s}) implies that the differentials
in $A_{\mathrm{s}}^{-}$ are only changed by $U_{1},U_{2}$ powers
from $A_{+\infty,+\infty}^{-}$. In order to read off the correct
powers of $U_{1},U_{2}$-coefficients, we can change the $\mathbb{Z}^{2}$-filtration
of $A_{\mathrm{s}}^{-}$, such that the upward and rightward arrows
in $A_{\mathrm{s}}^{-}$ are with $U_{1}$ and $U_{2}$ coefficients
respectively. For instance, when $s_{1}>0,s_{2}=0$, we can flip the
summands $A$ and $D$ about the $A_{1}$-axis to obtain the complex
$A_{s_{1},0}^{-}$. For convenience, we denote the vertical reflections
of $A,B,C,D$ by $\overline{A},\overline{B},\overline{C},\overline{D}$.
Thus, the complex $A_{s_{1},0}^{-}$ is still a square of chain complexes
as follows:
\[
\xyR{1pc}\xyC{1pc}\xymatrix{\overline{A}\ar[r] & \overline{D}\\
B\ar[u]\ar[r] & C.\ar[u]
}
\]
Thus, the fact that $\overline{A},B,C$ are acyclic implies that $H_{*}(A_{s_{1},0}^{-})=H_{*}(\overline{D})=\mathbb{F}[[U]]$.
Similarly, we denote the horizontal reflections of $A,B,C,D$ by $|A,|B,|C,|D$
respectively. Thus, the complex $A_{0,s_{1}}^{-}$ with $s_{1}>0$
is the following square of chain complexes
\[
\xyR{1pc}\xyC{1pc}\xymatrix{A\ar[r] & |D\\
B\ar[u]\ar[r] & |C.\ar[u]
}
\]

Following the same line, we list all the filtered homotopy types of
$A_{\mathrm{s}}^{-}$'s together with some generators of their homologies
in Table \ref{As(Wh)}.  Since $\max A_{1}=\max A_{2}=1,$ $\min A_{1}=\min A_{2}=-1$,
the notation $+\infty$ means a positive integer $s$, while $-\infty$
means a negative integer $s.$

\begin{table}[h]
\begin{tabular}{|p{2.09in}|p{1.8in}|p{2.09in}|}
\hline
$\begin{array}{c}
A_{-\infty,+\infty}^{-}=\\
\xymatrix{|A\ar[r] & |D\\
|B\ar[u]\ar[r] & |C,\ar[u]
}
\\
{}[d_{1}]=1\in H_{*}(A_{-\infty,+\infty}^{-});
\end{array}$ & $\begin{array}{c}
A_{0,+\infty}^{-}=\\
\xymatrix{A\ar[r] & |D\\
B\ar[u]\ar[r] & |C,\ar[u]
}
\\
{}[d_{1}]=1\in H_{*}(A_{0,+\infty}^{-});
\end{array}$ & $\begin{array}{c}
A_{+\infty,+\infty}^{-}=\\
\xymatrix{A\ar[r] & D\\
B\ar[u]\ar[r] & C,\ar[u]
}
\\
{}[d_{1}]=[d_{3}]=1\in H_{*}(A_{+\infty,+\infty}^{-});
\end{array}$\tabularnewline
\hline
$\begin{array}{c}
A_{-\infty,0}^{-}=\\
\xymatrix{|\overline{A}\ar[r] & |\overline{D}\\
|B\ar[u]\ar[r] & |C,\ar[u]
}
\\
{}[a_{2}]=1\in H_{*}(A_{-\infty,0}^{-});
\end{array}$ & $\begin{array}{c}
A_{0,0}^{-}=\\
\xymatrix{\overline{A}\ar[r] & |\overline{D}\\
B\ar[u]\ar[r] & |C,\ar[u]
}
\\
{}[d_{1}]=[d_{3}]=1\in H_{*}(A_{0,0}^{-});
\end{array}$ & $\begin{array}{c}
A_{+\infty,0}^{-}=\\
\xymatrix{\overline{A}\ar[r] & \overline{D}\\
B\ar[u]\ar[r] & C,\ar[u]
}
\\
{}[d_{3}]=1\in H_{*}(A_{+\infty,0}^{-});
\end{array}$\tabularnewline
\hline
$\begin{array}{c}
A_{-\infty,-\infty}^{-}=\\
\xymatrix{|\overline{A}\ar[r] & |\overline{D}\\
|\overline{B}\ar[u]\ar[r] & |\overline{C},\ar[u]
}
\\
{}[a_{2}]=[c_{2}]=1\in H_{*}(A_{-\infty,-\infty}^{-});
\end{array}$ & $\begin{array}{c}
A_{0,-\infty}^{-}=\\
\xymatrix{\overline{A}\ar[r] & |\overline{D}\\
\overline{B}\ar[u]\ar[r] & |\overline{C},\ar[u]
}
\\
{}[c_{2}]=1\in H_{*}(A_{0,-\infty}^{-});
\end{array}$ & $\begin{array}{c}
A_{+\infty,-\infty}^{-}=\\
\xymatrix{\overline{A}\ar[r] & \overline{D}\\
\overline{B}\ar[u]\ar[r] & \overline{C},\ar[u]
}
\\
{}[d_{3}]=[c_{2}]=1\in H_{*}(A_{+\infty,-\infty}^{-}).
\end{array}$\tabularnewline
\hline
\end{tabular}\caption{$A_{\mathrm{s}}^{-}(\mathit{Wh})$ and generators of their homology.}
\label{As(Wh)}
\end{table}

Note that $A,\overline{A},\vert A,B,\overline{B},\vert B,C,\overline{C},\vert C$
are all acyclic, and $D,\overline{D},\vert D,\vert\overline{D}$ all
have the same homology $\mathbb{F}[[U]]$. We can use the same argument
for $A_{+\infty,+\infty}^{-}$ to show that $A_{-\infty,+\infty}^{-},$
$A_{0,+\infty}^{-},$ $A_{+\infty,+\infty}^{-},$ $A_{0,0}^{-},$
$A_{+\infty,0}^{-},$ and $A_{+\infty,-\infty}^{-}$ all have the
same homology $\mathbb{F}[[U]]$. For those other $A_{\mathrm{s}}^{-}$,
we can use the conjugation symmetry, that is, $H_{*}(A_{\mathrm{s}}^{-}(L))=H_{*}(A_{-\mathrm{s}}^{-}(L)),\forall\mathrm{s}\in\mathbb{H}(L),\forall L.$
This is because $A_{\mathrm{s}}^{-}$'s are quasi-isomorphic to the
Floer complexes of large surgeries on $L$.

Now we explain the generators of their homologies in Table \ref{As(Wh)}.
The chain complex $A_{0,-\infty}^{-}$ can be viewed as a mapping
cone of a chain map from $\text{cone}(\overline{B}\rightarrow\overline{A})$
to $\text{cone}(|\overline{C}\rightarrow|\overline{D})$. Because
$\text{cone}(\overline{B}\rightarrow\overline{A}$) is acyclic, the
generator of $H_{*}\big(\text{cone}(|\overline{C}\rightarrow|\overline{D})\big)$
is also a generator of $H_{*}(A_{0,-\infty}^{-})$. Since $H_{*}(|\overline{C})=\mathbb{F}[[U]]/U,H_{*}(|\overline{D})=\mathbb{F}[[U]]$,
we derive a short exact sequence
\[
0\rightarrow\mathbb{F}[[U]]\rightarrow\mathbb{F}[[U]]\rightarrow\mathbb{F}[[U]]/U\rightarrow0
\]
from the long exact sequence of the homologies $\cdots\rightarrow H_{*}(|\overline{D})\rightarrow H_{*}\big(\text{cone}(|\overline{C}\rightarrow|\overline{D})\big)\rightarrow H_{*}(|\overline{C})\rightarrow\cdots.$
Because $[c_{2}]=1\in H_{*}(|\overline{C})=\mathbb{F}[[U]]/U$ and
$[c_{2}]\in\text{cone}(|\overline{C}\rightarrow|\overline{D})$ is
mapped to $[c_{2}]\in H_{*}(|\overline{C})$, the above short exact
sequence implies that $[c_{2}]=1\in H_{*}\big(\text{cone}(|\overline{C}\rightarrow|\overline{D})\big)$,
and thus $[c_{2}]=1\in H_{*}(A_{0,-\infty}^{-})$.

Similar arguments show that $[a_{2}]=[c_{2}]=1\in H_{*}(A_{-\infty,-\infty}^{-})$
and $[d_{3}]=1\in H_{*}(A_{+\infty,-\infty}^{-}).$ Moreover, in the
complex $A_{+\infty,-\infty}^{-}$, the equations $\partial^{-}c_{1}=U_{2}c_{2}+d_{1},\partial^{-}d_{2}=d_{1}+U_{1}d_{3}$
imply that $[c_{2}]=[d_{3}]=1\in H_{*}(A_{+\infty,-\infty}^{-})$.
\end{proof}

Taking the grading into account, we adopt the formula of the $\mathbb{Z}/\mathfrak{d}(\mathfrak{u})\mathbb{Z}$-grading
defined on the surgery complex for a $\text{Spin}^{c}$ structures
$\mathfrak{u}\in\mathbb{H}(L)/H(L,\Lambda)$ in \cite{link_surgery}
Section 7.4,
\begin{equation}
\mu(\text{s},{\bf x})=\mu_{\text{s}}^{M}({\bf x})+\nu(\text{s})-\Vert M\Vert,{\bf x}\in\mathfrak{A}^{-}(\mathcal{H}^{L-\overrightarrow{M}},\psi^{\overrightarrow{M}}(\text{s})),\label{eq:grading in Surgery Formula}
\end{equation}
where $\text{s}\in\mathfrak{u}$ and $\mu_{\text{s}}^{M}=\mu_{\psi^{\overrightarrow{M}}(s)}$
is a natural $\mathbb{Z}$-grading defined on $\mathfrak{A^{-}}\big(L-\overrightarrow{M},\psi^{\overrightarrow{M}}(s)\big)$.
In the torsion case, the quadratic function $\nu$ can be chosen as
$0$. The natural $\mathbb{Z}$-grading $\mu_{\text{s}}^{\emptyset}=\mu_{s_{1},s_{2}}$
on each $A_{s_{1},s_{2}}^{-}$ is given by
\[
\mu_{s_{1},s_{2}}({\bf x})=M({\bf x})-2\sum_{i=1}^{2}\max\{A_{i}({\bf x})-s_{i},0\},
\]
where $M({\bf x})$ is the Maslov grading. When we use the Schubert
Heegaard diagram, $A_{1}({\bf x})+A_{2}({\bf x})-M({\bf x})$ is constant.
Thus, up to a shift of a constant number, we can take $M({\bf x})=A_{1}({\bf x})+A_{2}({\bf x})$
for $\forall{\bf x}\in A_{s_{1},s_{2}}^{-}.$ In the primitive system
we identify $\mathfrak{A^{-}}\big(L-\overrightarrow{M},\psi^{\overrightarrow{M}}(\text{s})\big)$
with some $A_{s_{1}',s_{2}'}^{-}$ (where $s_{1}',s_{2}'$ can evaluate
$+\infty$), so the grading $\mu_{\text{s}}^{M}$ is actually $\mu_{\text{s'}}$.
We define some rules of $\infty$ as follows:
\[
0\cdot(+\infty)=+\infty;\ s+(+\infty)=+\infty,\forall s\in\mathbb{R};\ s+(-\infty)=-\infty,\forall s\in\mathbb{R};\ (\pm1)\cdot(+\infty)=\pm\infty.
\]

Recall the notations in Example \ref{(Twisted-gluing-in}. The complexes
$C_{(s_{1},s_{2})}^{(\varepsilon_{1},\varepsilon_{2})}=A_{s_{1}+\varepsilon_{1}\cdot\infty,s_{2}+\varepsilon_{2}\cdot\infty}^{-},\varepsilon_{i}\in\{0,1\}$
are setting at the position $(\varepsilon_{1},\varepsilon_{2})$ in
the square and with the index $(s_{1},s_{2})$ in the product complex
$C^{(\varepsilon_{1},\varepsilon_{2})}=\prod_{s_{1},s_{2}}C_{(s_{1},s_{2})}^{(\varepsilon_{1},\varepsilon_{2})}$.

We define the grading $\mu_{s_{1},s_{2}}^{\varepsilon_{1},\varepsilon_{2}}$
on the complex $C_{(s_{1},s_{2})}^{(\varepsilon_{1},\varepsilon_{2})}$
by the formula:
\[
\mu_{s_{1},s_{2}}^{\varepsilon_{1},\varepsilon_{2}}({\bf x})=M({\bf x})-2\sum_{i=1}^{2}\max\{A_{i}({\bf x})-s_{i}-\varepsilon_{i}(+\infty),0\}-\varepsilon_{1}-\varepsilon_{2}.
\]
Here $\mu_{s_{1},s_{2}}^{\varepsilon_{1},\varepsilon_{2}}$ plays
the role as $\mu$ in Equation (\ref{eq:grading in Surgery Formula}).

Let $W$ be the four-manifold cobordism corresponding to the surgery
from $S^{3}$ to $S_{\Lambda}^{3}(L)$. In \cite{link_surgery}, it
is shown that the cobordism map $F_{W,\text{s}}^{-}$ corresponds
to the inclusion $\iota:\mathfrak{A}^{-}(\mathcal{H}^{\emptyset})\rightarrow\mathfrak{A}^{-}(\mathcal{H}^{\emptyset},\psi^{\overrightarrow{L}}(\text{s}))\subset\mathcal{C}^{-}(\mathcal{H},\Lambda),\ \overrightarrow{L}=+L_{1}\cup+L_{2}.$
So we need to shift the grading such that $\iota$ is of the degree
$\deg(F_{W,\text{s}}^{-})=\frac{\text{c}_{1}(\text{s})^{2}-2\chi(W)-3\sigma(W)}{4}$.
In our case, the complex $\mathfrak{A}^{-}(\mathcal{H}^{\emptyset},\psi^{\overrightarrow{L}}(\text{s}))=C_{(s_{1},s_{2})}^{(1,1)}=A_{+\infty,+\infty}^{-}(\mathit{Wh})$
has a generator $[d_{1}]=1\in H_{*}(A_{+\infty,+\infty}^{-})$ of
Alexander grading $A(d_{1})=(0,1).$ Finally, the grading formula
turns out to be
\begin{equation}
\mu_{s_{1},s_{2}}^{\varepsilon_{1},\varepsilon_{2}}({\bf x})=A_{1}({\bf x})+A_{2}({\bf x})-2\sum_{i=1}^{2}\max\{A_{i}({\bf x})-s_{i}-\varepsilon_{i}(+\infty),0\}-\varepsilon_{1}-\varepsilon_{2}+\frac{\text{c}_{1}(\text{s})^{2}-2\chi(W)-3\sigma(W)}{4}+1,\label{eq:absolute grading}
\end{equation}
where $\text{c}_{1}(\text{s})=[2\text{s}]-\Lambda_{1}-\Lambda_{2}=(2s_{1}-p_{1},2s_{2}-p_{2})\in\mathbb{Z}^{2}/\Lambda.$
In the perturbed surgery complex, since all the perturbed maps have
the same degrees as the original, we can compute the gradings still
using Equation \eqref{eq:absolute grading}.

Now we restrict our scalars to $\bb{F}[[U_1]]$. By Proposition \ref{prop:homotopies on the diagonal} and Lemma \ref{lem:H_*(A_s(Wh))},
up to $\bb{F}[[U_1]]$-linear chain homotopy, all the edge maps $\Phi_{\text{s}}^{\pm L_{i}}$
are classified by their actions on the homologies. The actions of
$\Phi_{\text{s}}^{\pm L_{i}}$ on homologies are determined by the
corresponding inclusion maps $I_{\text{s}}^{\pm L_{i}}$. We denote
the induced maps on homologies by $(I_{\text{s}}^{\pm L_{i}})_{*}:\mathbb{F}[[U]]\rightarrow\mathbb{F}[[U]]$.

\begin{lem}
Regarding the inclusion maps, we have the following results for $I_{\text{s}}^{\pm L_{1}}$,
where $\text{s}=(s_{1},s_{2})$.

\begin{itemize}
\item

If $s_{1}>0$, then $(I_{\text{s}}^{+L_{1}})_{*}=id.$

\item

If $s_{1}=0,s_{2}\neq0$, then $(I_{\text{s}}^{+L_{1}})_{*}=id.$

\item

If $s_{1}=s_{2}=0,$ then $(I_{\text{s}}^{+L_{1}})_{*}=U\cdot id.$

\item

If $s_{1}<0$, then $(I_{\text{s}}^{+L_{1}})_{*}=U^{-s_{1}}\cdot id.$

\item

If $s_{1}>0$, then $(I_{\text{s}}^{-L_{1}})_{*}=U^{s_{1}}\cdot id.$

\item

If $s_{1}=0,s_{2}\neq0$, then $(I_{\text{s}}^{-L_{1}})_{*}=id.$

\item

If $s_{1}=s_{2}=0,$ then $(I_{\text{s}}^{-L_{1}})_{*}=U\cdot id.$

\item

If $s_{1}<0$, then $(I_{\text{s}}^{-L_{1}})_{*}=id.$

\end{itemize}\end{lem}

\begin{proof}
In fact, when $s_{1}>0$, by definition $I_{s_{1},s_{2}}^{+L_{1}}=id.$
When $s_{1}=0,s_{2}>0$, we have $I_{0,s_{2}}^{+L_{1}}(d_{1})=d_{1}$.
Therefore by Table \ref{As(Wh)}, the inclusion map $I_{0,s_{2}}^{+L_{1}}$
acts on the homology as $id:\mathbb{F}[[U]]\rightarrow\mathbb{F}[[U]].$
When $s_{1}=0,s_{2}<0$, we have $I_{0,s_{2}}^{+L_{1}}(c_{2})=c_{2}.$
Therefore by Table \ref{As(Wh)}, the inclusion map $I_{0,s_{2}}^{+L_{1}}$
acts on homology as the identity. When $s_{1}=s_{2}=0$, we have $I_{0,0}^{+L_{1}}(d_{3})=U_{1}\cdot d_{3}$.
Therefore by Table \ref{As(Wh)}, $I_{0,0}^{+L_{1}}$ acts on homology
as $U\cdot id$. When $s_{1}<0,s_{2}>0$, we have $I_{s_{1},s_{2}}^{+L_{1}}(d_{1})=U_{1}^{-s_{1}}\cdot d_{1}$.
Thus, $(I_{\text{s}}^{+L_{1}})_{*}=U^{-s_{1}}\cdot id.$ When $s_{1}<0,s_{2}\leq0$,
we have $I_{s_{1},s_{2}}^{+L_{1}}(a_{2})=U_{1}^{-1-s_{1}}\cdot a_{2}.$
In the complex $A_{+\infty,s_{2}}^{-},s_{2}\leq0$, the equation $\partial^{-}a_{3}=a_{2}+U_{1}d_{3}$
implies that $[a_{2}]=U_{1}[d_{3}]\in H_{*}(A_{+\infty,s_{2}}^{-})$.
Thus, $[a_{2}]=U\in\mathbb{F}[[U]]=H_{*}(A_{+\infty,s_{2}}^{-}).$
Therefore, it follows that $(I_{s_{1},s_{2}}^{+L_{1}})_{*}=U^{-s_{1}}\cdot id,$
when $s_{1}<0,s_{2}\leq0.$

In the same way, we get the following results for $I_{\text{s}}^{-L_{1}}$,
where $\mathrm{s}=(s_{1},s_{2})$.
\end{proof}

Now we can compute the homology of surgeries on $\mathit{S}$. In each case, we write
down the $d$-invariants, which are the gradings of the top element
in each $\mathbb{F}[[U]]$ summand.

\begin{prop}
\label{prop:surgery on Wh}
Let $\mathit{Wh}$ be the Whitehead link,
$\Lambda=\text{diag}(p_{1},p_{2})$ and $Y$ be the surgery manifold
$S_{\Lambda}^{3}(\mathit{Wh})$. Then $\mathrm{Spin}^{c}(Y)$ can
be identified with $\mathbb{Z}^{2}/\Lambda\cong\mathbb{Z}/p_{1}\mathbb{Z}\oplus\mathbb{Z}/p_{2}\mathbb{Z}$,
so we use $(t_{1},t_{2})\in\mathbb{Z}/p_{1}\mathbb{Z}\oplus\mathbb{Z}/p_{2}\mathbb{Z}$
to denote the \emph{$\text{Spin}^{c}$} structures over $Y$. Then,
the Floer homology of $Y$ is as follows.
\begin{itemize}
\item

If $p_{1}=p_{2}=0$, then $\mathbf{HF}^{-}(Y,(t_{1},t_{2}))=\begin{cases}
\mathbb{F}[[U]]^{\oplus4}, & (t_{1},t_{2})=(0,0);\\
0, & \text{otherwise,}
\end{cases}$ with $d=-1,-1,0,0.$

\item If $p_{1}>0,p_{2}=0$, then $\mathbf{HF}^{-}(Y,(t_{1},t_{2}))=\begin{cases}
\mathbb{F}[[U]]^{\oplus2}, & (t_{1},t_{2})=(t_{1},0);\\
0, & \text{otherwise.}
\end{cases}$ Their $d$-invariants are $d(Y,(0,0))=\frac{p_{1}}{4}-\frac{7}{4},\frac{p_{1}}{4}-\frac{3}{4},$
and $d(Y,(t_{1},0))=\frac{(2s_{1}+p_{1})^{2}}{4p_{1}}+\frac{1}{4},\frac{(2s_{1}+p_{1})^{2}}{4p_{1}}-\frac{3}{4},$
when $t_{1}\neq0$, where $s_{1}$ is an integer in the class $t_{1}\in\mathbb{Z}/p_{1}\mathbb{Z}$
such that $-p_{1}<s_{1}\leq0$.

\item  If $p_{1}<0,p_{2}=0$, then $\mathbf{HF}^{-}(Y,(t_{1},t_{2}))=\begin{cases}
\mathbb{F}[[U]]^{\oplus2}\oplus(\mathbb{F}[[U]]/U), & (t_{1},t_{2})=(0,0);\\
\mathbb{F}[[U]]^{\oplus2}, & t_{1}\neq0,t_{2}=0;\\
0, & \text{otherwise.}
\end{cases}$ Their $d$-invariants are $d(Y,(t_{1},0))=\frac{(2s_{1}-p_{1})^{2}}{4p_{1}}+\frac{3}{4},\frac{(2s_{1}-p_{1})^{2}}{4p_{1}}-\frac{1}{4},$
where $s_{1}$ is an integer in the class $t_{1}\in\mathbb{Z}/p_{1}\mathbb{Z}$
such that $p_{1}<s_{1}\leq0$.

\item If $p_{1}>0,p_{2}>0$, then $\mathbf{HF}^{-}(Y,(t_{1},t_{2}))=\mathbb{F}[[U]],\ \forall(t_{1},t_{2})\in\mathbb{Z}/p_{1}\mathbb{Z}\oplus\mathbb{Z}/p_{2}\mathbb{Z}.$
Their $d$-invariants are $d(Y,(0,0))=\frac{p_{1}+p_{2}-10}{4},$
and $d(Y,(t_{1},t_{2}))=\frac{(2s_{1}+p_{1})^{2}}{4p_{1}}+\frac{(2s_{2}+p_{2})^{2}}{4p_{2}}-\frac{1}{2},$
when $(t_{1},t_{2})\neq(0,0),$ where $s_{i}$ is an integer in the
class $t_{i}\in\mathbb{Z}/p_{i}\mathbb{Z}$ such that $-p_{i}<s_{i}\leq0$.

\item If $p_{1}>0,p_{2}<0$, then $\mathbf{HF}^{-}(Y,(t_{1},t_{2}))=\begin{cases}
\mathbb{F}[[U]]\oplus(\mathbb{F}[[U]]/U), & (t_{1},t_{2})=(0,0);\\
\mathbb{F}[[U]], & \text{otherwise.}
\end{cases}$ Their $d$-invariants are $d(Y,(t_{1},t_{2}))=\frac{(2s_{1}+p_{1})^{2}}{4p_{1}}+\frac{(2s_{2}-p_{2})^{2}}{4p_{2}}$,
where $s_{i}$ is an integer in the class $t_{i}\in\mathbb{Z}/p_{i}\mathbb{Z}$
such that $-\vert p_{i}\vert<s_{i}\leq0$.

\item  If $p_{1},p_{2}<0$, then $\mathbf{HF}^{-}(Y,(t_{1},t_{2}))=\begin{cases}
\mathbb{F}[[U]]\oplus(\mathbb{F}[[U]]/U), & (t_{1},t_{2})=(0,0);\\
\mathbb{F}[[U]], & \text{otherwise.}
\end{cases}$ Their $d$-invariants are $d(Y,(t_{1},t_{2}))=\frac{(2s_{1}-p_{1})^{2}}{4p_{1}}+\frac{(2s_{2}-p_{2})^{2}}{4p_{2}}+\frac{1}{2},$
where $s_{i}$ is an integer in the class $t_{i}\in\mathbb{Z}/p_{i}\mathbb{Z}$
such that $p_{i}<s_{i}\leq0$.

\end{itemize}
\end{prop}

\begin{proof}
First, let's look at the $(0,0)$-surgery on the Whitehead link. The
surgery complex splits into a direct product of squares of chain complexes
according to $\mathrm{Spin}^{c}$ structures. See Figure \ref{(0,0)surgery}.
In the $(s_{1},s_{2})$ $\mathrm{Spin}^{c}$ structure, the factor
of the direct product is the following square of chain complexes:
\[
\xyR{2pc}\xyC{4pc}\xymatrix{A_{s_{1},s_{2}}^{-}\ar[r]^{\Phi_{s_{1},s_{2}}^{+L_{1}}+\Phi_{s_{1},s_{2}}^{-L_{1}}}\ar[d]_{\Phi_{s_{1},s_{2}}^{+L_{2}}+\Phi_{s_{1},s_{2}}^{-L_{2}}}\ar[dr]|-{\sum\Phi_{s_{1},s_{2}}^{\pm L_{1}\cup\pm L_{2}}} & A_{+\infty,s_{2}}^{-}\ar[d]^{\Phi_{+\infty,s_{2}}^{+L_{2}}+\Phi_{+\infty,s_{2}}^{-L_{2}}}\\
A_{s_{1},+\infty}^{-}\ar[r]_{\Phi_{s_{1},+\infty}^{+L_{1}}+\Phi_{s_{1},+\infty}^{-L_{1}}} & A_{+\infty,+\infty}^{-}.
}
\]

For the torsion $\text{Spin}^{c}$ structure $(0,0)\in\mathbb{Z}^{2}$,
since $\Phi_{0,0}^{+L_{1}}\simeq\Phi_{0,0}^{-L_{1}},$ $\Phi_{0,0}^{+L_{2}}\simeq\Phi_{0,0}^{-L_{2}},$
$\Phi_{0,+\infty}^{+L_{1}}\simeq\Phi_{0,\infty}^{-L_{1}},$ $\Phi_{+\infty,0}^{+L_{2}}\simeq\Phi_{+\infty,0}^{-L_{2}}$,
the perturbed surgery complex is as follows:
\[
\xyR{1pc}\xyC{1pc}\xymatrix{A_{0,0}^{-}\ar[r]^{0}\ar[d]_{0} & A_{+\infty,0}^{-}\ar[d]^{0}\\
A_{0,+\infty}^{-}\ar[r]^{0} & A_{+\infty,+\infty}^{-}.
}
\]
Therefore, the homology is $\mathbb{F}[[U]]^{\oplus4}$ generated
by $d_{1}\in C_{(0,0)}^{(0,0)},d_{3}\in C_{(0,0)}^{(1,0)},d_{1}\in C_{(0,0)}^{(0,1)},d_{1}\in C_{(0,0)}^{(1,1)}$.
Since $\text{c}_{1}\big((0,0)\big)=(0,0)$, $\chi(W)=2,$ $\sigma(W)=0$,
from Equation (\ref{eq:absolute grading}), we also get their absolute
gradings $\mu_{0,0}^{0,0}([d_{1}])=-1,\mu_{0,0}^{1,0}([d_{3}])=0,\mu_{0,0}^{0,1}([d_{1}])=0,\mu_{0,0}^{1,1}([d_{3}])=-1.$

For the non-torsion $\text{Spin}^{c}$ structure $(s_{1},s_{2})\in\mbox{\ensuremath{\mathbb{Z}}}^{2},s_{1}>0$,
since $\Phi_{s_{1},s_{2}}^{+L_{1}}+\Phi_{s_{1},s_{2}}^{-L_{1}}$ acts
on homology as $id+U^{s_{1}}\cdot id=(1+U^{s_{1}})\cdot id$, which
is a quasi-isomorphism, it follows that the homology of this $\mathrm{Spin}^{c}$
structure is $0$. Indeed, one can consider the horizontal filtration
for this square, whose associated graded is the direct sum of the
two acyclic horizontal rows. A similar argument applies to all the other
non-torsion $\text{Spin}^{c}$ structures.

\begin{figure}

\includegraphics[scale=0.8]{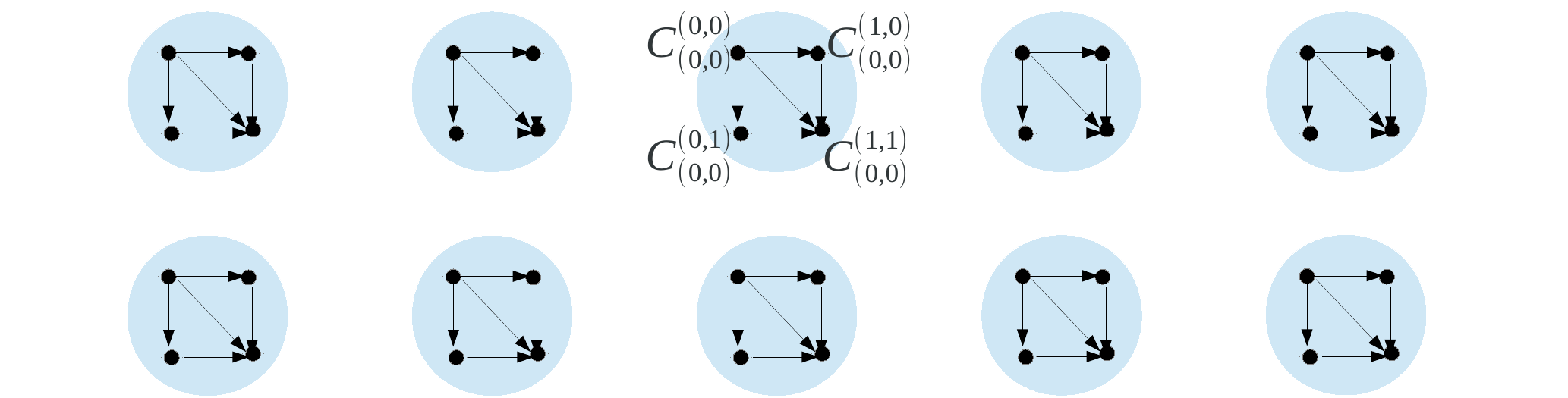}\caption{\textbf{The surgery complex for $\Lambda=(0,0)$.} Every dot represents
a complex $C_{\text{s}}^{\varepsilon}$ which is a certain generalized
Floer complex $A_{\text{s'}}^{-}(\mathit{Wh})$, and every arrow represents
a $\Phi$-map according to the endpoints of the arrow. We only label
the four complexes $C_{\text{s}}^{\varepsilon}$ for the $\mathrm{Spin}^{c}$
structure $\text{s}=(0,0)$, and the others are similar. }

\label{(0,0)surgery}

\end{figure}

\

Second, let's look at the $(p_{1},0)$-surgery with $p_{1}\neq0$,
which gives rise to a manifold with $b_{1}=1$. Suppose $p_{1}>0.$
In order to compute the homology, we need some filtrations to kill
acyclic subcomplexes and quotient complexes. Let $\mathcal{F}_{1}(C_{(s_{1},s_{2})}^{(\varepsilon_{1},\varepsilon_{2})})=-s_{1},\mathcal{F}_{2}(C_{(s_{1},s_{2})}^{(\varepsilon_{1},\varepsilon_{2})})=s_{1}-(\varepsilon_{1}-1)p_{1}.$
Without loss of generality, see Figure \ref{(1,0)surgery} for the
illustration of the surgery complex and the truncation in the case
of $\Lambda=(1,0).$

For any $(t_{1},t_{2})\in\text{Spin}^{c}(Y)=\mathbb{Z}/p_{1}\mathbb{Z}\oplus\mathbb{Z}$
with $t_{2}\neq0$, the Floer homology is $0$. Indeed, we can consider
the union of all these $\mathrm{Spin}^{c}$ structures, which corresponds
to the subcomplex
\[
\mathcal{R}_{1}=\bigoplus_{s_{2}\neq0}(C_{(s_{1},s_{2})}^{(0,0)}\oplus C_{(s_{1},s_{2})}^{(1,0)}\oplus C_{(s_{1},s_{2})}^{(0,1)}\oplus C_{(s_{1},s_{2})}^{(1,1)}).
\]
Since $\Phi_{s_{1},s_{2}}^{+L_{2}}+\Phi_{s_{1},s_{2}}^{-L_{2}},s_{2}\neq0$
acts on homology as $id+U^{|s_{2}|}\cdot id=(1+U^{|s_{2}|})\cdot id$,
which is a quasi-isomorphism, the following square is acyclic:
\[
\xyR{1.5pc}\xyC{5pc}\xymatrix{A_{s_{1},s_{2}}^{-}\ar[r]^{\Phi_{s_{1},s_{2}}^{+L_{1}}}\ar[d]_{\Phi_{s_{1},s_{2}}^{+L_{2}}+\Phi_{s_{1},s_{2}}^{-L_{2}}}\ar[dr]|-{\Phi_{s_{1},s_{2}}^{+L_{1}\cup+L_{2}}+\Phi_{s_{1},s_{2}}^{+L_{1}\cup-L_{2}}} & A_{+\infty,s_{2}}^{-}\ar[d]^{\Phi_{+\infty,s_{2}}^{+L_{2}}+\Phi_{+\infty,s_{2}}^{-L_{2}}}\\
A_{s_{1},+\infty}^{-}\ar[r]_{\Phi_{s_{1},+\infty}^{+L_{1}}} & A_{+\infty,+\infty}^{-}.
}
\]
The associated graded complex of $\mathcal{F}_{1}$ splits as a direct
product of the above squares, so $\mathcal{R}_{1}$ is acyclic.

For the $\text{Spin}^{c}$ structure $(t_{1},0)$, we first kill the
acyclic subcomplex
\[
\mathcal{R}_{2}=\bigoplus_{s_{1}>0}C_{(s_{1},0)}^{(\varepsilon_{1},\varepsilon_{2})}.
\]
Since the inclusion map $I_{s_{1},0}^{+L_{1}}$ is $id$ for all $s_{1}>0$,
the associated graded complex of the filtration $\mathcal{F}_{1}$
splits as a direct product of acyclic complexes in the form of
\[
\xyR{1.5pc}\xyC{5pc}\xymatrix{\ \ C_{(s_{1},0)}^{(0,0)}\ar[r]^{\Phi_{s_{1},0}^{+L_{1}}}\ar[d]_{\Phi_{s_{1},0}^{+L_{2}}+\Phi_{s_{1},0}^{-L_{2}}}\ar[dr]|-{\Phi_{s_{1},s_{2}}^{+L_{1}\cup+L_{2}}+\Phi_{s_{1},s_{2}}^{+L_{1}\cup-L_{2}}} & C_{(s_{1},0)}^{(1,0)}\ \ \ar[d]^{\Phi_{+\infty,0}^{+L_{2}}+\Phi_{+\infty,0}^{-L_{2}}}\\
\ \ C_{(s_{1},0)}^{(0,1)}\ar[r]_{\Phi_{s_{1},+\infty}^{+L_{1}}} & C_{(s_{1},0)}^{(1,1)}.\ \
}
\]
Thus $\mathcal{R}_{2}$ is acyclic.

On the other hand, we have another acyclic subcomplex
\[
\mathcal{R}_{3}=\bigoplus_{\mathcal{F}_{2}\leq0}C_{(s_{1},0)}^{(\varepsilon_{1},\varepsilon_{2})}.
\]
In fact, since the inclusion maps $I_{s_{1},0}^{-L_{1}}$ and $I_{s_{1},+\infty}^{-L_{1}}$
are both $id$ when $s_{1}<0$, the associated graded complex of the
filtration $\mathcal{F}_{2}$ splits as a direct product of acyclic
complexes in the form of
\[
\xyR{1.5pc}\xyC{5pc}\xymatrix{C_{(s_{1},0)}^{(0,0)}\ar[r]^{\Phi_{s_{1},0}^{-L_{1}}}\ar[d]_{\Phi_{s_{1},0}^{+L_{2}}+\Phi_{s_{1},0}^{-L_{2}}}\ar[dr]|-{\Phi_{s_{1},s_{2}}^{-L_{1}\cup+L_{2}}+\Phi_{s_{1},s_{2}}^{-L_{1}\cup-L_{2}}} & C_{(s_{1}+p_{1},0)}^{(1,0)}\ar[d]^{\Phi_{+\infty,0}^{+L_{2}}+\Phi_{+\infty,0}^{-L_{2}}}\\
C_{(s_{1},0)}^{(0,1)}\ar[r]_{\Phi_{s_{1},+\infty}^{-L_{1}}} & C_{(s_{1}+p_{1},0)}^{(1,1)}.
}
\]
 Thus $\mathcal{R}_{3}$ is acyclic. So the quotient complex $\mathcal{Q}=\mathcal{C}^{-}/\mathcal{R}_{1}\cup\mathcal{R}_{2}\cup\mathcal{R}_{3}$
is a direct product of
\[
C_{(s_{1},0)}^{(0,0)}\xrightarrow{\Phi_{s_{1},0}^{+L_{2}}+\Phi_{s_{1},0}^{-L_{2}}}C_{(s_{1},0)}^{(0,1)},
\]
where $-p_{1}+1\leq s_{1}\leq0.$ From the computations of the inclusion
maps, we know that $\Phi_{s_{1},0}^{+L_{2}}\simeq\Phi_{s_{1},0}^{-L_{2}}$.
Thus the homology of each $\text{Spin}^{c}$ structure $(t_{1},0)\in\mathbb{Z}/p_{1}\mathbb{Z}\oplus\mathbb{Z}$
is $\mathbb{F}[[U]]^{\oplus2}$. Note that $\chi(W)=2,\sigma(W)=1.$
When $-p_{1}+1\leq s_{1}<0$, the complex $C_{(s_{1},0)}^{(0,0)}=A_{s_{1},0}^{-}$
has $a_{2}$ as a generator of its homology of grading $\mu_{s_{1},0}^{0,0}(a_{2})=\frac{s_{1}^{2}}{p_{1}}+s_{1}+\frac{p_{1}}{4}+\frac{1}{4}$,
and the complex $C_{(s_{1},0)}^{(0,1)}=A_{s_{1},+\infty}^{-}$ has
$d_{1}$ as a generator of its homology of grading $\mu_{s_{1},0}^{0,1}(d_{1})=\frac{s_{1}^{2}}{p_{1}}+s_{1}+\frac{p_{1}}{4}-\frac{3}{4}$.
While for the $(0,0)$ $\mathrm{Spin}^{c}$ structure, $C_{(0,0)}^{(0,0)}=A_{0,0}^{-}$
has $d_{1}$ as a generator of its homology with grading $\mu_{0,0}^{0,0}(d_{1})=\frac{p_{1}}{4}-\frac{7}{4}$,
and $C_{(0,0)}^{(0,1)}=A_{0,+\infty}^{-}$ has $d_{1}$ as a generator
of its homology with grading $\mu_{0,0}^{0,1}(d_{1})=\frac{p_{1}}{4}-\frac{3}{4}.$

\begin{figure}

\includegraphics[scale=0.8]{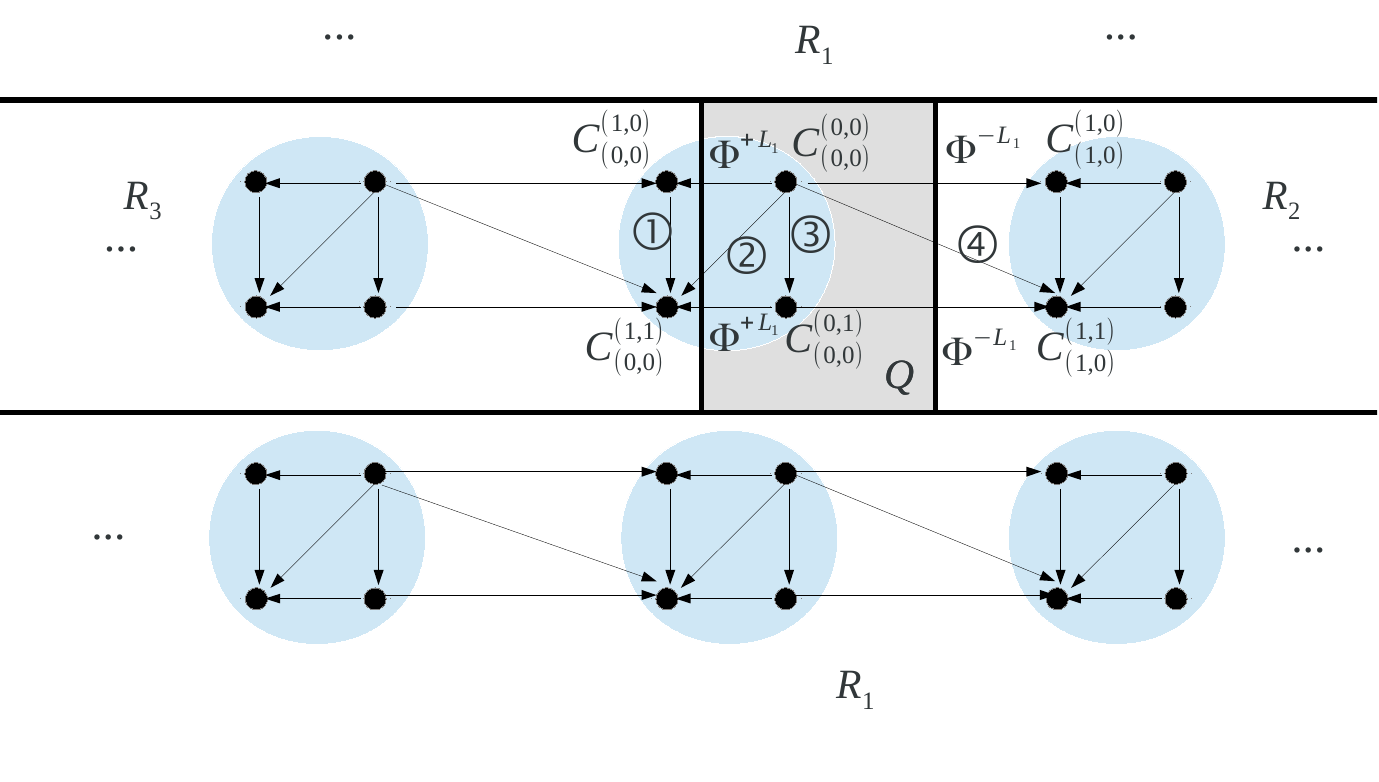}\caption{\textbf{The surgery complex for $\Lambda=(1,0)$.} Every dot represents
a complex $C_{\text{s}}^{\varepsilon}$ which is a certain generalized
Floer complex $A_{\text{s'}}^{-}(\mathit{Wh})$, and in every shaded
circle the complexes $C_{\text{s}}^{\varepsilon}$'s have the same
subscript $\text{s}$. Every arrow represents a $\Phi$-map according
to the endpoints of the arrow, where we omit the subscripts. All the
parallel arrows share the same type of $\Phi^{\protect\overrightarrow{M}}$,
i.e. having the same superscript $\protect\overrightarrow{M}$. The
arrows with circled numbers $1,2,3,4$ are $\Phi_{+\infty,0}^{+L_{2}}+\Phi_{+\infty,0}^{-L_{2}},$
$\Phi_{0,0}^{+L_{2}\cup+L_{1}}+\Phi_{0,0}^{-L_{2}\cup+L_{1}},$ $\Phi_{0,0}^{+L_{2}}+\Phi_{0,0}^{-L_{2}},$
and $\Phi_{0,0}^{+L_{2}\cup-L_{1}}+\Phi_{0,0}^{-L_{2}\cup-L_{1}}$
respectively. The regions $R_{1},R_{2},R_{3}$ divided by the (thicker)
lines are corresponded to the acyclic subcomplexes $\mathcal{R}_{1},\mathcal{R}_{2},\mathcal{R}_{3}.$
The shaded region $Q$ corresponds to the truncated complex $\mathcal{Q}$. }

\label{(1,0)surgery}

\end{figure}

The case of $p_{1}<0$ is similar. We first kill the acyclic subcomplex
\[
\mathcal{R}_{1}=\bigoplus_{s_{2}\neq0}(C_{(s_{1},s_{2})}^{(0,0)}\oplus C_{(s_{1},s_{2})}^{(1,0)}\oplus C_{(s_{1},s_{2})}^{(0,1)}\oplus C_{(s_{1},s_{2})}^{(1,1)}).
\]
 Thus, the homology for the $\text{Spin}^{c}$ structure $(t_{1},t_{2})$
with $t_{2}\neq0$ is $0$. Next, we kill the acyclic quotient complexes
\[
\mathcal{R}_{2}=\bigoplus_{s_{1}>0}C_{(s_{1},0)}^{(\varepsilon_{1},\varepsilon_{2})},\ \ \ \mathcal{R}_{3}=\bigoplus_{s_{1}-\varepsilon_{1}p_{1}<0}C_{(s_{1},0)}^{(\varepsilon_{1},\varepsilon_{2})}.
\]
In the $(0,0)$ $\text{Spin}^{c}$ structure, the remaining complexes
are as follows

\[
\xyR{1.5pc}\xyC{9pc}\xymatrix{C_{(p_{1},0)}^{(1,0)}\ar[d]_{\Phi_{+\infty,0}^{+L_{2}}+\Phi_{+\infty,0}^{-L_{2}}} & C_{(0,0)}^{(0,0)}\ar[d]|-{\Phi_{0,0}^{+L_{2}}+\Phi_{0,0}^{-L_{2}}}\ar[l]_{\Phi_{0,0}^{-L_{1}}}\ar[r]^{\Phi_{0,0}^{+L_{1}}}\ar[rd]|-{\ \ \ \Phi_{0,0}^{+L_{1}\cup+L_{2}}+\Phi_{0,0}^{+L_{1}\cup-L_{2}}}\ar[ld]|-{\Phi_{0,0}^{-L_{1}\cup+L_{2}}+\Phi_{0,0}^{-L_{1}\cup-L_{2}}\ \ \ } & C_{(0,0)}^{(1,0)}\ar[d]^{\Phi_{+\infty,0}^{+L_{2}}+\Phi_{+\infty,0}^{-L_{2}}}\\
C_{(p_{1},0)}^{(1,1)} & C_{(0,0)}^{(0,1)}\ar[l]^{\Phi_{0,+\infty}^{-L_{1}}}\ar[r]_{\Phi_{0,+\infty}^{+L_{1}}} & C_{(0,0)}^{(1,1)}.
}
\]
Since $\Phi_{+\infty,0}^{+L_{2}},\Phi_{+\infty,0}^{-L_{2}}$ are chain
homotopic, we can replace $\Phi_{+\infty,0}^{+L_{2}}+\Phi_{+\infty,0}^{-L_{2}}$
by $0$ in the perturbed surgery complex. Therefore, we can also replace
the diagonal maps by $0$. Thus, we have two split complexes,
\[
\xyR{1pc}\xyC{2pc}\xymatrix{C_{(p_{1},0)}^{(1,0)} & C_{(0,0)}^{(0,0)}\ar[l]_{\Phi_{0,0}^{-L_{1}}}\ar[r]^{\Phi_{0,0}^{+L_{1}}} & C_{(0,0)}^{(1,0)},}
\]
\[
\xyR{1pc}\xyC{2pc}\xymatrix{C_{(p_{1},0)}^{(1,1)} & C_{(0,0)}^{(0,1)}\ar[l]^{\Phi_{0,+\infty}^{-L_{1}}}\ar[r]_{\Phi_{0,+\infty}^{+L_{1}}} & C_{(0,0)}^{(1,1)}.}
\]
Since $C_{(p_{1},0)}^{(1,0)}=C_{(0,0)}^{(1,0)}$ and $\Phi_{0,0}^{+L_{1}}\simeq\Phi_{0,0}^{-L_{1}}$,
we can replace $\Phi_{0,0}^{+L_{1}}$ by $\Phi_{0,0}^{-L_{1}}$ in
the perturbed complex. By changing basis, we can split the first row
as $\text{cone}(\Phi_{0,0}^{+L_{1}})\oplus C_{(p_{1},0)}^{(1,0)}$
with homology $\mathbb{F}[[U]]\oplus(\mathbb{F}[[U]]/U).$ Similarly,
from that $\Phi_{0,+\infty}^{\pm L_{1}}$ are quasi-isomorphisms,
it follows that the second row is quasi-isomorphic to $C_{(0,0)}^{(1,1)}$
by changing basis. Thus in the $(0,0)$ $\text{Spin}^{c}$ structure,
the homology is $\mathbb{F}[[U]]^{\oplus2}\oplus(\mathbb{F}[[U]]/U).$

For the other $\text{Spin}^{c}$ structures $(t_{1},0),t_{1}\neq0$,
the remaining complexes are as follows
\[
C_{(s_{1},0)}^{(1,0)}\xrightarrow{\Phi_{+\infty,0}^{+L_{2}}+\Phi_{+\infty,0}^{-L_{2}}}C_{(s_{1},0)}^{(1,1)},
\]
where the integer $s_{1}$ is in the residue class $t_{1}\in\mathbb{Z}/p_{1}\mathbb{Z}$
such that $p_{1}<s_{1}<0$. Similarly, we replace $\Phi_{+\infty,0}^{+L_{2}}+\Phi_{+\infty,0}^{-L_{2}}$
by $0$, and get that the homology is $\mathbb{F}[[U]]^{\oplus2}.$
The correction terms can be computed similarly with $\chi(W)=2,\sigma(W)=-1.$

\

Finally, let's look at the $(p_{1},p_{2})$-surgery, where $p_{1}p_{2}\neq0$.
This breaks down to three cases: $p_{1},p_{2}>0$, $p_{1},p_{2}<0$,
and $p_{1}p_{2}<0.$ We apply the truncation tricks shown in \cite{link_surgery}
Section 8.3.

(1) When $p_{1},p_{2}>0$, the $(p_{1},p_{2})$-surgery is actually
a large surgery, so its homology can be derived from $A_{\mathrm{s}}^{-}$
directly. However, we still compute them by elementary methods. We
construct two filtrations,
\[
\mathcal{F}_{00}(C_{(s_{1},s_{2})}^{(\varepsilon_{1},\varepsilon_{2})})=-s_{1}-s_{2},\ \ \mathcal{F}_{11}(C_{(s_{1},s_{2})}^{(\varepsilon_{1},\varepsilon_{2})})=s_{1}-(\varepsilon_{1}-1)p_{1}+s_{2}-(\varepsilon_{2}-1)p_{2}.
\]
Without loss of generality, see Figure \ref{(1,1)surgery} for the
illustration of the surgery complex and the truncation in the case
of $\Lambda=(1,1).$

We first consider an acyclic subcomplex
\[
\mathcal{R}_{1}=\bigoplus_{\max\{s_{1},s_{2}\}>0}C_{(s_{1},s_{2})}^{(\varepsilon_{1},\varepsilon_{2})}.
\]
In fact, since the inclusion maps $I_{s_{1},s_{2}}^{+L_{1}},s_{1}>0$
and $I_{s_{1},s_{2}}^{+L_{2}},s_{2}>0$ are both $id$'s, the associated
graded complex of the filtration $\mathcal{F}_{00}$ splits as a direct
product of acyclic squares $R_{\mathrm{s},0,0}$ in Equation \ref{eq:surgery formula1}:
\[
\xyR{2pc}\xyC{4pc}\xymatrix{A_{s_{1},s_{2}}^{-}\ar[r]^{\Phi_{s_{1},s_{2}}^{+L_{1}}}\ar[d]_{\Phi_{s_{1},s_{2}}^{+L_{2}}}\ar[dr]|-{\Phi_{s_{1},s_{2}}^{+L_{1}\cup+L_{2}}} & A_{+\infty,s_{2}}^{-}\ar[d]^{\Phi_{+\infty,s_{2}}^{+L_{2}}}\\
A_{s_{1},+\infty}^{-}\ar[r]_{\Phi_{s_{1},+\infty}^{+L_{1}}} & A_{+\infty,+\infty}^{-}.
}
\]
Thus $\mathcal{R}_{1}$ is acyclic.

There is another acyclic subcomplex
\[
\mathcal{R}_{2}=\bigoplus_{\max\{s_{1}-(\varepsilon_{1}-1)p_{1},s_{2}-(\varepsilon_{2}-1)p_{2}\}\leq0}C_{(s_{1},s_{2})}^{(\varepsilon_{1},\varepsilon_{2})}.
\]
One can directly check $\mathcal{R}_{2}$ is a subcomplex by computation.
Because the inclusion maps $I_{s_{1},s_{2}}^{-L_{1}},s_{1}<0$ and
$I_{s_{1},s_{2}}^{-L_{2}},s_{2}<0$ are both $id$'s, the associated
graded complex of $\mathcal{F}_{11}$ splits as a product of acyclic
squares $R_{\mathrm{s},1,1}:$
\[
\xyR{2pc}\xyC{4pc}\xymatrix{A_{s_{1},s_{2}}^{-}\ar[r]^{\Phi_{s_{1},s_{2}}^{-L_{1}}}\ar[d]_{\Phi_{s_{1},s_{2}}^{-L_{2}}}\ar[dr]|-{\Phi_{s_{1},s_{2}}^{-L_{1}\cup-L_{2}}} & A_{+\infty,s_{2}}^{-}\ar[d]^{\Phi_{+\infty,s_{2}}^{-L_{2}}}\\
A_{s_{1},+\infty}^{-}\ar[r]_{\Phi_{s_{1},+\infty}^{-L_{1}}} & A_{+\infty,+\infty}^{-},
}
\]
where $s_{1}+p_{1}\leq0,s_{2}+p_{2}\leq0.$ Thus $\mathcal{R}_{2}$
is acyclic.

Let $\mathcal{C}_{1}=\mathcal{C}/(\mathcal{R}_{1}+\mathcal{R}_{2})$.
Inside $\mathcal{C}_{1}$, there are two acyclic subcomplexes
\begin{align*}
\mathcal{R}_{3} & =\{\bigoplus_{s_{1}-(\varepsilon_{1}-1)p_{1}\leq0,-p_{2}+1\leq s_{2}\leq0}C_{(s_{1},s_{2})}^{(\varepsilon_{1},\varepsilon_{2})}\}\cap\mathcal{C}_{1}=\bigoplus_{s_{1}+p_{1}\leq0,0\geq s_{2}>-p_{2}}\big(C_{(s_{1},s_{2})}^{(0,0)}\oplus C_{(s_{1}+p_{1},s_{2})}^{(1,0)}\big),\\
\mathcal{R}_{4} & =\{\bigoplus_{s_{2}-(\varepsilon_{2}-1)p_{2}\leq0,-p_{1}+1\leq s_{1}\leq0}C_{(s_{1},s_{2})}^{(\varepsilon_{1},\varepsilon_{2})}\}\cap\mathcal{C}_{1}=\bigoplus_{s_{2}+p_{2}\leq0,0\geq s_{1}>-p_{1}}\big(C_{(s_{1},s_{2})}^{(0,0)}\oplus C_{(s_{1},s_{2}+p_{2})}^{(1,1)}\big).
\end{align*}
In fact, the associated graded complex of $\mathcal{F}_{11}$ on $\mathcal{R}_{3}$
splits as a direct product of acyclic complexes $C_{(s_{1},s_{2})}^{(0,0)}\xrightarrow{\Phi_{s_{1},s_{2}}^{-L_{1}}}C_{(s_{1}+p_{1},s_{2})}^{(1,0)},$
because the inclusion map $I_{s_{1},s_{2}}^{-L_{1}},s_{1}<0$ is $id$.
Thus $\mathcal{R}_{3}$ is acyclic. Similar argument applies to $\mathcal{R}_{4}$.

At last, we look at the quotient complex
\[
\mathcal{Q}=\mathcal{C}_{1}/(\mathcal{R}_{3}+\mathcal{R}_{4})=\bigoplus_{-p_{1}<s_{1}\leq0,-p_{2}<s_{2}\leq0}C_{(s_{1},s_{2})}^{(0,0)},
\]
where $C_{(s_{1},s_{2})}^{(0,0)}=A_{s_{1},s_{2}}^{-}.$ There is only
one $A_{\mathrm{s}}^{-}$ left in each $\text{Spin}^{c}$ structure
$Y$ with homology $\mathbb{F}[[U]]$. For $(s_{1},s_{2})=(0,0)$,
the complex $C_{(0,0)}^{(0,0)}=A_{0,0}^{-}$ has $d_{1}$ as a generator
of its homology with grading $\mu_{0,0}^{0,0}(d_{1})=\frac{p_{1}+p_{2}-10}{4}.$
For $-p_{1}<s_{1}<0$, the complex $C_{(s_{1},s_{2})}^{(0,0)}=A_{s_{1},s_{2}}^{-}$
has $a_{2}$ as a generator of its homology with grading $\mu_{s_{1},s_{2}}^{0,0}(a_{2})=\frac{s_{1}^{2}}{p_{1}}+\frac{s_{2}^{2}}{p_{2}}+s_{1}+s_{2}+\frac{p_{1}+p_{2}-2}{4}.$
Similarly, we have the same formula for $-p_{2}<s_{2}<0,-p_{1}<s_{1}\leq0$.

\begin{figure}

\includegraphics[scale=0.8]{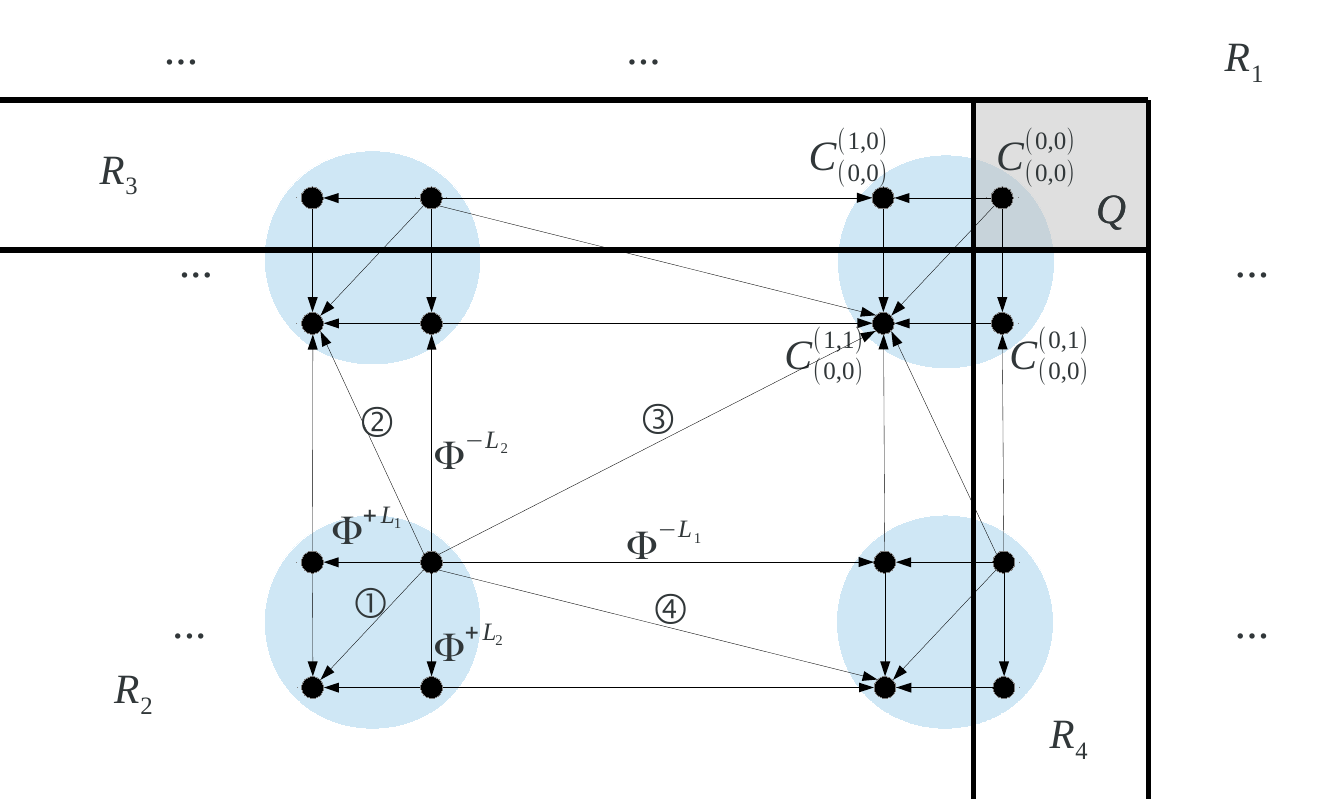}\caption{\textbf{The surgery complex for $\Lambda=(1,1)$.} The arrows with
circled numbers $1,2,3,4$ are $\Phi^{+L_{1}\cup+L_{2}},$ $\Phi^{+L_{1}\cup-L_{2}},$$\Phi^{-L_{1}\cup-L_{2}},$
and $\Phi^{-L_{1}\cup+L_{2}}$ respectively. The regions $R_{1},R_{2},R_{3},R_{4}$
divided by the (thicker) lines are corresponded to the acyclic subcomplexes
$\mathcal{R}_{1},\mathcal{R}_{2},\mathcal{R}_{3},\mathcal{R}_{4}.$
The shaded region $Q$ corresponds to the truncated complex $\mathcal{Q}$. }

\label{(1,1)surgery}

\end{figure}

(2) When $p_{1}p_{2}<0$, we might as well suppose $p_{1}>0,p_{2}<0$
due to the symmetry of the two components. We construct four filtrations
\begin{align*}
\mathcal{F}_{00}(C_{(s_{1},s_{2})}^{(\varepsilon_{1},\varepsilon_{2})})=-s_{1}+s_{2}, & \ \ \mathcal{F}_{01}(C_{(s_{1},s_{2})}^{(\varepsilon_{1},\varepsilon_{2})})=-s_{1}-s_{2}+(\varepsilon_{2}-1)p_{2},\\
\mathcal{F}_{10}(C_{(s_{1},s_{2})}^{(\varepsilon_{1},\varepsilon_{2})})=s_{1}-(\varepsilon_{1}-1)p_{1}+s_{2}, & \ \ \mathcal{F}_{11}(C_{(s_{1},s_{2})}^{(\varepsilon_{1},\varepsilon_{2})})=s_{1}-(\varepsilon_{1}-1)p_{1}-s_{2}+(\varepsilon_{2}-1)p_{2}.
\end{align*}

Without loss of generality, see Figure \ref{(1,-1)surgery} for the
illustration of the surgery complex and the truncation in the case
of $\Lambda=(1,-1).$ We first kill an acyclic subcomplex $\mathcal{R}_{1}$
composed of $C_{(s_{1},s_{2})}^{(\varepsilon_{1},\varepsilon_{2})}$
with $s_{1}>0$. Indeed, the associated graded complex of $\mathcal{F}_{00}$
on $\mathcal{R}_{1}$ splits as a direct product of acyclic squares,
since the inclusion map $I_{s_{1},s_{2}}^{+L_{1}},s_{1}>0$ is $id$.

We have another acyclic subcomplex
\[
\mathcal{R}_{2}=\bigoplus_{s_{1}-(\varepsilon_{1}-1)p_{1}\leq0}C_{(s_{1},s_{2})}^{(\varepsilon_{1},\varepsilon_{2})}.
\]
In fact, since $\Phi_{s_{1},s_{2}}^{-L_{1}}$ are quasi-isomorphisms
when $s_{1}<0$, the associated graded of the filtration $\mathcal{F}_{10}$
for $\mathcal{R}_{2}$ splits as a direct product of acyclic squares
$R_{\mathrm{s},1,0}.$ Thus $\mathcal{R}_{2}$ is acyclic.

Thus, $\mathcal{C}$ is quasi-isomorphic to the quotient complex
\[
\mathcal{C}_{1}=\mathcal{C}/(\mathcal{R}_{1}+\mathcal{R}_{2})=\bigoplus_{-p_{1}<s_{1}\leq0}\big(C_{(s_{1},s_{2})}^{(0,0)}\oplus C_{(s_{1},s_{2})}^{(0,1)}\big).
\]
We have an acyclic quotient complex $\mathcal{R}_{3}$ of $\mathcal{C}_{1}$
\[
\mathcal{R}_{3}=\bigoplus_{-p_{1}<s_{1}\leq0,s_{2}>0}\big(C_{(s_{1},s_{2})}^{(0,0)}\oplus C_{(s_{1},s_{2})}^{(0,1)}\big),
\]
since the inclusion maps $I_{s_{1},s_{2}}^{+L_{2}},s_{2}>0$ are all
the identities. Furthermore, we have another acyclic quotient complex
$\mathcal{R}_{4}$ of $\mathcal{C}_{1}$
\[
\mathcal{R}_{4}=\bigoplus_{-p_{1}<s_{1}\leq0,s_{2}<0}\big(C_{(s_{1},s_{2})}^{(0,0)}\oplus C_{(s_{1},s_{2}-p_{2})}^{(0,1)}\big).
\]
Thus $\mathcal{C}$ is quasi-isomorphic to
\[
\mathcal{Q}=\mathcal{C}_{1}\backslash(\mathcal{R}_{3}\cup\mathcal{R}_{4})=\{\bigoplus_{-p_{1}<s_{1}\leq0}\big(C_{(s_{1},0)}^{(0,0)}\oplus C_{(s_{1},0)}^{(0,1)}\oplus C_{(s_{1},p_{2})}^{(0,1)}\big)\}\oplus\{\bigoplus_{-p_{1}<s_{1}\leq0,p_{2}<s_{2}<0}C_{(s_{1},s_{2})}^{(0,1)}\}.
\]

In the $\text{Spin}^{c}$ structure $(t_{1},0)\in\mathbb{Z}/p_{1}\mathbb{Z}\oplus\mathbb{Z}/p_{2}\mathbb{Z}$,
we have the complex as follows,
\[
\xyR{1pc}\xyC{2pc}\xymatrix{C_{(s_{1},0)}^{(0,0)}=A_{s_{1},0}^{-}\ar[r]^{\Phi_{s_{1},0}^{+L_{2}}}\ar[rd]_{\Phi_{s_{1},0}^{-L_{2}}} & A_{s_{1},+\infty}^{-}=C_{(s_{1},0)}^{(0,1)}\\
 & A_{s_{1},+\infty}^{-}=C_{(s_{1},p_{2})}^{(0,1)},
}
\]
where $s_{1}$ is an integer such that $-p_{1}<s_{1}\leq0$ and $s_{1}\equiv t_{1}(\mathrm{mod}\ p_{1})$.
Since the inclusion maps $I_{0,0}^{\pm L_{2}}$ induce the same action
on homology, $\Phi_{0,0}^{\pm L_{2}}$ are chain homotopic to each
other. By Corollary \ref{cor:standard model of unknot}, we can replace
$A_{0,0}^{-},A_{0,+\infty}^{-}$ by the complex $\mathbb{F}[[U_{1},U_{2}]]\xrightarrow{U_{1}-U_{2}}\mathbb{F}[[U_{1},U_{2}]],$
where the generators are $g_{1},g_{2}$. Then, we can replace the
chain maps $I_{0,0}^{\pm L_{2}}$ by the same chain map $\tilde{I}$,
where $\tilde{I}(g_{i})=U_{1}g_{i}$. Thus, the homology of the $(0,0)$
$\text{Spin}^{c}$ structure can be computed by this perturbed complex,
which is $\mathbb{F}[[U]]\oplus\mathbb{F}[[U]]/U$. From above computation,
the generator corresponding to $\mathbb{F}[[U]]$ is actually the
generator of $H_{*}(C_{(0,0)}^{(0,1)})$, which is $d_{1}\in A_{s_{1},+\infty}^{-}$
with grading $\mu_{0,0}^{0,1}(d_{1})=\frac{p_{1}+p_{2}}{4}$ by Equation
(\ref{eq:absolute grading}).

On the other hand, since the inclusion maps $I_{s_{1},0}^{-L_{2}},s_{1}<0$
are all quasi-isomorphisms, we can kill the acyclic quotient complex
$A_{s_{1},0}^{-}\xrightarrow{\Phi^{-L_{2}}}A_{s_{1},+\infty}^{-}.$
Thus, the homology for the $\text{Spin}^{c}$ structure $(t_{1},0)\in\mathbb{Z}/p_{1}\mathbb{Z}\oplus\mathbb{Z}/p_{2}\mathbb{Z}$
with $t_{1}\neq0$ is $\mathbb{F}[[U]]$ generated by $d_{1}\in A_{s_{1},+\infty}^{-}$
of grading $\mu_{s_{1},0}^{0,1}(d_{1})=\frac{s_{1}^{2}}{p_{1}}+s_{1}+\frac{p_{1}+p_{2}}{4}$,
where $s_{1}$ is an integer with $-p_{1}<s_{1}<0$ in the class $t_{1}$
modulo $p_{1}$.

In every other $\text{Spin}^{c}$ structure in the complex $\mathcal{Q}$,
there is only one complex $C_{(s_{1},s_{2})}^{(0,1)}=A_{s_{1},+\infty}^{-},-p_{1}<s_{1}\leq0$
of homology $\mathbb{F}[[U]]$ of grading $\frac{s_{1}^{2}}{p_{1}}+\frac{s_{2}^{2}}{p_{2}}+s_{1}-s_{2}+\frac{p_{1}+p_{2}}{4}$.

\begin{figure}

\includegraphics[scale=0.8]{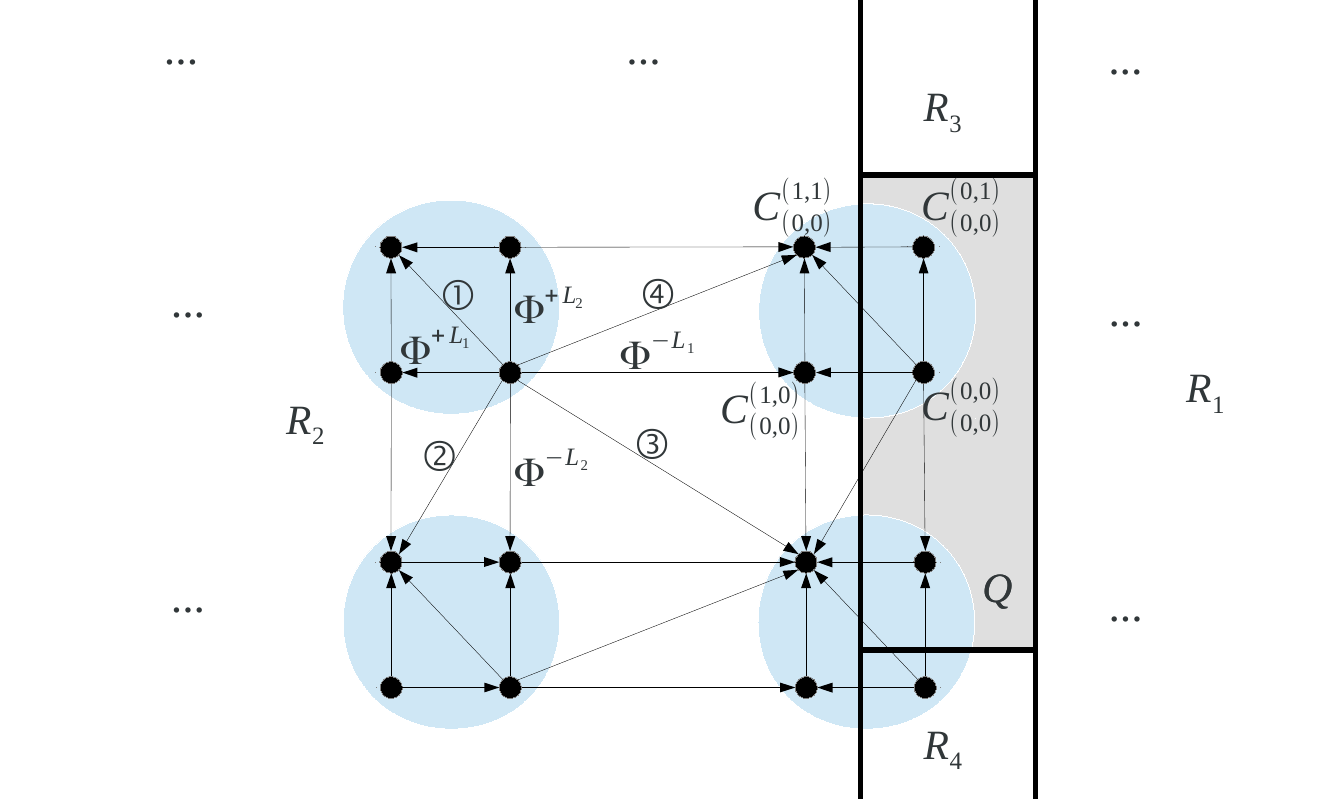}\caption{\textbf{The surgery complex for $\Lambda=(1,-1)$.} The arrows with
circled numbers $1,2,3,4$ are $\Phi^{+L_{1}\cup+L_{2}},$ $\Phi^{+L_{1}\cup-L_{2}},$$\Phi^{-L_{1}\cup-L_{2}},$
and $\Phi^{-L_{1}\cup+L_{2}}$ respectively. The regions $R_{1},R_{2},R_{3},R_{4}$
divided by the (thicker) lines are corresponded to the acyclic subcomplexes
$\mathcal{R}_{1},\mathcal{R}_{2},\mathcal{R}_{3},\mathcal{R}_{4}.$
The shaded region $Q$ corresponds to the truncated complex $\mathcal{Q}$. }

\label{(1,-1)surgery}

\end{figure}

(3) The last case is when $p_{1},p_{2}$ are both negative integers.
We use two filtrations
\[
\mathcal{F}_{00}(C_{(s_{1},s_{2})}^{(\varepsilon_{1},\varepsilon_{2})})=s_{1}+s_{2},\ \ \mathcal{F}_{11}(C_{(s_{1},s_{2})}^{(\varepsilon_{1},\varepsilon_{2})})=-s_{1}+(\varepsilon_{1}-1)p_{1}-s_{2}+(\varepsilon_{2}-1)p_{2}.
\]
Without loss of generality, see Figure \ref{(-1,-1)surgery} for the
illustration of the surgery complex and the truncation in the case
of $\Lambda=(-1,-1).$

We first kill an acyclic quotient complex
\[
\mathcal{R}_{1}=\bigoplus_{\max\{s_{1},s_{2}\}>0}C_{(s_{1},s_{2})}^{(\varepsilon_{1},\varepsilon_{2})}.
\]
By considering the filtration $\mathcal{F}_{00}$, we can see that
$\mathcal{R}_{1}$ is acyclic. We also have another acyclic quotient
complex
\[
\mathcal{R}_{2}=\bigoplus_{\min\{s_{1}-\varepsilon_{1}p_{1},s_{2}-\varepsilon_{2}p_{2}\}<0}C_{(s_{1},s_{2})}^{(\varepsilon_{1},\varepsilon_{2})}.
\]
In fact, the inclusion maps $I_{s_{1},s_{2}}^{-L_{i}},s_{i}<0$ are
quasi-isomorphisms. Thus the associated graded complex of the filtration
$\mathcal{F}_{11}$ splits as a direct product of acyclic complexes
\[
R_{\mathrm{s},1,1}\cap(\mathcal{C}\backslash\mathcal{R}_{1})
\]
where $\min\{s_{1},s_{2}\}<0$ and $R_{\mathrm{s},1,1}$ is in Equation
(\ref{eq:surgery formual4}). Therefore $\mathcal{R}_{2}$ is acyclic.

Hence, the subcomplex $\mathcal{Q}=\mathcal{C}\backslash(\mathcal{R}_{1}\cup\mathcal{R}_{2})$
is quasi-isomorphic to $\mathcal{C}$, where
\[
\mathcal{Q}=\bigoplus_{\max\{s_{1},s_{2}\}\leq0,\min\{s_{1}-\varepsilon_{1}p_{1},s_{2}-\varepsilon_{2}p_{2}\}\geq0}C_{(s_{1},s_{2})}^{(\varepsilon_{1},\varepsilon_{2})}.
\]
In the $\text{Spin}^{c}$ structure $(t_{1},t_{2}),$ $t_{1}\neq0,t_{2}\neq0$,
there is only one complex $C_{(s_{1},s_{2})}^{(1,1)}$ in $\mathcal{Q}$,
thus having the homology $\mathbb{F}[[U]]$ with grading $\frac{(2s_{1}-p_{1})^{2}}{4p_{1}}+\frac{(2s_{2}-p_{2})^{2}}{4p_{2}}+\frac{1}{2},$
where $s_{1},s_{2}$ are negative integers in the residue classes
$t_{1},t_{2}$ such that $s_{i}\geq p_{i}+1,i=1,2$. In the $\text{Spin}^{c}$
structure $(0,t_{2}),t_{2}\neq0$, there are three complexes $C_{(0,s_{2})}^{(0,1)},C_{(0,s_{2})}^{(1,1)},C_{(p_{1},s_{2})}^{(1,1)}$
in $\mathcal{Q}$, where $s_{2}$ is an integer in the residue class
$t_{2}$ such that $p_{2}<s_{2}<0$. Since the inclusion map $I_{0,s_{2}}^{\pm L_{1}},s_{2}\neq0$
are quasi-isomorphisms, we can replace $\Phi_{0,s_{2}}^{+L_{1}}$
by $\Phi_{0,s_{2}}^{-L_{1}}$ in the perturbed complex and thus split
it as a direct sum, $\mathrm{cone}(\Phi_{0,s_{2}}^{+L_{1}})\oplus C_{(0,s_{2})}^{(1,1)}$,
by changing the basis. Thus the homology is the same as the homology
of $C_{(0,s_{2})}^{(1,1)}=A_{+\infty,+\infty}^{-}$, which is $\mathbb{F}[[U]]$
generated by $[d_{1}]$ with grading $\frac{p_{1}}{4}+\frac{(2s_{2}-p_{2})^{2}}{4p_{2}}+\frac{1}{2}$.
It is similar for $(t_{1},0)\in\mathrm{Spin}^{c}(Y),t_{1}\neq0.$

\begin{figure}

\includegraphics[scale=0.8]{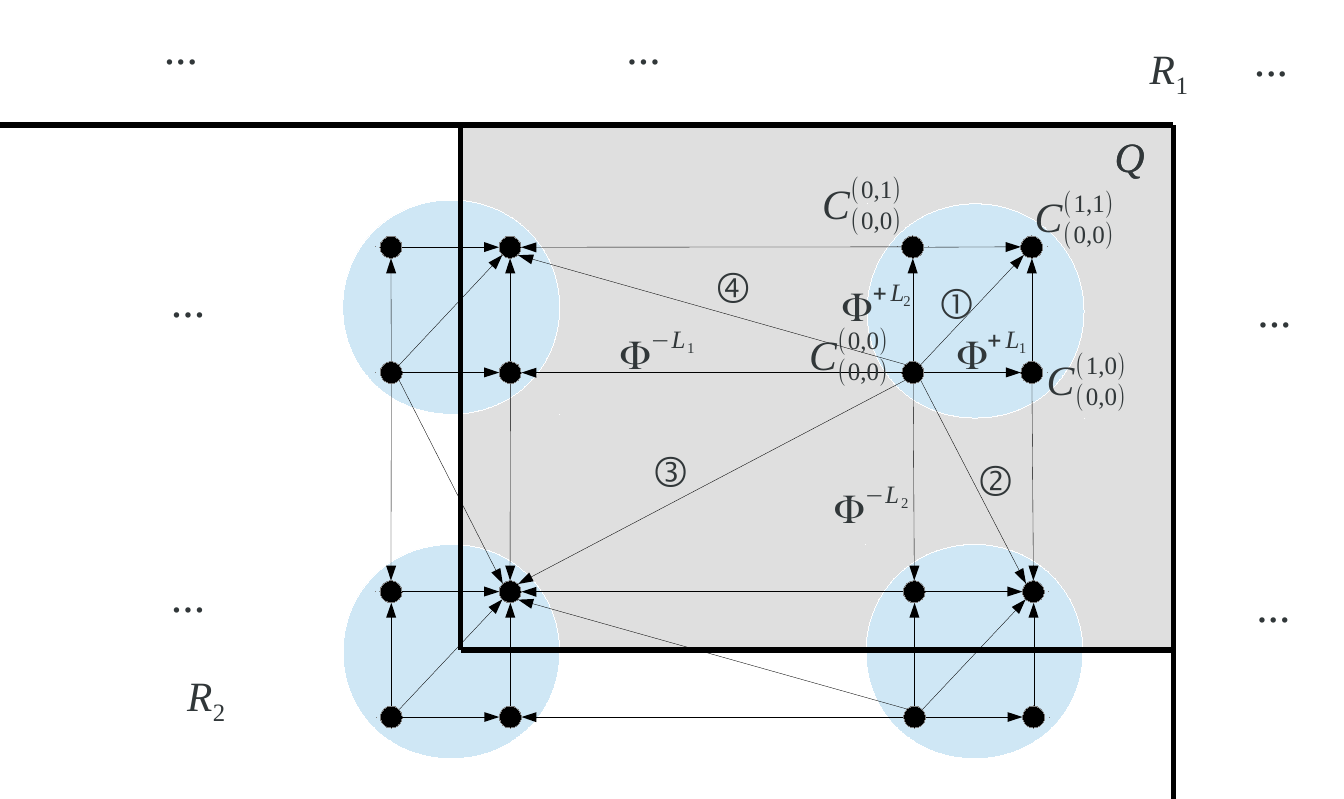}\caption{\textbf{The surgery complex for $\Lambda=(-1,-1)$.} The arrows with circled numbers $1,2,3,4$ are $\Phi^{+L_{1}\cup+L_{2}},$ $\Phi^{+L_{1}\cup-L_{2}},$$\Phi^{-L_{1}\cup-L_{2}},$
and $\Phi^{-L_{1}\cup+L_{2}}$ respectively. The regions $R_{1},R_{2}$
divided by the (thicker) lines are corresponded to the acyclic subcomplexes
$\mathcal{R}_{1},\mathcal{R}_{2}.$ The shaded region $Q$ corresponds
to the truncated complex $\mathcal{Q}$.}

\label{(-1,-1)surgery}

\end{figure}

The most interesting $\text{Spin}^{c}$ structure is $(0,0)$. It
consists of nine complexes, which are also illustrated in Figure \ref{(-1,-1)surgery}.
By Corollary \ref{cor:standard model of unknot} and the discussion in Section 5.6,
in the perturbed complex we can replace all the $A_{\mathrm{s}}^{-}$
by the complex $\bb{F}[[U]]$ with $0$ differential
and replace the edge maps by the corresponding maps on homology. Finally, the perturbed complex is the following
chain complex
\[\xymatrix{\bb{F}[[U]] &\bb{F}[[U]]\ar[l]_{1} \ar[r]^{1} &\bb{F}[[U]]\\
           \bb{F}[[U]]\ar[d]_{1}\ar[u]_1 &\bb{F}[[U]]\ar[l]_{U} \ar[r]^{U}\ar[u]^{U}\ar[d]_{U} &\bb{F}[[U]]\ar[u]^{1}\ar[d]_{1}\\
           \bb{F}[[U]]  &\bb{F}[[U]]\ar[l]_{1}\ar[r]^{1} &\bb{F}[[U]].}
\]
Direct computation shows that
\[
{\bf HF}^{-}(S_{\Lambda}^{3}(\mathit{Wh}),(0,0))=\mathbb{F}[[U]]\oplus(\mathbb{F}[[U]]/U),
\]
when $\Lambda=\text{diag}(p_{1},p_{2})$ with $p_{1},p_{2}<0$. Thereby, $[d_{1}]=1\in H_{*}(C_{(p_{1},p_{2})}^{(1,1)})$
is a generator of the $\bb{F}[[U]]$ summand with the absolute
grading $\frac{p_{1}+p_{2}+2}{4}$.
\end{proof}

\begin{thm}
\label{thm:surgery_on_b(8k,4k+1)}
Let $\overrightarrow{L}$ be the two-bridge link $b(8k,4k+1),k\in\mathbb{N}$
and $\Lambda=\text{diag}(p_{1},p_{2}),p_{1},p_{2}\in\mathbb{Z}$ be
the framing matrix of an integer surgery on $\overrightarrow{L}$.
As in Proposition \ref{prop:surgery on Wh}, we use $(t_{1},t_{2})\in\mathbb{Z}/p_{1}\mathbb{Z}\oplus\mathbb{Z}/p_{2}\mathbb{Z}$
to denote the \emph{$\text{Spin}^{c}$} structures over $S_{\Lambda}^{3}(\overrightarrow{L})$.
Then, we have the Floer homology
\begin{equation}
{\bf HF}^{-}(S_{\Lambda}^{3}(\overrightarrow{L}),(t_{1},t_{2}))=\begin{cases}
{\bf HF}^{-}(S_{\Lambda}^{3}(\mathit{Wh}),(0,0))\oplus\mathbb{F}^{\oplus(k-1)}, & (t_{1},t_{2})=(0,0),\\
{\bf HF}^{-}(S_{\Lambda}^{3}(\mathit{Wh}),(t_{1},t_{2})), & \text{otherwise.}
\end{cases}\label{eq:HF(b(8k,4k+1))}
\end{equation}
The correction terms of the elements in the ${\bf HF}^{-}(S_{\Lambda}^{3}(\mathit{Wh}))$-summand
are the same as in ${\bf HF}^{-}(S_{\Lambda}^{3}(\mathit{Wh})).$ \end{thm}

\begin{proof}
By Proposition \ref{prop:CFL^-(b(8k,4k+1))}, $CFL^{-}(\overrightarrow{L})=CFL^{-}(\mathit{Wh})\oplus\bigoplus_{i=1}^{k-1}(N,\partial^{-})$.
Let $\mathfrak{\mathscr{\mathscr{N}}}=\bigoplus_{i=1}^{k-1}(N,\partial^{-})$.
We define $\mathscr{N}_{\text{s}}$ similarly as $A_{\mathrm{s}}^{-}$
in (\ref{eq:A_s}). Concretely, suppose $G$ be a set of homogeneous
generators of $\mathscr{N}$ as a $\mathbb{F}[[U_{1},U_{2}]]$-module,
and for $x\in G$,
\[
\partial x=\sum_{y\in G}k_{xy}y,
\]
where $k_{xy}\in\mathbb{F}[[U_{1},U_{2}]]$. Let $A(x)=(A_{1}(x),A_{2}(x))$
denote the Alexander filtration of $x\in G$. Define $\mathscr{N}_{\text{s}}$
by
\[
\partial x=\sum_{y\in G}k_{xy}\cdot U_{1}^{\max\{A_{1}(x)-s_{1},0\}-\max\{A_{1}(y)-s_{1},0\}}U_{2}^{\max\{A_{2}(x)-s_{2},0\}-\max\{A_{2}(y)-s_{2},0\}}\cdot y.
\]

Thus $A_{\mathrm{s}}^{-}(\overrightarrow{L})=A_{\mathrm{s}}^{-}(\mathit{Wh})\oplus\mathscr{N}_{\text{s}}$.
Thus all the inclusion maps $I^{\pm L_{i}},i=1,2$ preserve this direct
sum decomposition. Since the complexes $\mathscr{N}_{s_{1},\pm\infty}$
are acyclic complexes, we can choose $\tilde{D}_{s_{1},-\infty}^{\pm L_{2}}:A_{s_{1},\pm\infty}^{-}(\overrightarrow{L})\rightarrow A_{s_{1},+\infty}^{-}(\overrightarrow{L})$
to be
\[
\tilde{D}_{s_{1},-\infty}^{\pm L_{2}}=D(\mathit{Wh})_{s_{1},-\infty}^{\pm L_{2}}\oplus0,
\]
where $D(\mathit{Wh})_{s_{1},-\infty}^{\pm L_{2}}$ is the destabilization
map for $\mathit{Wh}$. Therefore $\tilde{\Phi}_{\text{s}}^{\pm L_{i}}=\Phi(\mathit{Wh})_{\text{s}}^{\pm L_{i}}\oplus\Phi_{\mathscr{N},\text{s}}^{\pm L_{i}},$
where $\Phi_{\mathscr{N},\text{s}}^{\pm L_{i}}=0:\mathscr{N}_{\text{s}}\rightarrow\mathscr{N}_{\psi^{\pm L_{i}}(\text{s})}.$

Thus the perturbed surgery complex $(\tilde{\mathcal{C}}^{-}(\overrightarrow{L},\Lambda),\mathcal{\tilde{D}}^{-})$
is a direct sum of two twisted gluing of squares
\[
(\tilde{\mathcal{C}}^{-}(\overrightarrow{L},\Lambda),\mathcal{\tilde{D}}^{-})=(\mathcal{C}^{-}(\mathit{Wh},\Lambda),\mathcal{D}^{-})\oplus\prod_{\text{s}=(s_{1},s_{2})\in\mathbb{Z}^{2}}(\mathscr{N}_{\text{s}}\oplus\mathscr{N}_{s_{1},+\infty}\oplus\mathscr{N}_{+\infty,s_{2}}\oplus\mathscr{N}_{+\infty,+\infty},\tilde{\mathcal{D}}^{-}).
\]
From the fact that any $\mathscr{N_{\text{s}}}$ with $\text{s}\neq(0,0)$
is acyclic, it follows that $H_{*}(\tilde{\mathcal{C}}^{-}(\overrightarrow{L}),\mathcal{\tilde{D}}^{-})=H_{*}(\mathcal{C}^{-}(\mathit{Wh}),\mathcal{D}^{-})\oplus H_{*}(\mathscr{N}_{0,0})$.
For that $\mathscr{N}_{0,0}$ belongs to the $(0,0)$ $\text{Spin}^{c}$
structure and $H_{*}(\mathscr{N}_{0,0})=\mathbb{F}[[U]]/U,$ we have
the equations (\ref{eq:HF(b(8k,4k+1))}). The absolute gradings are
inherited from $H_{*}(\mathcal{C}^{-}(\mathit{Wh}),\mathcal{D}^{-}).$
\end{proof}

\bibliographystyle{plain}
\bibliography{Math}

\end{document}